\documentclass[english]{smfart}

\usepackage[english]{babel}
 \usepackage{amsfonts}
\usepackage{epsfig,graphics}
\usepackage{amssymb}
\usepackage{amscd}
\usepackage{enumerate}
\usepackage{smfthm}
\usepackage{calrsfs}

\author{Charles Frances}
\address{IRMA, 7 rue Ren\'e Descartes, 67000 Strasbourg.}
\email{cfrances@math.unistra.fr}
\urladdr{}
\title[Isometry group of Lorentz manifolds]{Isometry group of Lorentz manifolds: A coarse perspective}
\newtheorem{theoreme}{Theorem}[section]
\newtheorem*{theono}{Theorem}
\newtheorem{fact}[theoreme]{Fact}

\newtheorem{proposition}[theoreme]{Proposition}

\newtheorem{lemme}[theoreme]{Lemma}

\newtheorem{remarque}[theoreme]{Remark}

\newtheorem{corollaire}[theoreme]{Corollary}

\newtheorem{thmx}{Theorem}
\newtheorem{corox}[thmx]{Corollary}

\newcommand{\Ker}{\operatorname{Ker }}

\newcommand{\Hom}{\operatorname{Hom }}
\newcommand{\iso}{\operatorname{Iso }}
\newcommand{\xk}{{\bf X}_{\kappa}}

\newcommand{\lig}{{\Lambda}}
\newcommand{\eg}{{E_{\Lambda}}}
\newcommand{\dig}{{{\sc{d}}_{\Lambda}}}
\newcommand{\carg}{{{\sc{card}}_{\Lambda}}}

\newcommand{\ad}{\operatorname{ad}}

\newcommand{\SL}{\operatorname{SL}}
\newcommand{\SU}{\operatorname{SU}}

\newcommand{\GL}{\operatorname{GL}}
\newcommand{\PSL}{\operatorname{PSL}}

\newcommand{\SO}{\operatorname{SO}}
\newcommand{\OO}{\operatorname{O}}

\newcommand{\PO}{\operatorname{PO}}

\newcommand{\AdS}{\operatorname{\mathbb{ADS}}}
\newcommand{\dS}{\operatorname{\mathbb{DS}}}

\newcommand{\Homot}{\operatorname{Homot}}
\newcommand{\Iso}{\operatorname{Iso}}
\newcommand{\diff}{\operatorname{Diff}}
\newcommand{\homot}{\operatorname{Homot}}
\newcommand{\Isloc}{\operatorname{\isloc}}

\newcommand{\Ad}{\operatorname{Ad}}

\newcommand{\cald}{\mathcal{D}}
\newcommand{\calc}{\mathcal{C}}
\newcommand{\calf}{\mathcal{F}}
\newcommand{\ocalf}{\overline{\mathcal{F}}}

\newcommand{\calm}{\mathcal{M}}
\newcommand{\calo}{\mathcal{O}}
\newcommand{\calw}{\mathcal{W}}

\newcommand{\tm}{\tilde{{M}}}
\newcommand{\tw}{\tilde{{w}}}

\newcommand{\dd}{{\mathcal D}}

\newcommand{\hcalm}{\hat{\mathcal{M}}}

\def\kg{{\operatorname{{\kappa}^{\text{g}}}}}
\def\okg{{\operatorname{\overline{\kappa}^{\text{g}}}}}
\def\hkg{{\operatorname{{\hat{\kappa}}^{\text{g}}}}}

\def\NN{\mathbb{N}}
\def\TT{\mathbb{T}}
\def\HH{\mathbb{H}}

\def\RR{\mathbb{R}}

\def\ZZ{\mathbb{Z}}
\def\CC{\mathbb{C}}

\def\O{\mathcal{O}}
\def\OO{\operatorname{O}}

\newcommand{\mint}{{M^{\rm int}}}
\newcommand{\hmint}{{{\hat M}^{\rm int}}}
\newcommand{\mreg}{{M^{reg}}}
\newcommand{\hmreg}{{{\hat{M}}^{reg}}}

\newcommand{\RP}{{\operatorname{{\mathbb R}{P}}}}

\newcommand{\hm}{{\hat{M}}}
\newcommand{\hx}{{\hat{x}}}
\newcommand{\hy}{{\hat{y}}}

\newcommand{\kil}{\operatorname{{\mathfrak{kill}}}}
\newcommand{\kiloc}{\operatorname{{\mathfrak{kill}^{loc}}}}
\newcommand{\isloc}{\operatorname{{{Is}^{loc}}}}
\newcommand{\heis}{{\operatorname{\mathfrak{heis}}}}

\newcommand{\Asdim}{{\operatorname{{Asdim}}}}
\newcommand{\onabla}{\overline{\nabla}}

\newcommand{\liek}{{\mathfrak{k}}}

\newcommand{\lieu}{{\mathfrak{u}}}
\newcommand{\lieg}{{\mathfrak{g}}}

\newcommand{\liea}{{\mathfrak{a}}}
\newcommand{\liei}{{\mathfrak{i}}}

\newcommand{\lies}{{\mathfrak{s}}}
\newcommand{\liez}{{\mathfrak{z}}}

\newcommand{\liem}{{\mathfrak{m}}}

\newcommand{\oo}{{\mathfrak{o}}}

\newcommand{\Oun}{{\operatorname{O}(1,n)}}
\newcommand{\oun}{{\operatorname{\mathfrak{o}}(1,n)}}

\begin{document}
\frontmatter

\begin{abstract}
We prove a structure theorem for the isometry group $\Iso(M,g)$ of a compact Lorentz manifold, under the assumption 
that a closed subgroup has exponential growth. We don't assume anything about the identity component of $\Iso(M,g)$, so that 
 our results apply for discrete isometry groups. We infer a full classification of lattices that can act isometrically 
 on compact Lorentz manifolds.
 Moreover, without any growth hypothesis, we prove a Tits alternative for discrete subgroups of $\Iso(M,g)$.
\end{abstract}
%

\maketitle

\section{Introduction}

In this article, we are interested in the question of which groups can appear as the group of
isometries of a compact pseudo-Riemannian manifold $(M, g)$. Although this question makes sense, and has been considered, for 
 more general classes of
geometric structures,  it is striking to note that a complete answer  is  only known in a very small number of cases.
  One of the most natural examples, for which we have a complete picture,  is that of Riemannian structures 
   (all the structures considered in this article are assumed to be smooth, {\it i.e} of class $C^{\infty}$).
   A famous theorem of Myers and Steenrod \cite{myers-steenrod}, ensures that the group of isometries of a 
   compact Riemannian manifold is a compact Lie group, the Lie topology coinciding moreover with the $C^0$ topology. 
    Conversely, for any compact Lie group $G$, one can construct a  compact Riemannian manifold $(M, g)$ for which
     $\Iso (M, g) = G$. This was proved independently in  \cite{bedford}, \cite{saerens}. This settles completely the Riemannian 
     case.
    
    Regarding compact Lorentzian manifolds, which are the subject of this article, things get substantially more complicated.
  Indeed, although it is always true that the group of isometries of a Lorentzian manifold is a
Lie transformation group, a new phenomenon appears: even for compact manifolds $ (M, g) $, the group $ \Iso (M, g) $
  may well not be compact. For  instance, it is quite possible that this group is infinite discrete. 

Seminal works are owed to R. Zimmer, which paved the way  for the study of the isometry group of Lorentz manifolds.
 A first important contribution was  \cite{zimmer.iso}, where it is shown that any  connected, simple Lie group $G$, acting
  isometrically on a compact Lorentz manifold, must be 
   locally isomorphic to $\operatorname{SL}(2,\RR)$. Several deep contributions followed, \cite{gromov}, \cite{kowalsky}, before S. Adams, G. Stuck, 
   and independently A. Zeghib obtained, about ten years later, a complete classification of all possible Lie algebras for the group
   $ \Iso (M, g) $. Their  theorem can be stated as follows:
    
   
  \begin{theono}\cite{adams.stuck}, \cite{adams.stuck2}, \cite{zeghibhomogene}, \cite{zeghibidentity}.
  \label{thm.adams-zeghib}
  {Let $(M,g)$ be a compact Lorentz manifold. Then the Lie algebra ${\mathfrak{Iso}}(M,g)$ is of the form 
 $\lies \oplus \liea \oplus \liek$ where:}
 \begin{enumerate}
  {\item $\liek$ is trivial, or the Lie algebra of a compact semisimple group.}
  {\item $\liea$ is an abelian algebra (maybe trivial).}
  {\item $\lies$ is trivial, $\mathfrak{sl}(2,\RR)$, a Heisenberg algebra $\heis(2d+1)$, or an  oscillator algebra.}
 \end{enumerate}

\end{theono} 
  
 Altough this classification is only at the Lie algebra level, we have a fairly good picture of the possibilities for the 
 identity component $\Iso^o(M,g)$  
 (see \cite{zeghibidentity} for a discussion about this identity component). 
 
 Few attemps were done to go beyhond the understanding of  $\Iso^o(M,g)$, and complete the picture for the full
 group of isometries of a compact Lorentz manifold $(M,g)$. 
 An intermediate situation was studied in \cite{zeghibpic}, where  the authors assume
  that $\Iso^o(M,g)$ is compact (but still nontrivial), and the 
  discrete part  $\Iso(M,g)/\Iso^o(M,g)$ is infinite. 
   
   In the purely discrete case, the most advanced results were obtained by R. Zimmer, who proved in 
    \cite[Theorem D]{zimmer.iso} that no discrete, infinite, subgroup of $\Iso(M,g)$
 can have Kazhdan's property $(T)$. For instance, this rules out the possibility that $\Iso(M,g)$ is isomorphic
  to a lattice in some higher rank simple Lie group, or in the group $\operatorname{Sp}(1,n)$, $n \geq 2$.
  Besides this  result,  and to the best of our knowledge, almost nothing was known when the  group $\Iso(M,g)$ is
 infinite discrete. Our aim here, is to begin filling this gap.   Theorems \ref{thm.main} and \ref{thm.tits} below are a 
 step in this direction.

  
  \subsection{Lorentzian isometry groups with exponential growth}
  
  We begin our general study of the isometry group of a compact Lorentz manifold, by
   making  a growth assumption, namely there exists  a closed, compactly
  generated subgroup of $\Iso(M,g)$ having exponential growth.
  In this case, we get a  pretty clear picture of the situation: The group $ \Iso (M, g) $ is either a compact extension of $ \PSL (2, \RR) $, or 
  a compact extension
   of a {\it Kleinian group} (a Kleinian group is any discrete subgroup of $\PO(1,d)$, $d \geq 2$). In particular, 
    whenever $\Iso(M,g)$ is discrete, and contains a finitely generated subgroup of 
  exponential growth,  then $\Iso(M,g)$ must be virtually isomorphic to a Kleinian group, what reduces substantially the possibilities.
 
 \begin{thmx}
 \label{thm.main}
  Let $(M,g)$ be a smooth, compact,  $(n+1)$-dimensional Lorentz manifold, with $n \geq 2$.  Assume that the isometry group 
   $\Iso(M,g)$ contains a 
  closed, compactly generated 
  subgroup with exponential growth.  Then  we are in exactly one of the following cases:
  \begin{enumerate}
   \item{ The group  $\Iso(M,g)$ is virtually a Lie group extension of 
   $\PSL(2,\RR)$ by a compact Lie group. }
   \item{There exists  a discrete subgroup $\Lambda \subset \PO(1,d)$, for $2 \leq d \leq n$, 
    such that 
     $\Iso(M,g)$ is virtually a Lie group  extension of $\Lambda$ by a compact Lie group.} 
  \end{enumerate}

 \end{thmx}

 By a {\it Lie group extension}, we mean a  group extension $1 \to K \to G \to H \to 1$ in the usual sense, where each group $K,G,H$ is a 
 Lie group, and each arrow a Lie group homomorphism.
 
 Let us make a few comments about the theorem. Although the hypotheses relate to a subgroup of $\Iso(M,g)$, the conclusion
 holds for the full 
  isometry group. We do not make directly a growth assumption about $\Iso(M,g)$, because even if the
  Lorentz manifold $(M,g)$ is assumed to be compact, we don't know if $\Iso(M,g)$ itself is compactly generated (making the 
  notion of growth ill defined). 
  The closedness assumption about the subgroup of exponential growth is mandatory, to avoid for instance the ``trivial'' situation,  
    where a non abelian free group is embedded in a compact subgroup of $\Iso(M,g)$.

     It is easy to prove that for a compact Lorentz surface, the isometry group is either compact, or a compact extension of $\ZZ$, so
    that the assumptions of the theorem are never satisfied in dimension $2$, hence our hypothesis $n+1 \geq 3$ in the statement.
     Actually, when  $M$ is $3$-dimensional, Theorem \ref{thm.main} and its proof show that $(M,g)$ has to be of constant sectional 
     curvature $-1$ or $0$. In the first case, the isometry group is a finite extension of $\PSL(2,\RR)$, and in the second one, 
      $(M,g)$ is a flat $3$-torus. The group $\Lambda$ of Theorem \ref{thm.main} is then an arithmetic lattice in $\PO(1,2)$. All of this also 
      follows from the complete classification of compact $3$-dimensional Lorentz manifold having a noncompact 
      isometry group (see \cite{frances}).  In higher dimension, the proof of Theorem \ref{thm.main} suggests that the geometry of
       $(M,g)$ can be described completely. We will adress this issue in later work.
    
If we are in the first case  of Theorem \ref{thm.main}, then $\Iso(M,g)$ has finitely many connected components. The identity component 
$\Iso^o(M,g)$  is finitely covered by a product $K \times \PSL(2, \RR)^{(m)}$, where $K$ is a connected compact Lie group, and 
$\PSL(2, \RR)^{(m)}$ stands for the $m$-fold cover of $\PSL(2,\RR)$.
 Such a situation appears through the following well known construction. On the Lie group $\PSL(2,\RR)$, let us consider the Killing metric,
 namely the metric obtained by pushing the Killing form on ${\mathfrak sl}(2,\RR)$ by left translations.  This is a Lorentzian metric 
 of constant sectional curvature $-1$, which is actually  bi-invariant.
  If one considers a uniform lattice $\Gamma \subset \PSL(2,\RR)$, the quotient $\PSL(2,\RR)/\Gamma$ is a $3$-dimensional Lorentz manifold,
   and the left action of $\PSL(2,\RR)$ is isometric. Actually the isometry group of $\PSL(2,\RR)/\Gamma$ is virtually $\PSL(2,\RR)$.
    Taking products with compact Riemannian manifolds, provides examples of compact Lorentzian manifolds, having isometry group virtually isomorphic to 
    $K \times \PSL(2,\RR)$, for $K$ any compact Lie group.

Let us now investigate the second case of Theorem \ref{thm.main}. The situation is then reversed: The identity component $\Iso^o(M,g)$ and the
  discrete part $\Iso(M,g)/\Iso^o(M,g)$ is infinite, isomorphic to a (finite extension of a) Kleinian group. 
  Here is a classical construction illustrating this case. On $\RR^{n+1}$, $n \geq 2$, let us consider a Lorentzian quadratic form $q$.
    It defines a Lorentzian metric $g_0$ on $\RR^{n+1}$ which is flat and translation-invariant. If $\Gamma \simeq \ZZ^{n+1}$
     is the discrete subgroup of translations with integer coordinates, we get a flat Lorentzian metric $\overline{g}_0$
      on the torus $\TT^{n+1}=\RR^{n+1}/ \Gamma$.  The isometry group of $(\TT^{n+1}, {\overline g}_0)$ is easily seen to coincide with 
      $\OO(q,\ZZ) \ltimes \TT^{n+1}$.  If  $q$ is chosed to be a rational form (namely $q$ has rational coefficients), then due to a theorem of 
      A. Borel and Harish-Chandra, $\OO(q,\ZZ)$
       is   a lattice in $\OO(q) \simeq \OO(1,n)$ (in particular, it is finitely generated of exponential growth).  
        In section 
       \ref{sec.examples.lattices}, we will elaborate on this construction, and provide examples of compact Lorentzian manifolds 
       $(M,g)$ for which the group $\Iso(M,g)$ is {\it discrete} and isomorphic to $\OO(q,\ZZ)$ as above.  The question of which Kleinian 
        groups $\Lambda$ may appear in point $(2)$ of Theorem \ref{thm.main} is interesting would 
        certainly deserve further investigations.

Observe finally, that the statements of Theorem \ref{thm.main} also contains   known, but nontrivial facts, about isometric 
actions of connected Lie groups. For instance it is shown in  \cite{adams.stuck2}, \cite{zeghibidentity} that
 if the affine group of the line 
 $\operatorname{Aff}(\RR)$ acts faithfully and isometrically on a compact Lorentz manifold,  then the action extends 
 to an isometric action of a group $G$ locally isomorphic to $\PSL(2,\RR)$.  It is plain that if $\operatorname{Aff}(\RR)$
  acts isometrically, then we must be in the first case of  Theorem \ref{thm.main}, so that the theorem provides directly
   a subgroup locally isomorphic to $\PSL(2,\RR)$ in $\Iso(M,g)$. Actually, a generalization of \cite{adams.stuck2}, 
   \cite{zeghibidentity} will be needed during the proof of  Theorem \ref{thm.main} (see  
 Theorem \ref{thm.sl2}).

%
 \subsection{A Tits alternative for the isometry group}

 We now investigate what can be said if we remove the growth assumption made in Theorem \ref{thm.main}. It is then still possible 
  to describe the discrete, finitely generated subgroups of $\Iso(M,g)$. This is the content of our second main result.

\begin{thmx}[Tits alternative for discrete subgroups]
 \label{thm.tits}
 Let $(M^{n+1},g)$ be a smooth compact  $(n+1)$-dimensional Lorentz manifold, with $n \geq 2$. 
 
Then every discrete,   finitely generated, subgroup of $\Iso(M,g)$ either contains a free subgroup in two generators, in which case
 it is virtually  isomorphic to a discrete subgroup of $\PO(1,d)$, or is virtually nilpotent of growth degree $\leq n-1$.
\end{thmx}

In the previous statement, we say that a group $G_1$ is virtually isomorphic to a group $G_2$, if there 
exists a finite index subgroup
 $G_1' \subset G_1$, and a finite normal subgroup $G_1'' \subset G_1'$, such that $G_1'/G_1''$ is isomorphic 
 to $G_2$.
 
 Theorem \ref{thm.tits} is obtained by combining  Theorem \ref{thm.main}, together with 
  the recent results about coarse embeddings of amenable groups, obtained by R. Tessera in \cite{tessera} (see Section \ref{sec.deduce}).  

  Actually, a generalization of \cite[Theorem 3]{tessera1}, from the setting of finitely generated, to that of 
 compactly generated, unimodular amenable groups, would allow to prove a Tits alternative -- in its classical formulation -- for 
  all finitely generated subgroups of $\Iso(M,g)$ 
(without the discreteness assumption).  Also, it would  yield a statement similar to that of Theorem \ref{thm.tits}
  for finitely generated subgroups of $\Iso(M,g)/\Iso^o(M,g)$. 
 We defer those developpements to  subsequent works, when such  generalizations of \cite[Theorem 3]{tessera1} 
 will be available.

 \subsection{Lorentz isometric actions of lattices}
Another interesting  corollary  of Theorem \ref{thm.main}, is
a complete desciption 
 of lattices (in noncompact simple Lie groups), which can appear as discrete subgroups of $\Iso(M,g)$, where $(M,g)$ is a 
 compact Lorentz manifold.
 We will indeed obtain:

\begin{corox}
 \label{thm.lattices}
 Let $(M^{n+1},g)$ be a compact $(n+1)$-dimensional Lorentz manifold. Assume that a discrete
  subgroup $\Lambda \subset \Iso(M,g)$ is isomorphic to a lattice in a noncompact simple Lie group $G$.
  \begin{enumerate}
   \item Then $G$ is locally isomorphic to $\PO(1,k)$, for $2 \leq k \leq n$.
   \item If equality $k=n$ holds, then $(M^{n+1},g)$ is either a $3$-dimensional anti-de Sitter manifold, or 
   a flat Lorentzian torus, or a two-fold cover of a flat  Lorentzian  torus.
  \end{enumerate}

\end{corox}

We recover this way, but by other methods,  Zimmer's results (\cite[Theorem D]{zimmer.iso}) saying that 
lattices having property $(T)$ do not appear as discrete subgroups of $\Iso(M,g)$.
 The novelty in  Corollary \ref{thm.lattices}  is to cover the case of lattices in $\OO(1,k)$ or 
 $\operatorname{SU}(1,k)$, which was, to the best of our knowledge, 
 not known.

 \subsection{Coarse embeddings, ingredients of the proofs, and organisation of the paper}
%
%
%
%
%
 
Let us explain roughly what is the general strategy to prove our main Theorem \ref{thm.main}, and its corollaries.
 
The starting point of our study, and an essential tool throughout this article, is the notion of {\it coarse embedding}, introduced by
 M. Gromov in \cite{gromov}. Thus, Section \ref{sec.coarse} is devoted to recalling what this important notion is.  
 Gromov's fundamental observation was that the isometry group of a compact $(n+1)$-dimensional Lorentzian manifold admits a coarse 
 embedding into the real hyperbolic space $\HH^n$. This mere remark already provides important restrictions on the group 
  $\Iso(M,g)$, through basic coarse invariants such as the asymptotic dimension, or some growth properties. To give the reader a flavour
   of how those basic tools
   can be used in the context of isometric actions, we prove a particular case of Corollary \ref{thm.lattices}
   in Section \ref{sec.reseau1}. Then, we explain how to link
    our main Theorem \ref{thm.main} to the results of \cite{tessera}, in order 
    to obtain Theorem \ref{thm.tits}.
 
 Section \ref{sec.limit.set} is devoted to the notion of {\it limit set} associated to a coarse embedding into $\HH^n$,  and the related 
  limit set for the isometry group $\Iso(M,g)$. It is
   explained in Sections \ref{sec.expo.derivatives} and \ref{sec.exp.growth}, how our assumtion of a compactly generated subgroup of 
   exponential growth forces this limit set to be infinite. The crucial step here, is to prove that the growth assumption we made on a subgroup
    of $\Iso(M,g)$, implies exponential growth of derivatives for the action of $\Iso(M,g)$.
   
   This property is exploited further in  Section \ref{sec.exponential.killing}. Using Gromov's theory of rigid geometric structures, 
    one establishes a link between the limit set of the isometry group on the one hand, and local Killing fields on 
   the manifold $(M,g)$ on the other hand.  The general picture is  that when $\Iso(M,g)$ has a big (infinite) limit set,
    then the manifold $(M,g)$ must have a lot of local Killing fields. This statement is made precise in Theorem 
    \ref{thm.exponential.killing}.
    
 The abundance of local Killing fields proved in Section \ref{sec.exponential.killing}, allows us to exhibit, in 
 Section \ref{sec.compactorbits}, a compact Lorentz submanifold $\Sigma \subset M$, which is locally homogeneous 
 (with semisimple isotropy)  and left invariant by a finite index subgroup $\Iso'(M,g) \subset \Iso(M,g)$. This is the content 
 of Theorem \ref{thm.reduction}. The restriction morphism $\rho: \Iso'(M,g) \to \Iso(\Sigma,g)$ is proper 
 (it has closed image and compact kernel), what essentially reduces the proof of  Theorem \ref{thm.main} to 
  the setting of locally homogeneous manifolds.
  
 In Section \ref{sec.completeness}, we  thus tackle the problem of describing all compact Lorentz manifolds, which are 
  locally homogeneous 
 with semisimple isotropy, and have an isometry group of exponential growth. To this aim, we have to prove a completeness
 theorem for this class of  manifolds, akin to 
  the one obtained by Y. Carri\`ere and B. Klingler for compact Lorentz manifolds of constant sectional curvature. 
   This is done  in  Theorem \ref{thm.complete}.
   
 This completeness result opens the way for the final proof of Theorem \ref{thm.main}, which is achieved in the last Section 
 \ref{sec.proof-main}.

\section{Coarse embedding associated to an isometric action}
\label{sec.coarse}
\subsection{Coarse embeddings between metric spaces and groups}
\label{sec.coarse.metric}
Let $(X,d_X)$ and $(Y,d_Y)$ be two metric spaces.  A map $\alpha: X \to Y$ is called {\it a coarse embedding}
 if there exist two  nondecreasing functions $\phi^-: \RR_+ \to \RR$ and $\phi^+: \RR_+ \to \RR_+$ 
 satisfying $\lim_{ r \to + \infty}\phi^{\pm}(r)= + \infty$, such that for any $x,y$
  in $X$, one has:
  \begin{equation}
   \phi^-(d_X(x,y)) \leq d_Y(\alpha(x),\alpha(y)) \leq \phi^+(d_X(x,y))
  \end{equation}

Equivalently, for any pair of sequences $(x_k),(y_k)$ in $X$, one has 
$d_Y(\alpha(x_k),\alpha(y_k)) \to + \infty$  if and only if $d_X(x_k,y_k) \to + \infty$.  

The function $\phi^+$ (resp. $\phi^-$) is called the upper control (resp. the lower control).  In the case where 
$\phi^+$ and $\phi^-$ are affine functions, we recover the more popular notion of quasi-isometric embedding.
 The spaces $(X,d_X)$ and $(Y,d_Y)$ are said to be {\it coarsely equivalent} whenever there exists a second coarse embedding 
 $\beta: (Y,d_Y) \to (X,d_X)$ such that $\beta \circ \alpha$ is a bounded distance from identity.

It seems that this notion was considered for the first time  by M. Gromov in \cite[Section 4]{gromov}, under the name of 
{\it placement}. 
\subsubsection{Coarse embeddings between groups, adpated metrics}
\label{sec.coarse.groups}
We will often speak in this paper of coarse embeddings between groups, or from a group to a metric space. To
make this notion precise, say that 
 a distance $d$ on a topological group $G$ is called {\it an adapted metric}, when it is {\it right invariant}, {\it proper}
 (the  balls are relatively compact sets), and when {\it compact subsets have finite diameter} for $d$. Observe that we don't require $d$ to be continuous, so
 that the last condition is nontrivial. It is a classical result (see \cite{struble}) that any locally compact, second countable, 
 topological group admits an adapted metric (actually one can even find adapted metrics generating the topology of $G$).
 
 An important example of adapted metric will be the word metric on a second countable, locally compact, compactly generated group. Those are groups
  $G$, for which there exists a compact subset $S$, that we may assume symmetric, namely $S=S^{-1}$, such that 
  $G=\bigcup_{n \in \NN} S^n$. Here $S^n=S.S.\ldots.S$ denotes the set of $g \in G$ that can be written as a product $g=w_1....w_n$
   with $w_k \in S$ (and $S^0=\{1_G \}$).  One can define $\ell_S(g)=\min \{ n \in \NN, g \in S^n \}$. Then, defining 
   $d_S(g,h)=\ell_S(gh^{-1})$, we obtain an adpated metric on $G$, called 
   the word metric (associated to the generating set $S$). Observe that generally, this metric is not continuous.
 
 By a coarse embedding $\alpha: G_1 \to G_2$ between two (locally compact, second countable) topological groups, we 
 will mean in what follows that $\alpha$ is coarse between $(G_1,d_1)$ and $(G_2,d_2)$, with $d_1,d_2$ adapted metrics. The 
 notion does not depend on the choice of such metrics, because of the  following easy topological characterization:
 
 \begin{fact}
  \label{fact.coarse.groups}
  A map $\alpha: G_1 \to G_2$ between two topological groups, endowed with adapted metrics, is a coarse embedding, when for each pair of sequences
   $(f_k)$ and $(g_k)$ in $G_1$, $f_kg_k^{-1}$ stays in a compact subset of $G_1$ if and only if 
   $\alpha(f_k)\alpha(g_k)^{-1}$ stays in a compact subset of $G_2$.
 \end{fact}

%

\subsection{The derivative cocycle is a coarse embedding}
\label{sec.derivative.coarse}
\subsubsection{Lie topology on $\Iso(M,g)$}
\label{sec.topology}
Let $(M^{n+1},g)$ be a compact $(n+1)$-dimensional Lorentz manifold. We recall briefly how one makes $\iso(M,g)$ into a Lie transformation group.
 Let us 
call $\hat{M}$ the bundle of orthogonal frames on $M$, which is 
 a $\OO(1,n)$-principal bundle over $M$.  Notice that every $f \in \Iso(M,g)$ induces 
 naturally a diffeomorphism $\hat{f}: \hm \to \hm$, which moreover preserves {\it a parallelism} on $\hm$, coming from the 
 Levi-Civita connection of $g$. One then shows that $\Iso(M,g)$ acts freely on $\hm$,
  and orbits of $\Iso(M,g)$ on $\hm$ are closed submanifolds of $\hm$ (closed but of course generally not compact).
   This identification of $\Iso(M,g)$ as 
   a submanifold of $\hm$ is the way one defines the Lie group topology on $\Iso(M,g)$ (see 
   \cite[Cor VII.4.2]{sternberg} for a more detailed account). In particular $\Iso(M,g)$ is secound countable and locally compact, 
   hence 
   admits adapted metrics. Observe that for a compact Lorentz manifold $(M^{1+n},g)$, it is not clear (and maybe false, eventhough we don't have any example)
    that the group $\Iso(M,g)$ is compactly generated. 
   Using the fact that the exponential map of $g$ 
   locally linearizes isometries, it is not very hard to check that this Lie topology is actually the $C^0$-topology on 
   $\Iso(M,g)$, making $\Iso(M,g)$ a closed subgroup of ${\text{Homeo}}(M)$. The very definition of the Lie topology implies 
   that {\it $\Iso(M,g)$ acts properly on $\hm$}.
 \subsubsection{Coarse embedding of $\Iso(M,g)$ into $\OO(1,n)$}
 \label{sec.derivative.cocycle}
We  fix once for all a section $\sigma: M \to \hm$  such that { the image $\sigma(\hm)$ has compact 
closure} in $\hm$. Such sections exist since $M$ is compact.
  {\it In all the paper, we will always deal with such bounded sections.}
  
Having fixed a bounded section $\sigma: M \to \hm$ as above, we get for every $x \in M$ a map 
$$ {\dd}_x: \Iso(M,g) \to \OO(1,n)$$
defined by the following relation:
$$ \sigma(f(x))=\hat{f}(\sigma(x)).({\dd}_x(f))^{-1}.$$

The element ${\dd}_x(f)$ is nothing but the matrix of the tangent map $D_xf: T_xM \to T_{f(x)}M$, if we put the frames
$\sigma(x)$ and $\sigma(f(x))$ on $T_xM$ and $T_{f(x)}M$ respectively. 
 
In \cite{gromov}, M. Gromov made the following crucial observation: 

\begin{lemme}[see \cite{gromov} Sections 4.1.C and 4.1.D]
 \label{lem.fondamental}
 Let $(M^{n+1},g)$ be a compact Lorentz manifold. Then for every $x \in M$, the derivative cocycle
 ${\dd}_x: \Iso(M,g) \to \OO(1,n)$ is a coarse embedding. 
\end{lemme}

\begin{proof}
 The proof follows easily from the properness of the action of $\Iso(M,g)$ on $\hm$. 
  Indeed, given two sequences $(f_k)$ and $(g_k)$ of $\Iso(M,g)$, $f_kg_k^{-1}$ stays in a compact subset of
   $\Iso(M,g)$ if and only if ${\hat f}_k{\hat g}_k^{-1}\sigma(g_kx)$ stays in a compact set of $\hm$ (because 
   $\sigma(g_kx)$ is contained in a compact subset of $\hm$ and $\Iso(M,g)$ acts properly on $\hm$).
   But from the relation:
   $$ {\hat f}_k{\hat g}_k^{-1}\sigma(g_kx){\dd}_x(g_k){\dd}_x(f_k)^{-1}=\sigma(f_kx),$$
   and the properness of the right action of $\OO(1,n)$ on $\hm$, we see that our initial assertion is 
   equivalent to ${\dd}_x(g_k){\dd}_x(f_k)^{-1}$ (hence ${\dd}_x(f_k){\dd}_x(g_k)^{-1}$)  staying in a compact subset of $\OO(1,n)$, and we conclude by 
   Fact \ref{fact.coarse.groups}.
\end{proof}

\subsubsection{Coarse embedding into $\HH^n$}
\label{sec.embedding.hyp}
Let ${\mathbb H}^n$ denote the real hyperbolic space, and let $o \in {\mathbb H}^n$ be a base point. The group
 $\OO(1,n)$ acts isometrically on $\HH^n$. If $x \in M^{n+1}$ is given, we can derive from ${\dd}_x$
  a coarse embedding $\overline{\dd}_x: \Iso(M,g) \to (\HH^n, d_{hyp})$, defined by
   $\overline{\dd}_x(f)={\dd}_x(f)^{-1}.o$. The embedding $\overline{\dd}_x$ is coarse because so is ${\dd}_x$, and the 
    action of $\OO(1,n)$ on $\HH^n$ is proper. 


\subsection{Coarse embeddings and obstructions to isometric actions}

\subsubsection{Restriction to quasi-geodesic spaces of bounded geometry}
\label{sec.quasi-geodesic}
To get useful invariants under coarse equivalence of metric spaces, it is better to focus on spaces which are quasi-geodesic, and 
of bounded geometry.
 We recall that a metric space $(X,d_X)$ is quasi-geodesic if there exist $a,b >0$ such that each pair of points $(x,y)$
  in $X$ can be joigned by a $(a,b)$-quasi-geodesic. In other words, there exists, for any such pair $(x,y)$ an interval $[0,L]$, as well 
   as a map $\gamma: [0,L] \to X$ with $\gamma(0)=x$ and $\gamma(L)=y$, and for every $t,t' \in [0,L]$:
   $$ \frac{1}{a}|t'-t|-b \leq d_X(\gamma(t'),\gamma(t)) \leq a|t'-t|+b.$$
   
 A metric space $(X,d_X)$ has bounded geometry, if it is quasi-isometric to some $(X',d_{X'})$ satisfying:
 \begin{enumerate}
  \item $(X',d_{X'})$ is uniformly discrete, namely there exists $C>0$ such that $\inf_{x \not = x'}d_{X'}(x,x') \geq C$.
  \item For every $r>0$, there exists $n_r$ such that every ball of radius $r$ has at most $n_r$ elements.
 \end{enumerate}

For us, the basic examples of quasi-geodesic metric spaces of bounded geometry will be homogeneous Riemannian manifolds, for instance 
Euclidean space $\RR^n$ or real hyperbolic space $\HH^n$. Also very important is the case of a second countable, locally compact and 
 compactly generated group $G$, endowed with a word metric $d_S$. The very definition of the word metric makes $(G,d_S)$ a quasi-geodesic space.
  It may not be locally finite, but we can always consider $\Lambda \subset X$ be a $C$-metric lattice in $G$. It means that $d(x,x') \geq C$ if $x \not = x'$ in $\Lambda$, and also that there 
   is a constant $D>0$ such that any $x \in X$ is at distance at most $D$ from  $\Lambda$. Such lattices exist by Zorn lemma, and for $C>0$
    big enough, $(\Lambda,d_S)$ is locally finite, uniformly discrete,  and quasi-isometric to $(G,d_S)$. In particular
     $(\Lambda,d_S)$ is quasi-geodesic.

One important feature of the quasi-geodesic assumption is the following:
 
 \begin{lemme}
  \label{lem.linear}
  Let $(X,d_X)$ and $(Y,d_Y)$ be two  metric spaces, with $(X,d_X)$  quasi-geodesic. Then for any coarse embedding 
  $ \alpha: X \to Y$, the upper control function $\phi^+$ can be chosen affine.
 \end{lemme}
\begin{proof}
 By assumption, there exist $a,b>0$ such that between any pair of point $x$ and $y$, there exists a $(a,b)$-quasi-geodesic.
 Since $\alpha$ is a coarse embedding, there exists $c>0$ such that for every $x,y \in X$, 
 $d_X(x,y) \leq a+b$ implies $d_Y(\alpha(x),\alpha(y) \leq c$.
 Let $x,y \in X$, and $\gamma: [0,L] \to X$ be a $(a,b)$-quasi-geodesic between $x$ and $y$.  
  Then $d_Y(\alpha(x),\alpha(y)) \leq \Sigma_{i=0}^{E(L)}d_Y(\alpha(\gamma(i)), \alpha(\gamma(i+1))) \leq Lc$.
   But since $\gamma$ is a $(a,b)$-quasi-geodesic, we have $\frac{1}{a}L-b \leq d_X(x,y)$, so that $d_Y(\alpha(x),\alpha(y)) \leq ac\, d_X(x,y)+abc$.
\end{proof}

\subsubsection{Coarse embeddings and growth}
\label{sec.coarse.growth}

Let $(X,d_X)$ be a metric space of bounded geometry, that we assume first to be uniformly discrete.
 For any $x_0 \in X$, and $r >0$, we denote by $\beta_X(x_0,r)=|B(x_0,r)|$ (the number of points in $B(x_0,r)$).
 
 Given two functions $f: \RR_+ \to \RR_+$ and $g: \RR_+ \to \RR_+$, one says that $f \preceq g$ 
  when there exist constants $\lambda, \mu$ and $c$ such that $f(r) \leq \lambda g(\mu r+c)$, and $f \approx g$ when $f \preceq g$
   and $g \preceq f$. It is pretty clear that changing $x_0$ into $y_0$, one gets $\beta_X(x_0, \dot) \approx \beta_X(y_0,\dot)$. 
   
 Recall that $(X,d_X)$ is said to have  polynomial growth if
  there exists a constant $C >0$, and $d \in \NN$, such that $\beta_X(x_0,r) \leq Cr^d$ for $r$ big enough. The minimal $d$ for which 
  it holds is then called the growth degree of $(X,d_X)$. 
   
One says that $(X,d_X)$ has exponential growth, when there exists 
some $a>0$ such that $\beta_X(x_0,r) \geq e^{ar}$ for $r$ big enough.

The property $\beta_X(x_0, \dot) \approx \beta_X(y_0,\dot)$ ensures that this definition is independent of the choice of
$x_0$.

If $(X,d_X)$ is a space with bounded geometry, which is not locally finite, we take some $(X',d_{X'})$ which is quasi-isometric to $(X,d_X)$
 and locally finite, and define $\beta_X=\beta_{X'}$ (equality makes sense up to the relation $\approx$). 

\begin{lemme}
 \label{lem.growth}
 Let $(X,d_X)$ and $(Y,d_Y)$ be two quasi-geodesic metric spaces of bounded geometry.
  Then the growth function of $Y$ dominates that of $X$, namely $\beta_X \preceq \beta_Y$.
  In particular, if $(Y,d_Y)$ has polynomial growth of degree $d$, then $(X,d_X)$ has polynomial growth of degree $d' \leq d$,
   and if $(X,d_X)$ has exponential growth, then the same holds for $(Y,d_Y)$.
\end{lemme}

\begin{proof}
We may assume again that $X$ and $Y$ are uniformly discrete.
 Because $\alpha$ is a coarse embedding, there exists $C>0$ such that if $d_X(x,y) \geq C$, then $d_Y(\alpha(x),\alpha(y)) \geq 1$.
  Now let $\Lambda \subset X$ be a $C$-metric lattice in $X$. It means that $d(x,x') \geq C$ if $x \not = x'$ in $\Lambda$, and also that there 
   is a constant $D>0$ such that any $x \in X$ is at distance at most $D$ from  $\Lambda$.
   From the fact that all balls of radius $D$ have at most $k_D$ points, it is easy to check that
    if $x_0 \in \Lambda$, then 
    \begin{equation}
     \label{eq.equivalent}
     \beta_{\Lambda}(x_0,r) \leq \beta_X(x_0,r) \leq k_D \beta_{\Lambda}(x_0,r+D). 
    \end{equation}

    Let us call $y_0=\alpha(x_0)$. From Lemma \ref{lem.linear}, there exist $a,b>0$ such that $\alpha(B_X(x_0,r)) \subset B_Y(y_0,ar+b)$.
    By definition of $C$, $\alpha$ is one-to-one in restriction to $\Lambda$, so that 
    $\beta_{\Lambda}(x_0,r) \leq \beta_Y(y_0,ar+b)$. From (\ref{eq.equivalent}), we infer 
    $\beta_X(x_0,r) \leq k_D \beta_Y(y_0,ar+aD+b)$, which concludes the proof.
\end{proof}

If we consider $(M^{n+1},g)$ a Lorentz manifold, and $G \subset \Iso(M,g)$ a closed, compactly generated subgroup, then
 we inherits a coarse embedding $\overline{\dd}_x: G \to \HH^n$ (see Section \ref{sec.embedding.hyp}). Since the growth of 
 $\HH^n$ is exponential, Lemma \ref{lem.growth} does not put any restriction on $G$. 
  But when some extra geometric informations are available, then some useful conclusions can be drawn, as we see now (see 
  also Section \ref{sec.expo.derivatives}).

\subsubsection{Reduction of the structure group on compact sets}
\label{sec.reduction}
It is plain that one can restrict the derivative cocycle ${\dd}_x$ to {\it closed subgroups} of $\Iso(M,g)$,
 still getting a coarse embedding. This can be especially interesting when such a subgroup preserves a 
 reduction of  the bundle $\hm$.  We have indeed:
 
 \begin{corollaire}
  \label{coro.subbundle}
  Let $(M^{n+1},g)$ be a Lorentz manifold. Assume that there exists a closed subgroup 
  $G \subset \Iso(M,g)$, leaving invariant a compact subset $K \subset M$, as well as a 
  reduction of $\hm$ above $K$, to an $H$-subundle with $H \subset \OO(1,n)$ a closed subgroup.
  
  Then there exists a coarse embedding $\alpha: G \to H$.
 \end{corollaire}

 \begin{proof}
  The proof is a straigthforward rephrasing of that of Lemma \ref{lem.fondamental}, looking at a bounded section 
  $\sigma: K \to \hat{N}$, where $\hat{N}$ is the $H$-subbundle over $K$.
 \end{proof}

 As a toy application in the realm of Lorentz geometry, let us formulate the following
 
 \begin{proposition}
  \label{prop.centralizer}
  Let $(M^{1+n},g)$ be a compact, $(n+1)$-dimensional Lorentz manifold. Assume that there exists on $M$
   a nontrivial Killing field $X$ which is everywhere lightlike , namely $g(X,X)=0$. Then every closed, finitely generated, subgroup 
   $\Lambda \subset \Iso(M,g)$ which commutes with $X$ is virtually nilpotent, of growth degree at most $n-1$.
 \end{proposition}

 \begin{proof}
  It is classical that nontrivial lightlike Killing fields on Lorentz manifolds can not have singularities. Thus, the field
   $X$ defines a reduction of the frame bundle $\hm$ to the group $L \subset \OO(1,n)$, where $L$ is the stabilizer of a 
   lightlike vector in Minkowski space $\RR^{1,n}$. This group $L$ is isomorphic to $\OO(n-1) \ltimes \RR^{n-1}$.
   Any closed finitely generated $\Lambda \subset \Iso(M,g)$ commuting with $X$ preserves the reduction, hence coarsely
    embeds into $L$ by corollary \ref{coro.subbundle}. Lemma \ref{lem.growth} implies that $\Lambda$ has polynomial growth, with growth degree
    $\leq n-1$. The proposition follows from Gromov's theorem about groups with polynomial growth.
 \end{proof}

\subsubsection{Coarse embedding and asymptotic dimension}
\label{sec.asymptotic}
The notion of asymptotic dimension of a metric space appears in \cite[Section 4]{gromov}, and was developped in \cite{gromov2}.
  
 One says that $(X,d_X)$ has asymptotic dimension at most $n$, written $\Asdim(X) \leq n$, if for every $r>0$, there exists 
  a covering $X=\bigcup_{i \in I} U_i$ with the following properties
  \begin{itemize}
   \item All the $U_i$'s are uniformly bounded.
   \item The $U_i$'s can
   be splitted into $(n+1)$ families 
 ${\mathcal U}_0, \ldots, {\mathcal U}_n$,  with the property that whenever $U_i$ and $U_j$, $i \not = j$,
 belong to a same family ${\mathcal U}_k$, then $d_X(U_i,U_j)>r$.
  \end{itemize}

The asymptotic dimension of $(X,d_X)$ is the minimal $n$ for which one has $\Asdim(X) \leq n$.  
 
The following lemma follows almost directly from the definitions.

\begin{lemme}
 \label{lem.asdim}
 Let $(X,d_X)$ and $(Y,d_Y)$ be two metric spaces, and $\alpha: X \to Y$ a coarse embedding. Then 
 $\operatorname{Asdim}(X) \leq \Asdim(Y)$. 
\end{lemme}

In particular, two spaces which are coarse-equivalent will have the same asymptotic dimension. 

It is known that $\Asdim(\RR^n)=\Asdim(\HH^n)=n$ for every $n \in \NN$ (see for instance \cite[chap. 10]{schroeder}).  Hence, Lemma \ref{lem.asdim}
 says, for instance, that if $(M^{1+n},g)$ is a compact, $(n+1)$-dimensional manifold, and if $\Lambda \subset \Iso(M,g)$
  is a discrete subgroup isomorphic to $\ZZ^k$, then $k \leq n$. One has the sharper upper bound  $k \leq n-1$, as 
   the following statement shows:
   
\begin{proposition}
 \label{prop.nocoarse.Rn}
 Let $n \geq 1$. Then there is no coarse embedding $\alpha: \RR^n \to \HH^n$. 
\end{proposition}

Observe that horospheres in $\HH^n$ yield coarse embeddings of $\RR^{n-1}$ into $\HH^n$.

\begin{proof}
 The proposition can not be derived in a straigthforward way, because Lemmas \ref{lem.growth} and \ref{lem.asdim} are obviously useless
  in our situation. The ingredients of the proof are more elaborate (though classical for the experts), 
  and involve ``coarse topological arguments''. The details can be found in \cite[section 9.6]{drutu}, 
  and especially Theorem $9.69$ there. This theorem states  that if a coarse embedding $\alpha: \RR^n \to \HH^n$
   would exist, then it should be almost surjective (namely the image is cobounded). Then 
    it is easy to build a coarse inverse $\beta: \HH^n \to \RR^n$ to $\alpha$.  But the existence of such a $\beta$
    is forbidden
    because of Lemma \ref{lem.growth}. 
\end{proof}

\subsection{A first glance at Lorentz isometric actions of lattices}
\label{sec.reseau1}

\subsubsection{Isometric actions of lattices in $\OO(1,k)$}
To illustrate how the basic tools of coarse geometry presented so far can already say 
interesting things about 
 isometric actions on Lorentz manifolds, let us prove part of Corollary \ref{thm.lattices}, without appealing to 
 Theorem \ref{thm.main}.

\begin{proposition}
\label{prop.reseaux}
 Let $(M^{n+1},g)$ be a compact Lorentz manifold. Then if $k \geq n+1$,
  $\Iso(M,g)$ does not contain any discrete subgroup $\Lambda$
  isomorphic to a lattice of $\OO(1,k)$.
\end{proposition}

\begin{proof}
 Assume that $\Lambda \subset \Iso(M,g)$ is discrete and isomorphic to a lattice in $\OO(1,k)$.
  Observe first that $\Lambda$ is closed and finitely generated, hence we can apply all what we did so far.
   If $\Lambda$ is uniform in $\OO(1,k)$, then its asymptotic dimension is that of $\HH^k$, and by 
    the coarse embedding $\overline{\dd}_x: \Lambda \to \HH^n$ of Section \ref{sec.embedding.hyp} and 
    Lemma \ref{lem.asdim}, we get $k \leq n$, and we are done.
  If $\Lambda$ is not uniform, then the {\it thick-thin} decomposition (see \cite[Section 4.5]{thurston}) provides a discrete subgroup
   of $\Lambda$, virtually isomorphic to $\ZZ^{k-1}$. But then Proposition \ref{prop.nocoarse.Rn} forces $k-1 \leq n-1$, which concludes
   the proof.
    
\end{proof}

The statement is optimal. Indeed, we saw in the introduction that on any torus $\TT^{n+1}$ ($n \geq 2$), there exist flat Lorentz
 metrics, with  isometry group  $\Lambda \ltimes \TT^{n+1}$, where $\Lambda$ is some lattice in $\OO(1,n)$.
 
 \subsubsection{Compact Lorentz manifolds having a lattice as isometry group}
 \label{sec.examples.lattices}
At this point, it is worth noticing that one can produce examples of compact Lorentz $(n+1)$-manifolds ($n \geq 3$) admitting
an isometry group which is virtually a lattice
 in $\OO(1,n-1)$. Here is the construction. We start with a Lorentz quadratic form $q$ on $\RR^n$, $n \geq 3$, which is rational.
  Then, Borel Harish-Candra theorem ensures that $\Lambda:=\OO(q,\ZZ)$ is a lattice in $\OO(q)$. The quadratic form $q$ defines
   a Lorentz metric $g_0$ on $\RR^n$, which is flat and translation invariant. We consider $\RR \times \RR^n$, with coordinates $(t,x_1,\ldots,x_n)$, and we endow it with 
   the Lorentz metric $g_a:=dt^2+ a(t)g_0$, where $a: t \mapsto a(t)$ is a positive, $1$-periodic function on $\RR$. 
   Now, let $\Gamma$ be the subgroup generated by $\gamma: (t,x_1,\ldots,x_n) \mapsto (t+1, -x_1,\ldots,-x_n)$
    and the translations $\tau_i: (t,x_1,\ldots,x_i,\ldots,x_n) \mapsto (t,x_1,\ldots,x_i+1,\ldots,x_n)$, $1 \leq i \leq n$. This is
     a discrete subgroup, acting by isometries for $g_a$.
   
   The quotient manifold $M:=(\RR \times \RR^n)/ \Gamma$ is topologically a $\TT^n$-bundle over the circle. 
   It inherits a Lorentz metric ${\overline g}_a$, for which the action of $\Lambda$ is isometric. Actually, for a generic choice
    of the $1$-periodic function $a$, $\Iso(M,\overline{g}_a)$ will coincide virtually with $\Lambda$ 
    (see \cite[Sec. 2.2, Lemma 2.1]{frances}
     for the precise genericity condition that has to be put on $a$).

\subsection{Deducing Theorem \ref{thm.tits} from Theorem \ref{thm.main}}
\label{sec.deduce}
All the invariants of coarse embeddings (growth, asymptotic dimension) presented so far are very basic. More sofisticated 
 tools, leading to obstructions for a group to admit a coarse embedding into some real hyperbolic space $\HH^d$ can be found in 
  \cite{hume-sisto}. For instance, it is shown in  \cite[Cor. 1.3]{hume-sisto} that a finitely generated, virtually solvable group, which 
  is not virtually nilpotent does not admit such a coarse embedding (see also Theorem \ref{thm.tessera} below). It follows that
   those groups do not appear as 
 closed subgroups of isometries for a compact Lorentz manifold. Observe that this result can also be infered from Theorem \ref{thm.main}, 
 because it is easy to check that a discrete subgroup of $\OO(1,d)$ which is virtually solvable has to be virtually abelian.
 
 Recently, several authors (\cite{tessera1}, \cite{hume2}, \cite{lecoz}) studied the behaviour of notions such as separation, and isoperimetric  profiles, with
  respect to coarse embeddings. This led to the following quite amazing result of R. Tessera:

\begin{theoreme}\cite[Th. 1.1]{tessera}
 \label{thm.tessera}
 Let $\Lambda$ be a finitely generated, amenable group. If there exists a coarse embedding from $\Lambda$ to $\HH^n$, then
  $\Lambda$ is virtually nilpotent of growth degree at most $n-1$.
\end{theoreme}

This result, combined with Theorem \ref{thm.main} implies easily Theorem \ref{thm.tits}. Indeed, let $(M^{n+1},g)$ be a compact, 
$(n+1)$-dimensional, Lorentz manifold, and let $\Lambda \subset \Iso(M,g)$ be a discrete, finitely generated, subgroup.
 We endow $\Lambda$ with a word metric, and consider its growth function, namely $\beta: n \mapsto |B(1_{\Lambda},n)|$. 
 If the growth of $\Lambda$ is subexponential, namely if $\lim_{n \to \infty} \frac{1}{n}\log(\beta(n))=0$, then $\Lambda$
  is amenable. Because $\Lambda$ coarsely embeds into $\HH^n$ by Lemma \ref{lem.fondamental}, Tessera's theorem ensures that $\Lambda$
   is virtually nilpotent, of growth degree at most $n-1$.
   
 If on the contrary, the growth of $\Lambda$ is exponential, then we can apply Theorem \ref{thm.main}. 
 
 We obtain, in both cases covered by the theorem, a finite index subgroup  $\Lambda' \subset \Lambda$, and
  a proper homomorphism $\rho: \Lambda' \to \PO(1,n)$, $n \geq 2$. The image  $\rho(\Lambda')$ 
   is 
  a discrete subgroup $\Lambda_{\rho}$ of $\PO(1,n)$, and because the kernel of $\rho$ is finite, $\Lambda $ is virtually 
  isomorphic to $\Lambda_{\rho}$. Now discrete subgroups
  of $\PO(1,n)$ split into two categories (see \cite{matsu} for an introduction to Kleinian groups). The elementary ones, for 
  which the limit set is finite. Those groups are virtually abelian, hence 
  don't have exponential growth. This case is not compatible with our growth assumption on $\Lambda$.
  The non-elementary groups of $\PO(1,n)$ have infinite limit set. It is known that such groups contain
   pairs of loxodromic elements $\alpha, \beta$ with pairwise disjoint fixed points (see for instance \cite[Lemma 2.3]{matsu}). It is then clear that suitable powers
    $\alpha^n$ and $\beta^n$ satisfy the Ping-pong lemma, hence generate a free group. The inverse image of this free group 
    in $\Lambda'$ is also a 
    free subgroup, what concludes the
    proof of Theorem \ref{thm.tits}. 
    
    The remaining of the paper is devoted to the proof of Theorem \ref{thm.main}. 

\section{The limit set of an isometric action}
\label{sec.limit.set}
\subsection{Dynamical definition of the limit set}
\label{sec.dynamical.definition}
Let $(M,g)$ be a compact Lorentz manifold. 
 Following \cite[Section 9.1]{zeghibisom1}, one can introduce a notion of (fiberwise) {\it limit set}
 for the action 
 of the isometry group $\Iso(M,g)$. For each  $x \in M$
 the ``limit set'' $ \lig(x) \subset {\mathbb P}(T_xM)$  comprises all {\it nonzero lightlike  directions} $ [u] \in 
 {\mathbb P}(T_xM) $,
  for which there exists a sequence $ (f_k) \in \Iso(M,g)$ tending to infinity in $\Iso(M,g)$, and a sequence 
  $(u_k) \in TM$ tending to $u$, such that
  $Df_k(u_k)$ is bounded in $TM$.  It follows from the definition that if $\Iso(M,g)$ is compact, then 
  $\lig(x)= \emptyset$ for every $x \in M$.  Conversely, it is pretty easy to check that 
  $\lig(x) \not = \emptyset$ for every $x \in M$ 
   as soon as $\Iso(M,g)$ is noncompact. As we will see later, the existence of a ``big'' limit set at each $x \in M$ 
   has strong geometric consequences on the Lorentz manifold $(M,g)$. 
   To estimate the size of $\lig(x)$, and give a precise meaning of big limit set, we introduce the map $\carg: M \to \NN \cup \{ \infty \}$
   which to a point $x \in M$ associates the number of points in $\lig(x)$. It will be also useful to consider 
   $\eg(x)$, the vector subspace of $T_xM$ spanned by $\lig(x)$, and denote by $\dig(x)$ the 
   dimension of $\eg(x)$. 
   
 \subsection{Asymptotically stable distributions and semi-continuity properties} 
 \label{sec.semi-continuity}
   A totally geodesic, codimension one, lightlike foliation on $M$, is a codimension one
    foliation, the leaves of which are totally geodesic, and such that the restriction of the metric $g$ to the 
    leaves is degenerate (one says {\it lightlike}). The work \cite{zeghibisom1} makes a link between 
    directions in the limit set, and totally geodesic codimension one lightlike foliations on $M$. To be precise,
     A. Zeghib introduces, for every sequence $(f_k)$ going to infinity in $\Iso(M,g)$, the 
     {\it asymptotically stable distribution} $AS(f_k) \subset TM$ of $(f_k)$ as follows: the vectors of 
     $AS(f_k)$ are those $ u \in TM$ for which there exists a sequence $(u_k)$ of $TM$ converging to $u$, and
     such that $Df_k(u_k)$ is bounded.  
    
  \begin{theoreme}\cite[Theorem 1.2]{zeghibisom1}
   \label{thm.zeghib}
   Let $(M,g)$ be a compact Lorentz manifold, and let $(f_k)$ be a sequence of $\Iso(M,g)$
    going to infinity. Then, after considering a subsequence of $(f_k)$, the asymptotically stable set 
    $AS(f_k)$ is a codimension one lightlike distribution, which is tangent to a Lipschitz totally geodesic 
   codimension one lightlike foliation.
   \end{theoreme}
   
   If $[u]$ is a direction belonging to $\lig(x)$, then it follows from the definitions that there exists
    a sequence $(f_k)$ going to infinity in $\Iso(M,g)$ such that $u \in AS(f_k)$. Actually, after considering a subsequence,
     $AS(f_k)(x)$ is a lightlike hyperplane of $T_xM$ coinciding with $u^{\perp}$. Let us call $\calf$
      the totally geodesic  codimension one lightlike foliation integrating  $AS(f_k)$, the existence of which is asserted by Theorem 
      \ref{thm.zeghib}. For every 
      $y \in M$, if $v \in T_yM$ is a nonzero lightlike vector such that $v^{\perp}$ is 
    tangent to the leaf of $\calf$ containing $y$, then $[v] \in \lig(y)$.
   It is known that totally geodesic 
   codimension one lightlike foliations  have Lipschitz transverse regularity.
   We thus get:
   \begin{corollaire}
    \label{coro.extension}
    Each  direction $[u] \in \lig(x)$ 
   can be extended to a Lipschitz field of lightlike directions  $y \mapsto \Delta_{[u]}(y) $ on $M$, 
   such that $\Delta_{[u]}(y) \in \lig(y)$ for every $y \in M$. 
   \end{corollaire}
This 
   yields the following semi-continuity property.

  
   \begin{corollaire}
    \label{coro.semicontinuous}
    The map  $x \mapsto \dig(x)$ is lower semi-continuous. For any $m \in \NN^*$, the sets 
    $K_{\geq m}:=\{ x \in M \ | \ \carg(x) \geq m \}$ are open.
   \end{corollaire}

\subsection{The point of view of coarse embeddings}
\label{sec.limit.coarse}
Let us consider a subset ${\mathcal G} \subset \OO(1,n)$, and a point $o \in \HH^n$.  We call ${\mathcal O}_{\mathcal G}$
 the set:
 $$ {\mathcal O}_{\mathcal G}:=\{ g^{-1}.o \ | \ g \in {\mathcal G}  \}.$$
 Then we introduce {\it the limit set} 
of ${\mathcal G}$, denoted $\Lambda_{\mathcal G}$, as $\Lambda_{\mathcal G}:=\overline{{\mathcal O}_{\mathcal G}} \cap \partial \HH^n$.
 The closure $\overline{{\mathcal O}_{\mathcal G}}$ is taken in the topological ball $\overline{\HH}^n$.
 We notice that $\Lambda_{\mathcal G} = \emptyset$ if and only if $\mathcal G$ has compact closure in $\OO(1,n)$.
 
 If $G \subset \Iso(M,g)$ is a closed subgroup and 
  if for $x \in M$, $\dd_x: G \to \OO(1,n)$ is the coarse embedding given by the derivative cocycle 
  (see Section \ref{sec.derivative.cocycle}), then it makes sense to consider the limit set $\Lambda_{\dd_x(G)}$, that we will 
  rather write  $\Lambda_{\dd}(x)$ to ease the notations. 
  
  It turns out that  the sets $\lig(x)$ and $\Lambda_{\dd}(x)$ 
  encode the same object. Let us explain why. First, denote by $\RR^{1,n}$ the space $\RR^{n+1}$
   endowed with the quadratic form $q^{1,n}=2x_1x_{n+1}+x_2^2+ \ldots + x_n^2$. The group $\OO(1,n)$ is the subgroup
    of $\GL(n+1,\RR)$ preserving $q^{1,n}$. We will use the projective model for real hyperbolic space $\HH^n$, namely we see the $\HH^n$ as 
     the set of timelike lines  (satisfying $q^{1,n}(x) < 0$) in $\RR^{1,n}$. Its boundary $\partial \HH^n$ coincides with
      the set of lightlike lines. The action of $\OO(1,n)$  on those sets is the obvious one, coming from the linear 
      action of $\OO(1,n)$ on $\RR^{1,n}$. Let us consider a bounded section $\sigma: M \to \hm$ 
      as in section \ref{sec.derivative.cocycle}, which defines the coarse embedding $\dd_x: \Iso(M,g) \to \OO(1,n)$.
       Each frame $\sigma(x)$ can be seen as a linear isometry $\sigma(x): \RR^{1,n} \to T_xM$. Let us denote by 
       $\iota_x: T_xM \to \RR^{1,n}$ the inverse map. We claim:
        
\begin{lemme}
 \label{lem.identification.limitset}
 For every $x \in M$, we have $\Lambda_{\dd}(x)=\iota_x(\lig(x))$.
\end{lemme}

\begin{proof}
 The very definition of $\dd_x$ yields, for every
       $f \in \Iso(M,g)$,  every $x \in M$, and every $u \in T_xM$,  
       the equivariance relation 
       \begin{equation}
        \label{eq.equivariance-cocycle}
        D_xf(u)=\sigma(f(x))(\dd_xf(\iota_x(u))).
       \end{equation}

       Now if $[u] \in \lig(x)$, there exists a sequence $(u_k)$ in $T_xM$, and a sequence $(f_k)$ going to infinity in $\Iso(M,g)$
        such that $|D_xf_k(u_k)|$ remains bounded. Relation (\ref{eq.equivariance-cocycle}) shows that 
        $ \dd_x f_k(\iota_x(u_k))$ remains bounded in $\RR^{1,n}$. Let us perform a Cartan decomposition of 
        $\dd_x f_k$ in $\OO(1,n)$, namely write $\dd_x f_k=m_ka_kl_k$ where $(m_k)$ and $(n_k)$ stay in a maximal compact group of $\OO(1,n)$
         and $a_k$ is a diagonal matrix of the from 
         $$ a_k= \left(  \begin{array}{ccc}  
                     \lambda_k & 0 & 0\\
                     0 & I_{n-1} & 0\\
                     0 & 0 & \lambda_k^{-1}\\
                    \end{array}
\right),$$
 where $\lambda_k \to + \infty$. We may also consider a subsequence, so that $m_k$ converges to $m_{\infty}$
  and $l_k$ to $l_{\infty}$. Then it is plain that  because  $\dd_x f_k(\iota_x(u_k))$ is bounded and $\iota_x(u)$ is lightlike, 
  one must have  $\iota_x(u) \in \RR.l_{\infty}^{-1}(e_{n+1})$. Also quite obvious is the fact that for  every timelike vector $v$, 
   $(\dd_x f_k)^{-1}(v)$ tends projectively to $[l_{\infty}^{-1}(e_{n+1})]$. This shows $[\iota_x(u)] \in \Lambda_{\dd}(x)$.
   The inclusion $\Lambda_{\dd}(x) \subset \iota_x(\lig(x))$ is shown in the same way.
\end{proof}

%
%
\subsection{Exponential growth and exponential growth of derivatives}
\label{sec.expo.derivatives}
We consider $G \subset \Iso(M,g)$ a closed, compactly generated subgroup. We choose $S$ a symmetric compact set generating $G$, and 
we consider the associated word metric $d_S$ (see Section \ref{sec.coarse.groups}).  For $g \in G$, we will 
 denote by $\ell(g)$ the distance $d_S(1_{G},g)$.
 
 Given a function $\psi: G \to \RR_+$, one says that $\psi$ has exponential growth, if there exists $\lambda>0$,
  and a sequence $(g_i)$ in $G$, which tends to infinity, and such that $\psi(g_i) \geq e^{\lambda \ell(g_i)}$.

We fix from now on an auxiliary Riemannian metric $h$ on $M$. The statements below will involve estimates with respect to this metric, but by 
 compactness of $M$, their conclusions won't depend on the choice of $h$. The norm of a vector $u \in TM$ with respect to 
 the metric $h$  will be denoted by $|u|$.  For $x \in M$, and $\varphi \in \operatorname{Diff}(M)$, we denote by
  $|D_x \varphi|:=\sup_{|u|=1}|D_x\varphi(u)|$. 
 
\begin{proposition}
 \label{prop.growth.derivatives}
 Let $(M,g)$ be a compact Lorentz manifold.  Let $G \subset \Iso(M,g)$ be a closed, compactly generated subgroup.
  If $G$ has exponential growth, then for every $x \in M$, the function $g \in G \mapsto |D_xg|$
   has exponential growth.
\end{proposition}
  
 \begin{proof}
 
 We consider, for each $x \in M$,  the coarse embedding $\dd_x: G \to \OO(1,n)$ given by the 
 derivative cocycle (see Section \ref{sec.derivative.cocycle}). Because $\dd_x$ is defined relatively to 
 a bounded section $\sigma: M \to \hm$,  Proposition \ref{prop.growth.derivatives} amounts to proving that 
 $g \mapsto |\dd_xg|$ has exponential growth, where we put any norm $|\,.\, |$ on the space of matrices.
 
 Let us recall Iwasawa's decomposition in $\OO(1,n)$. Each $\dd_xg \in \OO(1,n)$ can be written in a unique way
  as  $\dd_xg=k(g)a(g)n(g)$, where $k(g)$ belongs to a maximal compact subgroup $K \subset \OO(1,n)$
  
  $$ a(g)=\left( \begin{array}{ccc}
  e^{\lambda(g)} & 0 & 0\\
  0 & I_{n-1} & 0\\
  0& 0 & e^{- \lambda(g)}\\
   \end{array}
  \right), \ \ \lambda(g) \in \RR$$
  
  $$ n(g)=\left( \begin{array}{ccc}
                1 & v(g) & -\frac{<v(g),v(g)>}{2}\\
                0 & I_{n-1} & -^t v(g)\\
                0& 0 & 1\\
                 \end{array}
                 \right), \ \ v(g) \in \RR^{n-1}.
$$
  In the formula above, if $v=(v_1,\ldots,v_{n-1})$, then $<v,v>=v_1^2+ \ldots+ v_{n-1}^2$. We will use the notation $|v|$
   for $\sqrt{<v,v>}$.

  \begin{lemme}
   \label{lem.either.a.or.n}
   Assuming that $G$ has exponential growth, either $g \mapsto e^{|\lambda(g)|}$ or 
   $g \mapsto |v(g)|$ has exponential growth.
  \end{lemme}

  \begin{proof}
  Assume for a contradiction that both maps have subexponential growth.  
   
   Working in  the upper-half space model $\RR_+^* \times \RR^{n-1}$ for $\HH^n$, with coordinates $(t,x)$, 
   the action of 
     $a(g)$ is given by $(t,x) \mapsto (e^{\lambda(g)}t,e^{\lambda(g)}x)$, and that of 
     $n(g)$ by $(t,x) \mapsto (t,x+v(g))$. We consider the point $o=(1,0) \in \RR_+^* \times \RR^{n-1}$, which is precisely the point 
     fixed by the compact group $K$. 
       
       We already observed in Section \ref{sec.embedding.hyp} that the map:
       $\overline{\dd}_x: G \to \HH^n$ defined by 
       $$\overline{\dd}_xg:=(\dd_xg)^{-1}.o=n(g)^{-1}a(g)^{-1}.o$$
       is a coarse embedding.
    
   Our subexponential growth assumption tells that  for every $\epsilon>0$, there exists
   $N_{\epsilon} \in \NN$ such that for every $k \geq N_{\epsilon}$, and every $g$ satisfying $\ell(g) \leq k$, 
   $e^{|\lambda(g)|} \leq e^{k \epsilon}$ and $|v(g)| \leq e^{k \epsilon}$. Geometrically, it means 
   that for $k \geq N_{\epsilon}$, the ball 
   $B(1_{G},k)$ is mapped by
    $\overline{\dd}_x $ into a rectangle 
    $R_{k,\epsilon}=[e^{-k \epsilon}, e^{k \epsilon}] \times [-e^{k \epsilon},e^{k \epsilon}]^{n-1} \subset  \RR_+^* \times \RR^{n-1}$.
    Let us compute the hyperbolic volume of this rectangle in $\HH^n$: 
    $$\operatorname{vol}_{hyp}(R_{k,\epsilon})=2^{n-1}e^{k(n-1)\epsilon}\int_{e^{-k \epsilon}}^{e^{k \epsilon}}\frac{1}{t^n}dt
    =\frac{2^{n-1}}{n-1}(e^{2k(n-1)\epsilon}-1).$$
    
  Now $\overline{\dd}_x$ being a coarse embedding, 
  there exists $c>0$ such that for avery pair of elements $g,h \in G$, 
  $d_S(g,h) \geq c$ implies $d_{\HH^n}(\overline{\dd}_xg,\overline{\dd}_xh) \geq 1$.
 Let us choose a $c$-metric lattice $L$ in $G$.  Because $G$ has exponential growth, there exists
  $\mu >0$ such that $n_k=\sharp(L \cap B(1,k)) \geq e^{\mu k}$ for $k$ big enough. 
  On the other hand, $\overline{\dd}_x(L \cap B(1_{G},k)) \subset R_{k, \epsilon}$  contains $n_k$ points at mutual hyperbolic distance at least $1$.
   Hence there exists a constant $C>0$ (independent of $k$ and $\epsilon$), such that $n_k \leq C \operatorname{vol}_{hyp}(R_{k,\epsilon})$.
    We end up with the inequality 
    $e^{\mu k} \leq C \frac{2^{n-1}}{n-1}(e^{2k(n-1)\epsilon}-1)$, available for $k$ big enough. If we chosed 
    $\epsilon $ small, for instance $0 < \epsilon <\frac{\mu}{4(n-1)}$, this yields 
    a contradiction as $k$ tends to infinity.\end{proof}
 
 We see that if $a(g)= \left( \begin{array}{ccc} e^{\lambda(g)} & 0 & 0 \\
 0 & Id_{n-1} & 0 \\
 0 & 0 & e^{-\lambda(g)}
 \end{array} \right)$ and $n(g)= \left( \begin{array}{ccc} 1 & v(g) & \frac{|v(g)|^2}{2} \\
 0 & Id_{n-1} & -{}^tv(g) \\
 0 & 0 & 1
 \end{array} \right)$ then:
 $$ a(g)n(g)=\left( \begin{array}{ccc} e^{\lambda(g)} & e^{\lambda(g)}v & \frac{e^{\lambda(g)} |v(g)|^2}{2} \\
 0 & Id_{n-1} & -^tv \\
 0 & 0 & e^{-\lambda(g)}
 \end{array} \right).$$
 
 If $g \mapsto e^{|\lambda(g)|}$ has exponential growth, then so has $|a(g)n(g)|$  (look at diagonal terms).  If 
 $g \mapsto |v(g)|$ has exponential growth, then so has $|a(g)n(g)|$  (look at the last column).
 Since $\dd_xg=k(g)a(g)n(g)$ with $k(g)$ in a compact set, it follows from Lemma \ref{lem.either.a.or.n} that $g \mapsto |\dd_xg|$ has exponential growth.  
 This concludes the proof.
 \end{proof}
 
 \subsection{Exponential growth implies infinite limit set}
 
 \label{sec.exp.growth}
 
 \begin{proposition}
  \label{prop.growth.limitset}
  Let $(M,g)$ be a compact Lorentz manifold.  Let $G \subset \Iso(M,g)$ be a closed, 
  compactly generated subgroup.
  If at some $x \in M$, the map $g \mapsto |D_xg|$
   has exponential growth, then the limit set $\Lambda(x)$ has at least two points.
 \end{proposition}

\begin{proof}
 The main ideas of the proof are borrowed from \cite{hurtado}.
 
 We will use again the map $\dd_y: G \to \OO(1,n)$ defined in Section \ref{sec.derivative.cocycle}, which 
 is a coarse embedding for every $y \in M$.
 We consider $S$ a compact, symmetric generating set for the group $G$. Let us 
 call  $\Sigma=S^{\ZZ}$ the set of bi-infinite words
 $w=(\ldots w_{-2}w_{-1}w_0w_1w_2 \ldots)$ in the alphabet $S$.  The shift $\sigma: \Sigma \to \Sigma$ is defined as $\sigma(w)=w'$
  where $w'_i=w_{i+1}$ (we shift to the right). We call  $\overline{O}$ the closure of the orbit $G.x$, and we 
   define $\theta: \Sigma \times \overline{O} \to \Sigma \times \overline{O}$
   by $\theta(w,y)=(\sigma(w),w_0.y)$.  We observe that $\theta^{-1}(w,y)=(\sigma^{-1}(w),w_{-1}^{-1}.y)$.
   
 The action of $G$ on $M$ defines naturally a map $A: \Sigma \times M \to O(1,n)$, by 
 the formula $A(w,y):=\dd_yw_0$.  
 Let us put 
 $$A^{(k)}(w,y)=A(\theta^{k-1}(w,y))A(\theta^{k-2}(w,y))\ldots A(\theta(y))A(y),$$
 which by the cocycle property of the derivative cocycle is nothing but
  $\dd_yw_{k-1}....w_1w_0$.  
 In the same way we define 
 $$A^{(-k)}(w,y)=A^{-1}(\theta^{-k}(w,y)) \ldots A^{-1}(\theta^{-1}(w,y)),$$
  which coincides $\dd_yw_{-k}^{-1} \ldots w_{-2}^{-1}w_{-1}^{-1}$.  
  
 We obtain in this way a cocycle over $\theta: \Sigma \times M \to \Sigma \times M$ with values into $\OO(1,n)$.
 Because $\Sigma \times \overline{O}$ is compact, there exists $\nu$  an ergodic $\theta$-invariant Borel probability measure 
 on $\Sigma \times \overline{O}$.
 Observe that $(w,y) \mapsto \log || A(w,y)||$ and $(w,y) \mapsto \log || A^{-1}(w,y)||$ are integrable for $\nu$ because those
  are Borel maps which are bounded (since $S$ is compact and the derivative cocycle 
  $(x,f) \mapsto {\mathcal D}_xf$ is defined relatively to a bounded 
  frame field).
 
 Let us recall the conclusions of Oseledec theorem in this context (see for instance \cite[Theorem 4.2]{ledrappier}).
  \begin{theoreme}[Oseledec Theorem]
   \label{thm.osseledecs}
   There exists a $\theta$-invariant Borel subset $B \subset \Sigma \times \overline{O}$  
  with $\nu(B)=1$,  a measurable
   decomposition $\RR^{n+1}=W_z^1 \oplus \ldots \oplus W_z^r$, $z \in B$, and real 
   Lyapunov exponents $\lambda_1 < \lambda_2 < \ldots < \lambda_r$ satisfying the properties:
   \begin{enumerate}
    \item The decomposition $\RR^{n+1}=W_z^1 \oplus \ldots \oplus W_z^r$ is invariant in the sense that
   $A(z)W_z^i=W_{\theta(z)}^i$ for all $z \in B$ and $1 \leq i \leq r$.
   \item A vector $v$ belongs to $W_z^i$ if and only if
   \begin{equation}
    \label{eq.lyapunov}
    \lim_{k \to \pm \infty} \frac{1}{k} \log || A^{(k)}(z).v|| =  \lambda_i.
   \end{equation}
   \end{enumerate}
  \end{theoreme}

   Recall that the matrices $A^{(k)}(z)$ preserve a Lorentz scalar product $< , >_{1,n}$ on $\RR^{n+1}$ hence have 
    determinant $1$. It follows that the sum $\lambda_1+\ldots+\lambda_r=0$ (see \cite[Prop. 1.1]{ledrappier}).   
   One infers  easily from equation (\ref{eq.lyapunov}) that the Lyapunov spaces 
   $W_z^i$ are mutally orthogonal for $< , >_{1,n}$, and spaces associated to nonzero exponents are isotropic 
   (hence $1$-dimensional). 
   We thus see that only two cases may occur. Either there is a single exponent and this exponent is $0$, 
   or there are exactly three exponents (if $n \geq 2$)
   $\lambda^+ >0, 0$ and $ \lambda^-=-\lambda^+$. In this last case, we say that the measure 
   $\nu$ is {\it partially hyperbolic}.  We then denote by $W_z^-$, $W_z^0$ and $W_z^+$ the  Lyapunov spaces associated
    to $\lambda^-,0$ and $\lambda^+$ respectively.
   
   Let us proceed with the proof of the proposition, assuming that  we have found such a partially hyperbolic measure $\nu$.
    Oseledec theorem 
   above yields a point $z=(w,y) \in B$, and linearly independent vectors $u^+ \in W_z^+$ and $u^- \in W_z^-$
       such that, 
       $$ \lim_{k \to + \infty} \frac{1}{k} \log( || A^{(k)}(w,y).u^-||) = - \lambda^+$$
       and 
       $$ \lim_{k \to + \infty} \frac{1}{k} \log( || A^{(-k)}(w,y).u^+||) = - \lambda^+.$$
         Let $w=(\ldots w_{-2}w_{-1}w_0 w_1 w_2 \ldots)$. We  define  $g_k=w_{k-1}\ldots w_0$ and   
         $g_k'=w_{-k}^{-1} \ldots w_{-1}^{-1}$.  Then:
         $$ \lim_{k \to + \infty}\frac{1}{k}\log ||\dd_yg_k(u^-)|| = -\lambda^+ \ \  \text{and} \ \  \lim_{k \to + \infty}
         \frac{1}{k}\log ||\dd_yg_k'(u^+)|| = -\lambda^+.$$
          In particular both sequences $(g_k)$ and $(g_k')$ tend to infinity, 
          and $u^-$  (resp. $u^+$) is a lightlike asymptotically
           stable vector for $(g_k)$ (resp. for $(g_k')$). It follows that the directions
            $[u^{\pm}]$ belong to the limit set $\Lambda(y)$, showing that this set  has at least two points.
         In an open set $U$ around $y$, we must have $d_{\Lambda} \geq 2$  
         (semi-continuity property shown in Corollary \ref{coro.semicontinuous}).
          Because $y \in \overline{O}$, it follows that for some point, hence any 
          (by isometric invariance) point of $O$, $d_{\Lambda} \geq 2$.
           This concludes the proof of Proposition \ref{prop.growth.limitset} in this case.

 It remains to show the existence of a $\theta$-invariant, ergodic, {\it partially hyperbolic} measure on 
 $\Sigma \times \overline{O}$.  This is here that our  assumption 
  of exponential growth for the derivatives of  $G$ comes into play.  
  We fix an auxiliary Riemannian metric $h$
  on $M$, and denote by $UT_M$ the unit tangent bundle of $h$. Exponential growth of derivatives at $x$
  yields   $\lambda >0$,  
   a sequence $(g_k)$ in 
  $G$ such that $\ell(g_k) \to \infty$, as well as a sequence $v_k$ of $UT_xM$, such that 
  $||D_xg_k(v_k)|| \geq e^{\lambda \ell(g_k)}$.
  
This growth condition  can also be formulated using the sequence $A^{(k)}$ in the following way. 
For each $k$, $g_k$ can be written as $g_k=s_{\ell(g_k)-1}^{(k)}s_{\ell(g_k)-2}^{(k)}\ldots s_1^{(k)}s_0^{(k)}$,
 which is a word of length $\ell(g_k)$ in elements of $S$. 
     Then, let us consider the bi-infinite, periodic word
    ${w}^{(k)} \in S^{\ZZ}$  which is defined by $w_i^{(k)}=s_i^{(k)}$ for $0 \leq i \leq \ell(g_k)-1$, and completed by periodicity. 
  The growth condition yields the existence of $u_k$ a sequence of vectors in $\RR^{n+1}$, such that
  $||u_k||=1$ (we use here any norm) and
  $$|| A^{(\ell(g_k))}(w^{(k)},x).u_k|| \geq e^{\lambda \ell(g_k)}.$$
  
  We now lift the map $\theta$  to a map 
  $\hat{\theta} : \Sigma \times \overline{O} \times \RP^{n} \to \Sigma \times \overline{O} \times \RP^{n}$ by the formula
  $$ \hat{\theta}(w,y,[u])=(\theta(w,y),A(w,y).[u])$$
  
  Let us also introduce 
  $\psi: \Sigma \times \overline{O} \times \RP^{n} \to \RR$ defined by $\psi(w,y,[u]))=\log(||A(w,y).u||/||u||)$

    Now look at the sequence of measures on $\Sigma \times  \overline{O} \times \RP^{n}$, defined by 
    $\hat{\nu}_k=\frac{1}{\ell(g_k)} \Sigma_{m=1}^{\ell(g_k)} (\hat{\theta}^m)_*\delta(w^{k},u_k)$.
    
    We compute $\int_{\Sigma \times \overline{O} \times \RP^{n}} \psi d\hat{\nu}_k=
    \frac{1}{\ell(g_k)} \Sigma_{m=1}^{\ell(g_k)} \psi(\hat{\theta}^m(w^{(k)},x,[u_k]))$.
    This expression is nothing but  $\frac{1}{\ell(g_k)} \log || A^{(\ell(g_k))}(w^{(k)},x).u_k||$, so that 
    $\int_{\Sigma \times \overline{O} \times \RP^{n}} \psi d\hat{\nu}_k \geq \lambda$.
    
    Because the space $\Sigma \times  \overline{O} \times \RP^{n}$ is compact, we can find a subsequence $(\hat{\nu}_{i_k})$
     converging for the weak-star topology to a probability measure $\hat{\nu}$ on $\Sigma \times  \overline{O} \times \RP^{n}$.
     It is easily checked that $\hat{\nu}$ is $\hat{\theta}$-invariant, and we still have 
     $\int_{\Sigma \times \overline{O} \times \RP^{n}} \psi d\hat{\nu} \geq \lambda$.
     Performing an ergodic decomposition, we get an ergodic, $\hat{\theta}$-invariant measure $\hat{\nu}_e$ on 
     $\Sigma \times  \overline{O} \times \RP^{n}$ satisfying 
     \begin{equation}
      \label{eq.integrale}
      \int_{\Sigma \times \overline{O} \times \RP^{n}} \psi d\hat{\nu}_e \geq \lambda.
     \end{equation}

      We push $\hat{\nu}_e$ forward 
     to an ergodic, $\theta$-invariant measure $\nu_e$ on $\Sigma \times  \overline{O}$.  
     From (\ref{eq.integrale}) and \cite[Prop. 5.1]{ledrappier}, we conclude that the cocycle $A^{(k)}$ admits a Lyapunov exponent 
     which is $\geq \lambda$ (hence positive). It means precisely that $\nu_e$ is partially hyperbolic, and the proof is complete.

  
\end{proof}
 
 Proposition \ref{prop.growth.limitset} uses only an hypothesis involving growth of derivatives. 
When the group $G$ itself has exponential growth, the conclusions of Proposition \ref{prop.growth.limitset}
 can be strengthen in the following way:
 
\begin{corollaire}
  \label{coro.infinite.limit}
  Let $(M,g)$ be a compact Lorentz manifold. 
  Let $G \subset \Iso(M,g)$ be a closed, compactly generated subgroup. If $G$ has exponential growth, then
   the limit set $\Lambda(x)$ is infinite for every $x \in M$. In particular $\dig(x) \geq 3$ for every $x \in M$.
 \end{corollaire}
\begin{proof}
 By Propositions \ref{prop.growth.derivatives} and \ref{prop.growth.limitset}, we now that  $\carg(x) \geq 2$
   for every $x \in M$. We call $c_{min} \in \{2,3,\ldots\} \cup \{ \infty \}$ the minimal value achieved 
    by $x \mapsto \carg(x)$ on $M$. If the minimal value $c_{min}$ is $+ \infty$, we are done. If 
     on the contrary $c_{min}$ is finite, we are going to get a contradiction. To see this, we define:
     $$ K_{min}:=\{ x \in M \ | \ \carg(x)=c_{min}  \}.$$
    By corollary \ref{coro.semicontinuous}, the set $K_{\geq c_{min}+1}$ is open, hence $K_{min}$
     is a compact subset of $M$. 
   
   If $c_{min}=2$ then 
   we have two lightlike Lipschitz directions on $K$, which are preserved by $G$, or an index two subgroup 
   of $G$. Because the subgroup of $\OO(1,n)$ leaving invariant two 
   linearly independent lightlike directions is isomorphic to $\RR \times \OO(n-1)$, we get a $G$-invariant
    (Lipschitz) reduction of the bundle $\hm$ to the group $\RR \times \OO(n-1)$ above $K_{min}$.
  Lemma \ref{lem.restriction}  then provides a coarse embedding $\alpha: G \to \RR \times \OO(n-1)$.
   By Lemma \ref{lem.growth}, this is impossible since $G$ is assumed to have exponential growth, while $\RR \times \OO(n-1)$ has 
    linear growth.  
    It means that we have  $+ \infty > c_{min} \geq 3$.  But the subgroup of $\OO(1,n)$ leaving individually invariant a finite family
     of lightlike directions spanning a subspace of dimension $\geq 3$, is a compact group isomorphic to some 
      group
      $\OO(k)$,  $1 \leq k \leq n-2$. Again, looking at a finite index subgroup of $G$, we have a $G$-invariant 
      reduction of the bundle $\hm$ to the group $\OO(k)$. Lemma \ref{lem.restriction} 
      then provides a coarse embedding $\alpha: G \to \OO(k)$. This is a new contradiction since no noncompact group can 
      be coarsely embedded into a compact one.
\end{proof}

\section{Exponential growth and Killing fields}
\label{sec.exponential.killing}

We have shown in Corollary \ref{coro.infinite.limit} that the existence of a closed subgroup with exponential growth in the isometry
 group of a compact Lorentz manifold $(M^{n+1},g)$ forces the limit set to be infinite at each point.  The aim of the present section
  is to derive the first geometric consequences of this fact. 
  
 Recall that a local Killing field on $M$ is a vector field defined on some open subset $U \subset M$, such that 
  the Lie derivative $L_Xg=0$. In other words, the local flow of $X$ acts isometrically for $g$. In the neighborhood 
  of each point $x \in M$, the algebra of local Killing fields is a finite dimensional Lie algebra, that
   will be denoted $\kiloc(x)$. The isotropy algebra ${\liei}_x$ is the subalgebra of $\kiloc(x)$ comprising
    all local Killing fields vanishing at $x$. 
    
  The main theorem of this section shows that the existence of a big limit set for $\Iso(M,g)$ produces many local Killing fields, at least on a nice {\it open and dense subset} $\mint$,
   called  the integrability locus, to be defined later on.
   
\begin{theoreme}
 \label{thm.exponential.killing}
 Let $(M,g)$ be a compact Lorentz manifold. Assume that the limit set $\Lambda(x)$ of $\Iso(M,g)$ is infinite for every $x \in M$. Then for every $x$ in the integrability locus $\mint$, the isotropy Killing 
  algebra ${\liei}_x$ is isomorphic to $\oo(1,k_x)$, with $k_x \geq 2$.
\end{theoreme}

Observe that by Corollary \ref{coro.infinite.limit}, Theorem \ref{thm.exponential.killing} will hold as soon as
 $\Iso(M,g)$ contains a closed, compactly generated subgroup of exponential growth.

%

The Killing fields appearing in Theorem \ref{thm.exponential.killing} will be obtained  using integrability 
results, which 
were first proved in \cite{gromov}, and that we present below.

\subsection{Canonical Cartan connection, and the generalized curvature map}
\label{sec.generalized}
In all what follows, we will denote by $\lieg_0$ the Lie algebra  $\oun \ltimes \RR^{n+1}$. 
We  consider $(M^{n+1},g)$ a $(n+1)$-dimensional Lorentz manifold.
  Let $\pi: \hm \to M$ denote the  bundle of orthonormal frames on $M$.  This 
 is a principal ${\operatorname O}(1,n)$-bundle over $M$, and it is classical (see \cite{kobanomi}[Chap. IV.2 ])  that  the Levi-Civita connection associated to
  $g$ can be interpreted as an Ehresmann connection $\alpha$ on $\hm$, namely a $\Oun$-equivariant $1$-form with values 
  in the Lie algebra $\oo(1,n)$.
  Let $\theta$ be the soldering 
  form on
   $\hm$, namely the $\RR^{n+1}$-valued $1$-form on $\hm$, which to every $\xi \in T_{\hx}\hm$  associates
   the coordinates of the vector 
   $\pi_*(\xi) \in T_xM$
    in the frame $\hx$.  The sum $\alpha + \theta$ is a $1$-form $\omega: T\hm \to \lieg_0$ called 
    the {\it canonical Cartan 
    connection} associated 
    to $(M,g)$. 
    
    Observe that  for every $\hx \in \hm$,
      $\omega_{\hx}: T_{\hx}\hm \to \lieg_0$ is an isomorphism of vector spaces, and 
        the form $\omega$ is $\Oun$-equivariant (where $\Oun$ acts on 
      $\lieg_0=\oun \ltimes \RR^{n+1}$ via the adjoint action).

   The notion of Riemannian curvature for $g$, as well
  as its higher order covariant derivatives have a counterpart in $\hm$. 
 The {\it curvature} of the Cartan connection $\omega$ is a $2$-form $K$ 
 on $\hm$, with values in $\lieg_0$, defined as follows.  If $X$ and $Y$ are 
 two vector fields on $\hm$, the curvature is given by the relation:
 $$ K(X,Y)=d \omega(X,Y) +[\omega(X),\omega(Y)].$$
 Because at each point $\hx$ of $\hm$, the Cartan connection $\omega$ establishes an isomorphism between $T_{\hx}\hm$ and $\lieg_0$, 
 it follows that
  any $k$-differential form on $\hm$, with values in some vector space ${\mathcal W}$,  can be seen as a map from $\hm$ to 
  $ \Hom( \otimes^k  \lieg_0,{\mathcal W})$.  This remark applies for the curvature form $K$ itself, yielding 
  {\it a curvature map}
  $\kappa: \hm \to {\mathcal W}_0$, where the vector space ${\mathcal W}_0$ is a sub $\Oun$-module 
  of $\Hom(\wedge^2(\RR^{n+1});\lieg_0)$ (the curvature
   is antisymmetric and vanishes when one of its arguments is tangent to the fibers of $\hm$). 
   
   We  now differentiate the map $\kappa$, getting a map $D \kappa : T\hm \to {\mathcal W}_0$.  
  The connection $\omega$ allows to identify $D \kappa$  with
  a map  ${\mathcal D} \kappa : \hm \to {\mathcal W}_1$, where ${\mathcal W}_1=\Hom(\lieg_0,{\mathcal W}_0)$. 
     the $r$th-derivative of the curvature ${\mathcal D}^r \kappa : \hm \to \Hom(  \lieg_0,{\mathcal W}_r)$ (with ${\mathcal W}_r$ defined inductively
     by ${\mathcal W}_r=\Hom(\lieg_0, {\mathcal W}_{r-1})$).  
      The {\it generalized curvature map} of our Lorentz manifold $(M,g)$ 
    is the map  $\kg =(\kappa, {\mathcal D} \kappa, {\mathcal D}^2 \kappa,\ldots,{\mathcal D}^{\dim(\lieg_0)}\kappa)$. 
    The $\Oun$-module $\Hom(\lieg_0, {\mathcal W}_{\dim(\lieg_0)})$ will be rather denoted ${\mathcal W}_{\kg}$ in the sequel.

\subsection{Integrating formal Killing fields}

\label{sec.integration}

\subsubsection{Integrability locus}     
   \label{sec.integrability.locus}
    One defines {\it the integrability locus of $\hm$}, denoted $\hmint$, as 
    the set of points $\hx \in \hm$ at
     which the rank of $ D\kg$ is locally constant. Notice that 
     $ \hmint$ is a $\OO(1,n)$-invariant  open subset 
     of $\hm$. Because the rank of a smooth map can only increase locally, this open subset is dense.
      We define also  $ \mint  \subset  M$, the integrability locus of 
      $M$, as the projection of $\hmint$ on $M$.  This is a dense open subset of $M$.
  
\subsubsection{The integrability theorem}
\label{sec.integrability.theorem}
   Local flows of isometries on $M$ clearly induce local flows on the bundle of orthonormal frames, which moreover preserve $\omega$.
    It follows  that any local Killing field $X$ on $U \subset M$ lifts to a vector field $\hat{X}$ on $\hat{U}:=\pi^{-1}(M)$, satisfying
    $L_{\hat{X}}\omega=0$. 
     Conversely, local vector fields of $\hm$ such that $L_{\hat{X}}\omega=0$, that we will henceforth call
     {\it  $\omega$-Killing fields}, 
     commute with the right $\OO(1,n)$-action on $\hm$.  Hence, they induce local vector fields $X$ on $M$, which are Killing because 
     their local flow maps orthonormal frames to orthonormal frames.  It is easily checked that a $\omega$-vector field
      which is everywhere tangent to the fibers of the bundle $\hm \to M$ must be trivial.  As a consequence, there is a one-to-one 
      correspondence between local $\omega$-Killing fields on $\hm$ and local Killing fields on $M$.  We will use this correspondence 
      all along the paper.
      The same remark holds for local isometries.

   Observe finally that if $\hat{X}$ is a $\omega$-Killing field on $\hm$  (namely $L_{\hat{X}} \omega=0$), then the local flow of $\hat{X}$ preserves 
   $\kg$, hence $\hat{X}$ belongs to $\Ker(D_{\hx}\kg)$ at each point. 
   The integrability theorem below says that the converse is 
   true on the set $\hmint$.

   \begin{theoreme}[Integrability theorem]
    \label{thm.integrabilite}
    Let $(M,g)$ be a Lorentz manifold. Let $ \mint \subset M$ denote the integrability locus. 
      { For every $\hx \in \hmint$, and every $\xi \in \Ker(D_{\hx}\kg)$, there exists a local $\omega$-Killing field $\hat{X}$
     around $\hx$ such that 
     $\hat{X}(\hx)=\xi$.}

    \end{theoreme}
  
  An akin integrability result for Killing fields of finite order  first appeared in the seminal paper \cite{gromov}. 
   The results were recast in the framework of real analytic Cartan geometry in \cite{melnick}, and  \cite{vincent} provides
   an alterative approach for smooth Cartan geometries, 
  leading to the statement of Theorem \ref{thm.integrabilite} (see also Annex A of \cite{frances.open}, which elaborates 
  slightly on the statement proved in \cite{vincent}).  
  
  \subsubsection{Connected components of $\mint$, and ``analytic continuation'' of Killing fields}
  \label{sec.analyticcont}
  A first important consequence of Theorem \ref{thm.integrabilite} is that on the set $\mint$, Killing fields 
   have a particularly nice behavior.  To see that, let us recall that 
   for any $x \in M$, there is a good notion of {\it local Killing algebra at $x$}. 
  Indeed, there exists  $U$ a small enough
   neighborhood of $x$, such that for every neighborhood $V \subset U$ containing $x$, any  Killing field on $V$ will be the
    restriction of a Killing field of $U$.  We then call $\kiloc(x)$ the (abstract) Lie algebra $\kil(U)$ of all Killing 
    fields defined on $U$.  Theorem \ref{thm.integrabilite} shows that if $\calm$ is a connected component of $\mint$,
     then the dimension of $\kiloc(x)$ does not depend of $x \in \calm$, because this
     dimension is just the corank of $\kg$
      on $\hat{\calm}$. As a consequence, the local Killing fields on $\calm$ 
      behave much like Killing fields of a real analytic
      metric. In particular, given a Killing field $X$ defined on some open set $U \subset \calm$, and given a  path $\gamma$ starting 
      at a point of $U$, one can perform the ``analytic continuation'' of $X$ along $\gamma$. It follows that if $U$ is a $1$-connected
      open subet of $\calm$, and if $X$ is a Killing field defined on $V \subset U$, then there exists a Killing field defined 
      on $U$, whose restriction to $V$ is $X$.  
      Those nice properties will often be used implicitely in the sequel. 
  
  \subsubsection{Isotropy algebra and stabilizers of the generalized curvature}
  Another  corollary of  Theorem \ref{thm.integrabilite}, which will be of particular interest to 
  prove Theorem \ref{thm.exponential.killing}, is the following:
  
  \begin{corollaire}
  \label{coro.stabilizer}
   For every point $x \in \mint$, the isotropy algebra $\liei_x$ is isomorphic to the Lie algebra $\lies_x$
    of the stabilizer of $\kg(\hx)$ in $\OO(1,n)$, for any $\hx \in \hm$ in the fiber of $x$.
  \end{corollaire}

  \begin{proof}
   Every element $X \in \liei_x$ defines a local Killing field in a neighborhood of $x$, that can be lifted to
    a $\omega$-Killing field $\hat{X}$ on a neighborhood of $\hx$. Observe that  $\hat{X}$ is tangent to the fiber of
     $\hx$ since $X$ vanishes at $x$.  The map $\rho: X \mapsto \omega_{\hx}(\hat{X}(\hx))$ yields a Lie algebra morphism
      from $\liei_x$ to $\lies_x$.  The map is injective since two local $\omega$-Killing fields coinciding 
       at $\hx$ must coincide on an open subset around $\hx$.  The map is onto, because if $Y \in \lies_x$,
        and if $\xi =\omega_{\hx}^{-1}(Y)$, then $\xi \in \Ker(D_{\hx}\kg)$. Theorem \ref{thm.integrabilite} then
         provides a local Killing field around $x$, such that $\hat{X}(\hx))=\xi$. In particular $X(x)=0$ so that 
         $X \in \liei_x$.
  \end{proof}

\subsection{Infinite limit set implies semisimple isotropy} 
\label{sec.infinite.semisimple}
We can now proceed to the proof of Theorem \ref{thm.exponential.killing}. In light of Corollary \ref{coro.stabilizer},
 it is enough to show that for every $\hx \in \hmint$, the stabilizer of $\kg(\hx)$ in $\OO(1,n)$
  contains a subgroup isomorphic to $\SO^o(1,k)$, for $k \geq 2$.  
  
We recall the bounded section $\sigma: M \to \hm$, thanks to which we built the derivative cocycle
$\cald: M \times \Iso(M,g) \to \OO(1,n) $, with $\cald(x,f)=\cald_xf$ 
(see Section \ref{sec.derivative.cocycle}). 
We fix $x \in M$ in the sequel.
Because $\sigma(M)$ is included in a compact subset of $\hm$,
 we have that $\kg(\sigma(M))$ is also contained in a compact subset ${\mathcal K}$ of ${\mathcal W}_{\kg}$.
 Let $f \in \Iso(M,g)$. By the definition of $\dd_xf$, one has the relation 
 $$ f(\sigma(x)).(\dd_xf)^{-1}=\sigma(f(x)).$$
 This yields
 $$ \kg(\sigma(f(x)))=\dd_xf.\kg(\sigma(x)),$$
 and we infer that $\dd_xf.\kg(\sigma(x))$ belongs to the compact set ${\mathcal K}$ for every $f \in \Iso(M,g)$.
 
 This leads to the following general notion of stability.
 Let  $n \geq 2$, and $\rho: \OO(1,n) \to \GL(V)$ be a finite dimensional representation. 
 If ${\mathcal G}$ is a subset of $\OO(1,n)$, and
  $v \in V$ is a vector, we say that $v$ is stable under ${\mathcal G}$, 
  if $\rho({\mathcal G}).v$ is a bounded subset of $V$.
  The previous discussion shows that $\kg(\sigma(x)) \in {\mathcal W}_{\kg}$ is stable under 
   the set $\dd_x(\Iso(M,g))$. 
 
 In full generality, we wonder if  a vector  $v \in V$ is stable under 
 a set ${\mathcal G} \subset \OO(1,n)$ having a big limit set $\Lambda_{\mathcal G}$ in $\partial \HH^n$ (see 
 Section \ref{sec.limit.coarse}), 
   then $v$ is fixed by a big subgroup of $\OO(1,n)$. 
   
 To make things a little bit precise, we see $ \partial \HH^n$ as the set of lightlike directions 
 in $\RR^{1,n}$, and we introduce the  {\it linear hull}  of the limit set $\Lambda_{\mathcal G}$, denoted 
 $E_{\Lambda_{\mathcal G}}$,  as 
 the linear
  span of $\Lambda_{\mathcal G}$ in $\RR^{1,n}$.  The dimension of $E_{\Lambda_{\mathcal G}}$ will be denoted 
  $d_{\Lambda_{\mathcal G}}$. We can now state:

\begin{proposition}[Big limit set implies big stabilizer]
  \label{prop.stability}
  Let $\rho : \OO(1,n) \to \GL(V)$ be a  finite dimensional representation.
  Let ${\mathcal G} \subset \OO(1,n)$ such that the limit set $\Lambda_{\mathcal G} \subset \partial \HH^n$
   satisfies the property $d_{\Lambda_{\mathcal G}} \geq 3$.  
   Then for every vector $v_0 \in V$ which is stable under ${\mathcal G}$, the stabilizer of $v_0$ in $\OO(1,n)$
    contains a subgroup isomorphic to
   $\SO^o(1,d_{\Lambda_{\mathcal G}}-1)$. 
   \end{proposition}  
 
If we take this proposition for granted, then Theorem \ref{thm.exponential.killing} follows easily.
 Indeed, the asumption  of Theorem \ref{thm.exponential.killing} implies that $\dig(x) \geq 3$. Proposition \ref{prop.stability}, applied for
   the representation $\rho: \OO(1,n) \to \GL({\mathcal W}_{\kg})$, and ${\mathcal G}=\dd_x(\Iso(M,g))$
    then says that the stabilizer of $\kg(\sigma(x))$ in $\OO(1,n)$ contains a subgroup isomorphic to 
    $\SO^o(1,\dig(x)-1)$.  We then conclude thanks to Corollary \ref{coro.stabilizer}.

\subsection{Proof of Proposition \ref{prop.stability}}
The remaining of this section is devoted to the proof of Proposition \ref{prop.stability}.
 \subsubsection{First easy reduction, and dynamical properties}
 \label{sec.simplifications}
 By hypothesis, $n\geq 2$, hence $\oo(1,n)$ is simple. The representation $\rho: \OO(1,n) \to \GL(V)$  is thus a direct sum of irreducible representations.  It is enough to prove proposition \ref{prop.stability}
  for irreducible representations $\rho$. To avoid cumbersome notations, we will denote in the following $g.v$
   instead of $\rho(g).v$ (the image of the vector $v \in V$ under the linear transformation $\rho(g)$).
 
 Let us consider a subset ${\mathcal G} \subset \OO(1,n)$, and recall the notion of limit set 
 $\lambda_{\mathcal G}$ introduced in Section \ref{sec.limit.coarse}.
 If ${\mathcal K} \subset \OO(1,n)$ is a compact subset, and if $k: {\mathcal G} \to {\mathcal K}$, we can define a new set
  ${\mathcal G}^{'}:=\{k(g)g \ | \ g \in {\mathcal G}\}$.  If for some point $\nu \in \HH^n$, a sequence $(g_k)$ of $\OO(1,n)$
   satisfies $g_k^{-1}. \nu \to p$, where $p \in \partial \HH^n$, then for 
   any compact subset $C \subset \HH^n$, we have $g_k^{-1}.C \to p$ 
    (the limit has to be understood for the Hausdorf distance between compact subsets of $\overline{\HH}^n$).
  Thus it is clear that ${\mathcal G}$ and ${\mathcal G}^{'}$ have the same limit set, and if   
  $\rho : \OO(1,n) \to \GL(V)$ is a  finite dimensional representation, then stable vectors for ${\mathcal G}$ 
  and ${\mathcal G}^{'}$ coincide. 
  It follows from Iwasawa decomposition  $\OO(1,n)=KAN$ 
  (see the proof of Proposition \ref{prop.growth.derivatives}) that we may assume, 
  to prove Proposition \ref{prop.stability}, that ${\mathcal G} \subset AN$. We will do this in the sequel.
  
  We recall that in the upper-half 
  space model $\HH^n=\RR_+^* \times \RR^{n-1}$, elements of $A$ act as homothetic transformations
   $a^s: (t,x) \mapsto (e^{s}t,e^{s}x)$, $s \in \RR$. The group $N$ is abelian, isomorphic to $\RR^{n-1}$
    and it acts as $n(v): (t,x) \mapsto (t,x+v)$, $v \in \RR^{n-1}$.  The dynamics of $a^s$ on $\overline{\HH}^n$
     has two distinct fixed points $p^+$ and $p^-$ on $\partial \HH^n$, and for every  $p \in \HH^n$,
      $\lim_{s \to + \infty} a^s.p=p^+$ (resp. $\lim_{s \to - \infty} a^s.p=p^-$).  The group $N$ fixes $p^+$
       and acts simply transitively on $\partial \HH^n \setminus \{ p^+ \}$.

  We can now formulate the following dynamical lemma.
  
  \begin{lemme}
  \label{lem.dynamical}
 Let us consider an unbounded sequence  $g_k=a^{s_k}n_k$ of $AN \subset \OO(1,n)$, and let us pick a point $o \in \HH^n$.  
 After considering maybe a subsequence, we are in one of the following four cases.
 
 \begin{enumerate}
  \item {The sequence $(s_k)$ converges to $- \infty$.  Then $g_k^{-1}.o \to p^+$.}
  \item { The sequence $(s_k)$ converges in $\RR$ and  $n_k \to  \infty$ in $\NN$.  Then  $g_k^{-1}.o \to p^{+}$.}
  \item { The sequence $(s_k)$ tends to $+ \infty$ and $n_k \to  \infty$ in $N$.  Then $g_k^{-1}.o$ tends to $p^+$.}
  \item{The sequence $(s_k)$ tends to $+ \infty$ and   $(n_k)$ tends to $n_{\infty} \in N$. Then $g_k^{-1}.o$ tends to $p=n_{\infty}^{-1}.p^- \not = p^+$.}

 \end{enumerate}

 \end{lemme}

 \subsubsection{Proof of Proposition \ref{prop.stability} for $n=2$}
\label{sec.dim2}
  We first do the proof of Proposition \ref{prop.stability} for a representation $\rho: \OO(1,2) \to \GL(V)$.  The Lie algebras 
  $\oo(1,2)$ and ${\mathfrak{sl}(2, \RR)}$ are isomorphic, and there is a 2-fold covering $\pi: \SL(2,\RR) \to \SO^o(1,2)$.
   Thus there exists a representation $\rho': \SL(2,\RR) \to \GL(V)$ such that $\rho'=\rho \circ \pi$. 
   If $V$ is $m$-dimensional, then irreducible finite 
   representations of 
   $\SL(2,\RR)$ occur from the natural action of $\SL(2,\RR)$ on homogeneous polynomials of degree $m-1$ in two variables. Here $m=2l+1$ 
   must be odd for $\rho'$ to induce a representation of $\SO^o(1,2)$.
   
   At the level of Lie algebras, let us introduce $H=\left( \begin{array}{cc}
      1 & 0 \\
      0 & -1
     \end{array} \right) , \ \ E= \left( \begin{array}{cc}
      0 & 1 \\
      0 & 0
     \end{array} \right).$
     
     We assume in the following that $l \geq 1$ (namely the representation is not $1$-dimensional).
     If $\overline{\rho} : {\mathfrak{sl}(2,\RR)} \to {\mathfrak{gl}(V)}$ is the induced representation, then 
     there
 is a suitable basis $e_1,\ldots,e_{2l+1}$ of $V$ where 
  
$$\overline{\rho}(H)=\left( \begin{array}{ccccc} 2l & & & & \\
 & 2l-2 & & & \\ & & \ddots & & \\
  & & & -2l+2 & \\
   & & & &-2l \end{array} \right), \ \ \  \overline{\rho}(E)=\left( \begin{array}{ccccc} 0 & 1 &0 & & \\
   & \ddots & 2&\ddots & \\ &  & \ddots & \ddots&0 \\
  & & & \ddots & 2l\\
   & & & &0 \end{array} \right).$$
   
   After identifying ${\mathfrak{sl}(2,\RR)} $ and $\oo(1,2)$ under the adjoint action, we see that if $n^t=e^{tE}$, then  
   
   $${\rho}(n^t)=\left( \begin{array}{ccccc} 1 & a_{12}t &a_{13}t^2 & \ldots & a_{1m}t^{m-1}\\
   0& 1 & a_{23}t&\ldots &a_{2m}t^{m-2} \\
   0& 0 & \ddots & \vdots& \vdots \\
  0&0 &0 &1  & a_{m-1,m}t\\
   0&0 &0 &0 &1 \end{array} \right), $$
  for  positive coefficients $a_{i,j}$ where  $1 \leq i \leq j \leq m=2l+1$.
 We also see that  the  vectors $v \in V$ which are stable under $a^s:=e^{sH}$, $s \geq 0$, are vectors in 
 $V^-={\operatorname{Span}}(e_{l+1}, \ldots , e_{2l+1})$. 
     
     
 
 We are now ready to prove the following lemma, which is a reformulation of Proposition 
 \ref{prop.stability} for $n=2$.
 \begin{lemme}
 \label{lem.stability.dim2}
  Let $\rho : \OO(1,2) \to \GL(V)$ be a finite dimensional representation.
  Let $o \in \HH^2$ be a point, and let 
   $(g_k), (g_k')$ and $(g_k'')$ three sequences in $\OO(1,2)$.  Assume that there exists
    three pairwise distinct points $p,p',p''$ in $\partial \HH^2$ such that $g_k^{-1}.o \to p$, $(g_k')^{-1}.o \to q$
     and $(g_k'')^{-1}.o \to r$.  Then any vector $v \in V$ which is stable for $(g_k), (g_k')$ and $(g_k'')$
      is actually $\SO^o(1,2)$-invariant.
 \end{lemme}

 \begin{proof}
  We first assume that $\rho$ is irreducible, and $V$ is not $1$-dimensional.
  We observe that the conclusions of the Lemma are unaffected if we conjugate the three sequences $(g_k)$, $(g_k')$ and 
  $(g_k'')$ by an element $h \in \OO(1,2)$ and replace $p,p',p''$ by $h.p,h.p',h.p''$.
   Therefore, because the action of $\OO(1,2)$ is transitive on triplets in $\partial \HH^2$, we may assume that $p=p^-$, $p'=p^+$
   and $p'' \in \partial \HH^2$ is different
       from $p^+$ and $p^-$.   As explained in Section \ref{sec.simplifications}, after having performed such a conjugacy, we 
       still may left-multiply our sequences by sequences with values in a maximal compact group 
       $K \subset \OO(1,2)$ without affecting the conclusions.
        Hence we will also assume that $(g_k)$, $(g_k')$ and 
       $(g_k'')$ are sequences in $AN \subset \SO^o(1,2)$.

 We thus write  $g_k=a^{s_k}n^{t_k}$, $g_k'=a^{s_k'}n^{t_k'}$ $ g_k''=a^{s_k''}n^{t_k''}$.
 
 Our first assumption is that $g_k^{-1}.o \to p^+$.  We have seen in Lemma \ref{lem.dynamical} that it 
 can happen in three different ways.
 
 \begin{enumerate}
  \item {{First case:} The sequence $(s_k)$ is bounded in $\RR$ (and $|t_k| \to \infty$).  Then, $v$ is stable 
  under $e^{s_kH}e^{t_kE}$ if and only if it is stable under $e^{t_kE}$.  This only occurs if $v \in \RR.e_1$.}
  \item{{Second case:} The sequence $(s_k)$ tends to $- \infty$.  Then if $v$ is stable under $e^{s_kH}e^{t_kE}$, the coordinates
   of $e^{t_kE}.v$ along $\operatorname{Vect}(e_{l+2}, \ldots, e_{2l+1})$ must tend to $0$.  
   In particular, $v_{2l+1}=0$.  Then 
   $v_{2l}+a_{2l,2l+1}t_kv_{2l+1} \to 0$ as $k \to \infty$, hence $v_{2l}=0$.  We proceed in the same way to get $v_{2l-1}=\ldots=v_{l+2}=0$, and 
   $v \in \operatorname{Vect}(e_1,\ldots, e_{l+1})$.}
  \item{{Third case:} The sequence $(s_k)$ tends to $+ \infty$. Then $g_k^{-1}.o \to p^+$ means that $|t_k| \to + \infty$.
   If the vector $v$ is stable under $e^{s_kH}e^{t_kE}$, then necessarily, the coordinates of $n^{t_k}.v$ along 
   $\operatorname{Vect}(e_1,\ldots,e_l)$ tend to $0$.  Looking at the coordinate along $e_1$ we get that
   $$ v_1+a_{12}t_kv_2+\ldots+a_{1,2l+1}t_k^{2l}v_{2l+1} \to 0.$$
   because the coefficient $a_{ij}$ are positive, this only occurs when $v_1=\ldots,v_{2l+1}=0$, namely $v=0$.} 
 \end{enumerate}

 We now  use our hypothesis that $(g_k')^{-1}.o \to p^{-}$.  By Lemma \ref{lem.dynamical}, this happens exactly
 when $s_k' \to + \infty$ and $t_k' \to 0$.
  The vector $v$ is stable under $a^{s_k'}n^{t_k'}$ only when $v$ belongs to $V^-=\operatorname{Vect}(e_{l+1}, \ldots, e_{2l+1})$.
  Together with the conclusions of the three possible cases above, we end up with  $v \in \RR.e_{l+1}$.
  
  We write $v=\lambda e_{l+1}$, and finally use our third assumption: $(g_k'')^{-1}.o \to p''$, with $p'' \not = p^+$, $p'' \not = p^-$.
  This assumption is equivalent to $s_k'' \to + \infty$ and $t_k'' \to t_{\infty}$. Observe that 
  $t_{\infty} \not = 0$ since $p'' \not = p^-$.
  Again, $v$ can be stable under $a^{s_k''}n^{t_k''}$ only when the coordinates of $n^{t_{\infty}}.v$ along the vectors 
  $e_1, \ldots, e_l$ are zero.  But the coordinate of $n^{t_{\infty}}.e_{l+1}$ along $e_1$ is $a_{1,l+1}t_{\infty}^{l}$, which 
  is nonzero since $a_{1,l+1}>0$ and $t_{\infty} \not = 0$. This third stability condition 
  thus forces $\lambda=0$, namely $v=0$.
  
  To conclude te proof, we now write the respresentation $\rho$ as a direct sum of irreducible representations.
   By what we showed above, a stable vector $v$ has only nonzero components on irreducible factors of dimension $1$.
    Dimension $1$ representations of $\SO^o(1,2)$ being trivial, we infer that $v$ is $\SO^o(1,2)$-invariant, and 
    Lemma \ref{lem.stability.dim2} is proved.
 \end{proof}

Lemma \ref{lem.stability.dim2} has the following algebraic consequence:

\begin{corollaire}
\label{coro.algebrique}
 Let $\rho: \OO(1,2) \to \GL(V)$ be a  finite dimensional representation.  Let $V^- \subset V$ be the stable 
 subspace of $\{ a^s \}_{s \geq 0}$.   Let $n_1=n^{t_1}, n_2=n^{t_2}$ and $n_3=n^{t_3}$ be three elements of $N$, that we assume
  to be pairwise distinct.  Then any vector $v \in n_1.V^- \cap n_2.V^- \cap n_3.V^-$ is fixed by $\SO^o(1,2)$. 
\end{corollaire}

\begin{proof}
 Let us consider the three sequences $g_k:=a^kn_1^{-1}$, $g_k':=a^k n_2^{-1}$ and $g_k''=a^k n_3^{-1}$. 
 Any vector  $v$ belonging to  $n_1.V^- \cap n_2.V^- \cap n_3.V^-$  is stable  
  under $(g_k), (g_k')$ and $(g_k'')$.  On the other hand $g_k^{-1}.o=n_1 a^{-k}.o$ hence 
    $g_k^{-1}.o \to n_1.p^{-}$.  In the same way, $(g_k')^{-1}.o \to n_2.p^{-}$ and $(g_k'')^{-1}.o \to n_3.p^{-}$.
     Because $n_1,n_2,n_3$ are pairwise distinct, the points  $n_1.p^-,n_2.p^-$ and $n_3.p^-$ are pairwise distinct too.  
    Lemma \ref{lem.stability.dim2}  ensures that $v$ is fixed by $\SO^o(1,2)$.
\end{proof}

\subsubsection{Proof of Proposition \ref{prop.stability} in any dimension}

Let $n \geq 2$, and let $\rho: \OO(1,n) \to \GL(V)$ be a finite dimensional representation. We consider 
${\mathcal G} \subset \OO(1,n)$. 
The simplifications mentioned in Section \ref{sec.simplifications} are still in force.  In particular we still assume that
${\mathcal G} \subset AN$.  Let us recall that we identify $\partial \HH^n$ with the set of lightlike directions
in $\RR^{1,n}$. The {\it linear hull} of a set $S \subset \partial \HH^n$ is then the linear span of $S$ in 
$\RR^{1,n}$. If $E$ is this linear hull, then $E$ has Lorentz signature as soon as $S$ contains at least two points.
 We then denote by $\SO^o(E)$ the subgroup of $\SO^o(1,n)$ which leaves $E$ invariant and acts trivially on 
 $E^{\perp}$.

\begin{lemme}
 \label{lem.cercle.dimn}
 Let $p$, $q$ and $r$ be three pairwise distinct points in $\Lambda_{\mathcal G}$, and let 
   $E \subset \RR^{1,n}$ be their linear hull.
 Then any vector $v \in V$ which is stable for ${\mathcal G}$  is fixed by $SO^o(E) $.
\end{lemme}

\begin{proof}
 As already seen, we may  conjugate ${\mathcal G}$ into $\OO(1,n)$ to prove the lemma.  Hence, we may assume that $\Delta=\Delta_0$
  and $p, q, r$ different from $p^+$.  Our assumptions imply the existence of $(g_k), (g_k')$ and $(g_k'')$ such that
   $g_k^{-1}.o \to p$,  $(g_k')^{-1}.o \to q$,  $(g_k'')^{-1}.o \to r$.  
   We write $g_k=a^{s_k}n_k$, $g_k'=a^{s_k'}n_k'$ and $g_k''=a^{s_k''}n_k''$.
    We must be in case $3$ of Lemma \ref{lem.dynamical}, namely the three sequences $s_k,s_k',s_k''$  tend to $ + \infty$ and 
    $n_k,n_k',n_k''$
     tend  to $\nu, \nu', \nu''$ respectively, with  $p=\nu.p^{-}$, $q=\nu'.p^{-}$ and $r=\nu''.p^-$.
   Actually, because we assume that  $p, q, r$    belong to $S_{\Delta_0}$, we get that  $\nu,\nu'$ and $\nu''$
    belong to $N \cap S_{\Delta_0}$. Stability of $v$ under $(g_k), (g_k')$ and $(g_k'')$ implies that $v$ belongs to
    $\nu^{-1}.V^- \cap (\nu')^{-1}.V^- \cap (\nu'')^{-1}.V^-$.  By 
    Corollary \ref{coro.algebrique}, $v$ must be fixed by $\SO^o(E)$.
\end{proof}

We are now in position to prove Proposition \ref{prop.stability} by induction on the integer $d_{\mathcal G}$.  
Lemma \ref{lem.cercle.dimn} settles the case $d_{\mathcal G}=3$.  We assume that we proved Proposition 
\ref{prop.stability} for all subsets ${\mathcal G}' \subset \OO(1,n)$ satisfying $3 \leq d_{{\mathcal G}'} 
\leq d_{{\mathcal G}}-1$.
We pick  $p_1, \ldots, p_{d_{\mathcal G}+1}$ pairwise distinct points
in $\Lambda_{\mathcal{G}}$. For each index $1 \leq i \leq d_{\mathcal G}+1$, there exists by assumption 
a sequence $(g_k^{(i)})$ in ${\mathcal G}$,
   such that given $o \in \HH^n$, we have $(g_k^{(i)})^{-1}.o \to p_i$. The linear hull
    of $\{p_1,\ldots,p_{d_{\mathcal G}}\}$ is a subspace $E \subset E_{\Lambda_{\mathcal G}}$, 
    and the linear hull of $\{p_{d_{\mathcal G}-1},p_{d_{\mathcal G}},p_{d_{\mathcal G}+1}\}$
     is a subspace $F \subset E_{\Lambda_{\mathcal G}}$.  We can apply our induction hypothesis to the subset
     ${\mathcal G}'$
     comprising all elements $g_k^{(i)}$ for $1 \leq i \leq d_{\mathcal G}$ and $k \in \NN$, thus 
     obtaining that the vector $v$, which
     is obviously stable by ${\mathcal G}'$ is fixed by $\SO^o(E)$.  We also apply  the induction hypothesis
      to the set ${\mathcal G}''$ comprising the elements 
      $g_k^{(d_{\mathcal G}-1)}, g_k^{(d_{\mathcal G})}, g_k^{(d_{\mathcal G}+1)}$, $k \in \NN$, and we get that $v$
      is fixed by $\SO^o(F)$.  Since it is readily checked that $\SO^o(F)$ and $\SO^o(E)$ generate
     $\SO^o(E_{\Lambda_{\mathcal G}})$, we have finally proved that $v$ is fixed by 
     $\SO^o(E_{\Lambda_{\mathcal G}})$, which is indeed isomorphic to $\SO^o(1,d_{\mathcal G}-1)$. 
      




\section{Invariant locally homogeneous Lorentz submanifolds}
\label{sec.reduc}
We say that a Lorentz manifold $(M^{n+1},g)$ is {\it locally homogeneous}, when $M$ consists  in a single 
$\kiloc$-orbit (see Section \ref{sec.kilocandisloc} below for the definition). It is plain that for such manifolds, $\mint=M$. Also, if $x$ and $y$ are two points of 
 $M$, the isotropy algebras $\liei_x$ and $\liei_y$ are  conjugated in $\oo(1,n)$, because there exists a local
  isometry $f: U \to V$ such that $f(x)=y$.  We say that a locally homogeneous Lorentz manifold
   $(M^{n+1},g)$ has {\it semisimple isotropy}, if for some (hence any) $x \in M$, the isotropy algebra $\liei_x$
    contains a subalgebra isomorphic to $\oo(1,d)$, for $d \geq 2$.  It actually follows from 
    Theorem \ref{thm.exponential.killing}, that if  $(M^{n+1},g)$ is a compact locally homogenous Lorentz manifold,
     and if $\Iso(M,g)$ contains a compactly generated, closed subgroup with exponential growth, then 
     $(M^{n+1},g)$ must have semisimple isotropy.  The aim of this section is to explain that when proving 
     Theorem \ref{thm.main},  we may 
      actually make this hypothesis of local homogeneity without losing any generality. This is the content of the 
      following theorem.

\begin{theoreme}
 \label{thm.reduction}
 Let $(M,g)$ be a compact Lorentz manifold. Assume that $\Iso(M,g)$ contains a closed, compactly generated 
 subgroup 
 $G$ having exponential growth. Then there exists a finite index subgroup $\Iso'(M,g) \subset \Iso(M,g)$ 
 and 
 a compact  Lorentz submanifold $\Sigma \subset M$, which is  locally homogeneous with semisimple isotropy,
  preserved by 
   $\Iso'(M,g)$, and such that the restriction morphism 
    $\Iso'(M,g) \to \Iso(\Sigma,g_{|\Sigma})$ is one-to-one and proper.
\end{theoreme}

  \subsection{Regular locus,  and the structure of $\isloc$-orbits}
  We have seen in Section \ref{sec.integrability.locus} the existence of a  dense 
  open subset $\mint \subset M$, the integrability locus of $M$, where the behavior 
   of the Killing fields is aspecially nice (see also Section \ref{sec.analyticcont}).  We are now going to see 
   that there exists a ``regular'' subset $\mreg \subset \mint$, which is still open and dense, where the behavior 
   of the $\kiloc$-orbit is also nice. This property was first observed by M. Gromov in \cite{gromov}.
   
  \subsubsection{$\kiloc$ and $\isloc$-orbits}
  \label{sec.kilocandisloc}
  Let us recall that if $x \in M$, the $\kiloc$-orbit of $x$ is the set of points $y \in M$ that can be 
  reached from $x$ 
   by flowing along (finitely many)  successive local Killing fields.
   We will be also interested by the $\isloc$-orbit of $x$, namely the set of points $y \in M$ for which there exists a 
    local isometry 
  $f:U\subset M \to V \subset M$ such that $y=f(x)$.  
  
  We will also, for $\hx \in \hm$, define the $\kiloc$-orbit of $\hx$ as the set of points $\hy \in \hm$ that can be reached
   by flowing along (finitely many)  successive local $\omega$-Killing fields.  
   It is pretty clear from the discution at the begining of Section \ref{sec.integrability.theorem} 
   that $\kiloc$-orbits in $M$ are exactly the projections of $\kiloc$-orbits on $\hm$.


   \subsubsection{Structure of $\isloc$-orbits in the regular locus}
   \label{sec.structure}
   
  Theorem \ref{thm.integrabilite}, together with the nice structure 
  for orbits of algebraic group actions, ensures that
   if we restrict our attention to a smaller (but still dense) open subset of $\mint$, that will be called the regular locus of $M$,
    $\kiloc$-orbits 
    define a simple foliation. By a simple foliation, we mean a foliation, such that each point is contained in
    a transversal meeting each leaf at most once. Those nice properties of $\kiloc$-orbits were first
    proved in Gromov's paper \cite{gromov}.
  
  \begin{theoreme}[Structure of $\isloc$-orbits, see \cite{gromov}]
  \label{thm.structure}
   Let $(M,g)$ be a Lorentz manifold. There exists a dense open set  $ \mreg$ called the regular locus of $M$, 
   satisfying
   $\mreg \subset \mint \subset M$, and having the following properties:
   \begin{enumerate}
    \item The set $\mreg$ is $\Iso(M,g)$-invariant and saturated in $\kiloc$-orbits.
    \item There exists a smooth manifold $Y$, as well as a smooth map $\okg: \mreg \to Y$ with
    locally constant rank, such that $\kiloc$-orbits of $\mreg$  coincide with the connected components 
   of the  fibers of $\okg$.
   \item The $\kiloc$-orbits and $\isloc$-orbits of $\mreg$ are closed in $\mreg$, and 
    $\kiloc$-orbits define a simple 
    foliation on any connected component of $\mreg$.
   \end{enumerate}

  \end{theoreme}
  
  \begin{proof}
  For the reader's convenience, we recall the proof of the theorem.
  The main ingredient is a classical result of M. Rosenlicht, about orbits of algebraic group actions 
  (see \cite{rosenlicht}).
   
   \begin{theoreme}
    \label{thm.rosenlicht}
    Let $H$ be an $\RR$-algebraic group, acting on a real algebraic variety ${\bf X}$.  Then there exists a stratification 
    into $H$-invariant Zariski closed sets $F_0 \subsetneq F_1 \subsetneq \ldots \subsetneq F_m={\bf X}$, and for each 
    $0 \leq i \leq m$, a variety $Y_i$ and a regular map $\psi_i: F_i \setminus F_{i-1} \to Y_i$, such that the fibers of $\psi_i$
     coincide with the $H$-orbits on $F_i \setminus F_{i-1}$.
   \end{theoreme}
In the statement we put $F_{-1}=\emptyset$.  The field  $\RR(X)^H$  of $H$-invariant rationnal functions 
is generated by a finite number of 
function $f_1, \ldots, f_r$.  The theorem follows from the fact that outside a Zariski closed  set $F$ 
(the set $F_{m-1}$ in the statement above), $f_1, \ldots, f_r$ separate the orbits, so that we can put $Y=Y_{m-1}=\RR^r$, and
 $\psi=\psi_{m-1}=(f_1,\ldots,f_r)$.  One then applies the same result to $F_{m-1}$ and so on.
 
For every $0 \leq i \leq m$, we call $\hat{\Omega}_i$ the interior of the set $(\kg)^{-1}(F_i \setminus F_{i-1})$.
 One checks that $\bigcup_{i=0}^m {\hat{\Omega}}_i$ is open and dense in $\hm$, and thus so is 
 $\hat{\Omega}:=\hmint \cap \left( \bigcup_{i=0}^m {\hat{\Omega}}_i \right)$.
 Let us call $Y=\bigsqcup_{i=0}^m Y_i$, and define $\hkg: \hat{\Omega} \to Y$, by $\hkg=\psi_i \circ \kg$ on 
 ${\hat{\Omega}}_i \cap \mint$.  Since $\hkg$ is $\Oun$-invariant, it induces a smooth map
  $\okg: \Omega \to Y$, where $\Omega$ is the projection of $\hat{\Omega}$ on $M$.  We consider 
  the open set $\mreg \subset \Omega \subset \mint$ where the rank of $\okg$ is locally constant.  We thus obtain
     a dense open set called {\it the regular locus of $M$}.
  
  The theorem is now a direct consequence of Theorems \ref{thm.integrabilite} and \ref{thm.rosenlicht}.
   Invariance of $\kg$ by local isometries show that  $\mreg$ is
   $\Isloc$-invariant and  saturated by $\kiloc$-orbits. From Theorem \ref{thm.integrabilite}, we 
   infer that  if $\hx \in \hmint$, the $\kiloc$-orbit of $\hx$ 
     coincides with the connected component of 
    $(\kg)^{-1}(w) \cap \hmint$ containing $\hx$. It follows that if $x \in \mreg$, the $\kiloc$-orbit of $x$ coincides with 
     the connected component of the fiber $(\okg)^{-1}(\okg(x))$ containing $x$.
 
 Because  $\okg$ has locally constant rank on $\mreg$, the fibers of $\okg$ are  submanifolds of $\mreg$, and are closed in $\mreg$.
  The same property holds for 
 $\kiloc$-orbits (resp. the $\isloc$-orbits) which are connected components (resp. unions of connected components)
 of those  submanifolds.  
  
  \end{proof}
  
\subsection{Exponential growth and existence of compact $\kiloc$-orbits}
 \label{sec.compactorbits}
 
 We have seen in Corollary \ref{coro.infinite.limit} that the presence of a subgroup of exponential growth in 
 $\Iso(M,g)$ forces the limit set $\Lambda(x)$ to be infinite for every $x \in M$.  We are now going to see that such
  a property forces the existence of {\it compact} $\kiloc$-orbits.
  
 \begin{proposition}
  \label{prop.compact.orbits}
  Let $(M^{n+1},g)$ be a compact Lorentz manifold.
   Assume that the limit set $\Lambda(x)$ is infinite at each  $x \in M$.
   Then there exists a compact, Lorentz $\kiloc$-orbit $\Sigma$ contained in $\mreg$. 
    \end{proposition}
 The proof will show that there are actually infinitely such Lorentz compact $\kiloc$-orbits.
 
 \begin{proof}
  
 Because of our asumption $\carg(x)=\infty$, 
     Theorem \ref{thm.exponential.killing} shows that 
    for every $x \in \mint$,  the isotropy algebra $\liei_{x}$ contains a subalgebra isomorphic 
    to $\oo(1,k)$, for some $k \geq 2$. We thus infer:
    
 \begin{lemme}
 \label{lem.signature}
  Under the asumption that $\carg(x)=\infty$ for all $x \in M$, then every $kiloc$-orbit $\Sigma$ contained 
  in $\mint$ has Lorentz signature.
 \end{lemme}
\begin{proof}   
   Let $x \in \mint$, and let $\Sigma$ be the $\kiloc$-orbit of $x$. Since $\liei_x$ contains a copy of the Lie algebra $\oo(1,k)$, for $k \geq 2$,
    there is a 
local Killing 
field $X$ around $x$, vanishing at $x$, and  such that the flow $\{D_{x}\phi_X^t\} 
\subset \operatorname{O}(T_{x}M)$ is a hyperbolic $1$-parameter group. 
Linearizing $X$ around $x$ thanks to the exponential map, 
we see there are two distinct 
lightlike directions $u$ and $v$ in ${T}_{x}M$ such 
that the two geodesics $\gamma_u: s \mapsto \exp(x,su)$ 
and $\gamma_v: s \mapsto \exp(x,sv)$ are left invariant by $\phi_X^t$.
 In particular,  for $s \not =0$ close to $0$, $\dot{\gamma}_u(s)$ and 
$\dot{\gamma}_v(s)$ are 
colinear to $X$, hence  tangent to the $\kiloc$-orbits ${\mathcal O}(\gamma_u(s))$ and ${\mathcal O}(\gamma_v(s))$ respectively.  
By continuity, this property must still hold for $s=0$. 
We infer that ${T}_{x}(\Sigma)$ contains the two distinct lightlike 
directions $u$ and $v$, hence has Lorentz signature. This holds on all of $\Sigma$ by local 
homogeneity of the $\kiloc$-orbit.
\end{proof}

 We now consider $\calm$ a connected  component of $\mreg$ where the rank of the map $\kg$ is the maximal value $r_{max}$
  that the 
  rank of $\kg$ does achieve on $\hm$.  We consider the pullback $\hat{\calm} \subset \hm$.
   
 We  pick $x \in \calm$, and
 we choose $\hx \in {\hat{\calm}}$  in the fiber of $x$.
        If $U \subset \hcalm$ is a 
       small open set around
        $\hx$, $\kg(U)$ is a $r_{max}$-dimensional submanifold of ${\mathcal W}_{\kg}$. 
        Let us now call $\hat{\Lambda}$ the closed subset of $\hm$
         where the rank of $\kg$ is $ \leq r_{max}-1$. By Sard's theorem, the 
         $r_{max}$-dimensional Hausdorff measure of $\kg(\hat{\Lambda})$ is zero.
         We infer the existence of
          $w \in \kg(U) \setminus \kg(\hat{\Lambda})$. Moving $\hx$ inside $U$, we assume that
          $w=\kg(\hx)$, and we denote by $\calo(w)$ the $\OO(1,n)$-orbit
           of $w$ in $\calw$.  
         By $\Oun$-equivariance of $\kg$, the inverse image $(\kg)^{-1}(\calo(w))$ avoids $\hat{\Lambda}$, hence 
         the rank of $\kg$ is constant equal to $r_{max}$ on
         $(\kg)^{-1}(\calo(w))$.  Let us observe that because the rank can not drop locally, any 
 point of $\hm$ where the rank of $\kg$ is $r_{max}$ must belong to $\hmint$, hence the inclusion 
         $(\kg)^{-1}(\calo(w)) \subset \hmint$.  From the discussion following Theorem \ref{thm.structure}, the 
         connected components of $(\kg)^{-1}(w)$ are $\kiloc$-orbits.  Because $\hmreg$ is saturated in $\kiloc$-orbits, and is
          $\Oun$-invariant, we infer that $(\kg)^{-1}(\calo(w)) \subset \hmreg$.  By Theorem \ref{thm.structure}, the projection of 
          $(\kg)^{-1}(\calo(w))$ on $M$  is a submanifold, whose  connected components are 
           $\kiloc$-orbits. Hence proposition \ref{prop.compact.orbits} will be proved 
           if we show that $(\kg)^{-1}(\calo(w))$
            is closed in $\hm$ (because $\Oun$-invariant closed sets of $\hm$ project
            on compact subsets of $M$). This is 
            a direct consequence of:
         
    \begin{lemme}
     \label{lem.closed}
     Under the hypothesis $\carg(x)=\infty$, the orbit $\calo(w)$ is a closed subset of ${\mathcal W}_{\kg}$.
    \end{lemme}
    
    \begin{proof}
    As already observed the isotropy algebra $\liei_{x}$ contains a subalgebra isomorphic 
    to $\oo(1,k)$, for some $k \geq 2$. Recall that $\liei_{x}$ is identified with the Lie algebra 
    of the stabilizer 
    of $w$ in $\OO(1,n)$ (see Corollary \ref{coro.stabilizer}).
     Since it contains a copy of $\oo(1,k)$, $\liei_{x}$ contains a conjugate of 
     the Cartan algebra $\liea \subset \oo(1,n)$. 
      Hence the Stabilizer of $w$ in $\Oun$ contains a conjugate of the 
     Cartan group $A$ (the exponentiation of $\liea$), that we may assume to be $A$ itself,  
     after moving $\hx$ in its fiber if necessary.  Let now $g_k.w$ be a sequence 
     in $\calo(w)$ that converges to some point $w' \in \calw_{\kg}$. We write $\Oun=KAN$  (Iwasawa decomposition), and 
      decompose $g_k$ accordingly: $g_k=m_ka_ku_k$, with $m_k \in K$, $a_k \in A$, and $u_k \in N$. We may assume that $m_k \to m_{\infty}$.
        Because $w$ is fixed by $a_k$, we observe that $g_k.w=m_ka_ku_ka_k^{-1}.w=m_ku'_k.w$, for some sequence $u'_k \in N$ (the group $N$ is normalized by $A$).
         It follows that $u'_k.w$ converges to $m_{\infty}^{-1}.w'$.  But the orbit  $N.w$ is closed in $\calw_{\kg}$, because
          in a linear representation, the orbits of
          unipotent groups are closed (this is a general property for algebraic actions of unipotent groups, see 
          \cite[Theorem 2]{rosenlichtunipotent}). We finally get $m_{\infty}^{-1}.w' \in \calo(w)$, hence $w' \in \calo(w)$. 
     \end{proof}

 \end{proof}        


\subsection{An embedding theorem in presence of compact $\kiloc$-orbits}
\label{sec.embed}
We now use the nice structure of $\kiloc$-orbits on $\mreg$, 
to prove the following embedding property:

\begin{theoreme}
  \label{thm.embedding1}
  Let $(M^{n+1},g)$ be a compact Lorentz manifold of dimension $n+1$, and let $\mreg$ be its regular locus.
   Assume that $\Sigma \subset \mreg$ is a compact $\kiloc$-orbit, of Lorentz signature. 
  Then
  there exists a finite index subgroup $\Iso'(M,g) \subset \Iso(M,g)$, such that $\Sigma$
   is left invariant by $\Iso'(M,g)$, and such that the restriction morphism 
    $\Iso'(M,g) \to \Iso(\Sigma,g_{|\Sigma})$ is one-to-one and  proper.
\end{theoreme}
  
 The proof will be made in several steps. The first one is to exhibit distinguished neighborhoods for 
  compact Lorentz $\kiloc$-orbits.
 
 \subsubsection{Regular neighborhoods for compact $\kiloc$-orbits}
\label{sec.standardneighbor}

When there exists a compact $\kiloc$-orbit in $\mreg$ having Lorentz signature, the nice structure described
 in Theorem \ref{thm.structure} can be strengthen in the following way.
 
\begin{lemme}[Existence of standard foliated neighborhoods]
 \label{lem.standard.neighborhood}
  Let $\Sigma \subset \mreg$ be a compact  $\kiloc$-orbit.
 Then $\Sigma$ admits a connected  open neighborhood $U \subset \mreg$ satisfying the following properties:
 \begin{enumerate}
  \item The closure $\overline{U}$ is saturated in $\kiloc$-orbits, all of which are compact Lorentz 
  submanifolds.
  \item There exists a total transversal $\overline{B}$ to the $\kiloc$-orbits of $\overline{U}$, which is a submanifold with
  boundary diffeomorphic to a closed $d$-ball.
  \item Every $\Isloc$-orbit of $M$ intersects $\overline{B}$ at most once. In particular, every $\kiloc$-orbits of $\overline{U}$
   intersects  $\overline{B}$ exactly once.
 \end{enumerate}

\end{lemme}

\begin{proof} Let $\Sigma \subset \calm \subset \mreg$ be a compact $\kiloc$-orbit having Lorentz signature. 
We fix 
 $x_0 \in \Sigma$  a reference point.  We introduce, for $x \in \Sigma$ the following notations:
 $$ T_x\Sigma^{\perp}:=\{ u \in T_xM \ | \ u \perp T_x\Sigma  \},$$
 $$  T_x^1\Sigma^{\perp}:=\{ u \in T_x\Sigma^{\perp} \ | \ g(u,u)=1  \},$$
 and for $\epsilon>0$ small enough:
 $$  B_{\epsilon}:=\{  \exp(x_0,su) \ | \ s\in[0,\epsilon), \ u \in T_{x_0}^1\Sigma^{\perp} \},$$
 $$ {\overline B}_{\epsilon}:=\{  \exp(x_0,su) \ | \ s\in[0,\epsilon], \ u \in T_{x_0}^1\Sigma^{\perp} \}.$$
 
 We recall that on $\calm$, the smooth map $\okg: \calm \to Y$ has constant rank, hence defines a foliation ${\mathcal F}$ on $\calm$.
  Each leaf of this foliation is a $\kiloc$-orbit  (see Section \ref{sec.structure}). Moreover the $\isloc$-orbits, intersected with $\calm$, 
  are contained in the fibers of $\okg_{| \calm}$.
   
 For $\epsilon_0>0$ small enough, $B_{\epsilon_0}$ is a transversal to $\mathcal F$. Observe that 
 (taking $\epsilon_0>0$ even smaller) the metric $g$ restricts to a Riemannian metric on  $B_{\epsilon_0}$.  Indeed, $\Sigma$
  is Lorentzian by assumption, hence the spaces $T_x\Sigma^{\perp}$ are spacelike. 
    A direct application of the inverse mapping 
  theorem shows that if $\epsilon_0$ is chosen small enough,  $\okg$  realizes a smooth diffeomorphism from 
  $B_{\epsilon_0}$ onto a submanifold $Y' \subset Y$. It follows that the fibers of $\okg$ intersect $B_{\epsilon_0}$
   at most once. We call $V$ the saturation of $B_{\epsilon_0}$ by leaves of $\calf$. We then have a natural 
    smooth surjective submersion $\pi: V \to B_{\epsilon_0}$.
    
   For $0 < \epsilon_1 < \epsilon_0$ small enough, we define $\overline{V}_{\epsilon_1}:=\{ \exp(x,su) \ | \  x \in \Sigma, \ u \in 
   T_{x}^1\Sigma^{\perp}, \ s \in [0,\epsilon_1]\}.$  If $\epsilon_1$ is small enough,  $\overline{V}_{\epsilon_1}$ is a compact 
   neighborhood of $\Sigma$ contained in $V$.  We finally call ${U}$ the saturation of ${B}_{\epsilon_1}$
    by $\kiloc$-orbits. We claim that ${U}$  satisfies the properties of Lemma \ref{lem.standard.neighborhood}.
     Actually the only point which is still unclear and that we must prove
     is that $\overline{U}$ is compact in $V$, and saturated by $\kiloc$-orbits.  The compactness of all $\kiloc$-orbits in $\overline{U}$ will follow because
      $\kiloc$-orbits are closed in $V$.
     Let us call $\overline{U}'$ the closure of $U$ in $V$, which is nothing but the saturation of ${\overline{B}}_{\epsilon_1}$
      by $\kiloc$-orbits.   
      
    Let us pick $x \in \overline{U}' \cap \overline{V}_{\epsilon_1}$, and let $X$ be a Killing field defined in a neighborhood of 
     $x$.  Because $x \in \overline{V}_{\epsilon_1}$, there exists $z \in \Sigma$, $u \in T_{z}^1\Sigma^{\perp}$  and 
     $a \in [0,\epsilon_1]$  such that $x=\exp(z,au)$.  Let us call, for every $s \in [0,1]$, $\gamma(s)=\exp(z,asu)$.
      This is a geodesic segment, homeomorphic to the closed unit interval (this is so if $\epsilon_1$ is small enough). We can 
      consider a 
      $1$-connected open neighborhood $W \subset V$ containing $\gamma$, and extend $X$ to a Killing field on $W$ 
      (see Section \ref{sec.analyticcont}).  For $t \in (-\delta,\delta)$ with $\delta>0$ small, $\varphi_X^t(\gamma)$ is well 
      defined.  It is
       a geodesic segment joining $z(t):=\varphi_X^t(z)$ to $x(t):=\varphi_X^t(x)$, orthogonal to $\Sigma$ at $z(t))$, and 
        of positive Lorentz length $a$.  Hence one can write $x(t)=\exp(z(t),au(t))$, where $u(t) \in T_{x(t)}^1\Sigma^{\perp}$.
         This means $x(t) \in \overline{U}' \cap \overline{V}_{\epsilon_1}$ for all $t \in (-\delta,\delta)$.
       What we have proved is that in any $\kiloc$-orbit of $\overline{U}'$, the set of points belonging to  $ \overline{V}_{\epsilon_1} $
        is open.  Since it is obviously closed, and because ${\overline B}_{\epsilon_1} \subset \overline{V}_{\epsilon_1}$,
         we infer that all $\kiloc$-orbits of $\overline{U}'$ are included in  $\overline{V}_{\epsilon_1}$.
      
We thus have the inclusion  $\overline{U}' \subset \overline{V}_{\epsilon_1}$.  The compactness of $\overline{U}'$ follows because
     $\overline{U}'=\pi^{-1}(\overline{B}_{\epsilon_1})$ is closed in $V$, and $\overline{V}_{\epsilon_1}$ is compact in $V$.
      We conclude that $\overline{U}=\overline{U}'$, and the lemma is proved.
\end{proof}

 \subsubsection{Proof of Theorem \ref{thm.embedding1}}
  We consider $U$ a standard foliated neighborhood of $\Sigma$, as given in
 Lemma \ref{lem.standard.neighborhood}, and $\pi: \overline{U} \to \overline{B}$ the projection, whose fibers are exactly
  the $\kiloc$-orbits in $\overline{U}$.  
  For $f \in \Iso(M,g)$, we denote by $\Lambda_f$ the set of
  points $x \in U$ such that the $\kiloc$-orbit $\calo(x)$ satisfies $f(\calo(x))=\calo(x)$.  Observe that 
  because of property $(3)$ of Lemma \ref{lem.standard.neighborhood}, if $f(U)\cap U \not = \emptyset$, then 
  $\Lambda_f \not = \emptyset$.
  
\begin{lemme}
 \label{lem.open.closed}
 The set $\Lambda_f$  
  is open and closed in $U$.
\end{lemme}
 
\begin{proof}
 We first prove that $\Lambda_f$ is closed. For this, let us consider $(x_k)$ a sequence of $\Lambda_f$ converging
  to $x_{\infty} \in U$.  Then $f(x_k)$ converges to $f(x_{\infty})$, and $f(x_{\infty}) \in \overline{U}$.
   Now $\pi(f(x_k))=\pi(x_k)$ for all $k$, so $\pi(f(x_{\infty}))=\pi(x_{\infty})$.  It follows that $x_{\infty}$ and 
   $f(x_{\infty})$ belong to the same $\kiloc$-orbit, hence $x_{\infty} \in \Lambda_f$.
   
To check that $\Lambda_f$ is open, let us pick $x \in  \Lambda_f$.  By assumption, $f(x) \in U$.  Because $U$ is open, there exists a 
small open set $V \subset U$ containing $x$ such that $f(V) \subset U$.  Then, by the third property of
Lemma \ref{lem.standard.neighborhood}, $V \subset \Lambda_f$.

\end{proof}
 
\begin{lemme}
 \label{lem.stabilizer.index}
  The stabilizer of $U$ in $\Iso(M,g)$ has finite index in 
  $\Iso(M,g)$.
\end{lemme}

 \begin{proof}
Let $S$ denote the stabilizer of $U$ in $\Iso(M,g)$, and   let $(f_i)_{i \in I}$ be a family
 of elements in $\Iso(M,g)$, such that $f_iS \not = f_jS$ whenever $i \not = j$.  We oberve that $f_i(U) \cap f_j(U) =\emptyset$.
  Indeed, if $ f_i(U) \cap f_j(U) \not = \emptyset$, then  by the remark before Lemma \ref{lem.open.closed}, we would have 
  $\Lambda_{{f_j^{-1}}f_i} \not = \emptyset$.  By connectedness of $U$ and Lemma \ref{lem.open.closed}, this would
    imply $\Lambda_{f_j^{-1}f_i} = U$, or in other words ${f_j^{-1}f_i} \in S$: Contradiction.  Because all sets $f_i(U)$
     have a same positive Lorentz volume, there can be only finitely many $(f_i)$ such that  $f_i(U)$ are 
     pairwise disjoint, and we are done.
\end{proof}  
  
We call in the following $\Iso'(M,g)$ the stabilizer of $U$ in $\Iso(M,g)$.  The $\kiloc$-orbit $\Sigma$ is 
left invariant by $\Iso'(M,g)$, because of point $(3)$ in Lemma \ref{lem.standard.neighborhood}.  
 
 \begin{lemme}
  \label{lem.restriction}
  The restriction morphism $\rho: \Iso'(M,g) \to \Iso(\Sigma, g_{|\Sigma})$ is one-to-one 
 and proper.
 \end{lemme}
\begin{proof}
  Properness comes from the fact that $\Iso(M,g)$ acts properly on the bundle of orthonormal frames 
  on $M$.  Thus, because $\Sigma$ has Lorentz signature, if the $1$-jet of $\rho(f_k)$ remains bounded
   along a sequence $(x_k)$ of $\Sigma$, then the $1$-jet of $(f_k)$ along $(x_k)$ remains bounded. In particular
    if $(\rho(f_k))$ has  
  compact closure in 
  $\Iso(\Sigma,g)$, then $(f_k)$ has compact closure in $\Iso'(M,g)$.
  
  The fact that $\rho$ is one-to-one comes from the properties of the neighborhood $U$.  Assume indeed that some
  $f \in \Iso'(M,g)$ acts trivially on $\Sigma$.  We pick $x \in \Sigma$ and look at the transformation $D_xf$.
   The tangent space $T_xM$ splits as an orthogonal sum $T_xM=T_x\Sigma \oplus T_x\Sigma^{\perp}$.  The linear transformation 
   $D_xf$ acts trivially on $T_x\Sigma$.  If the action of $D_xf$ on $T_x\Sigma^{\perp}$ is nontrivial, then because the
   exponential map conjugates the action of $D_xf$ around $0_x$ and that of $f$ around $x$, we would 
    get, for a small disc $D \subset T_x\Sigma^{\perp}$, two points $y$ and $y'$ of $\exp(x,D)$
     in the same $\Iso'(M,g)$ orbit, hence in the same $\kiloc$-orbit because of Lemma \ref{lem.standard.neighborhood}.
      This is absurd since $\exp(x,D)$ must be transverse to $\kiloc$-orbits if $D$ is small enough.
       As a conclusion, the map $D_xf$ is trivial, hence $f$ is the identity map on $M$
       (Lorentz isometries having a trivial $1$-jet at one point are trivial).  This concludes the proof of 
       Theorem \ref{thm.embedding1}.
 \end{proof}


 \subsection{The proof of Theorem \ref{thm.reduction}}
 \label{sec.proof.reduction}
 
 Theorem \ref{thm.reduction} readily follows from what we have done so far. By Corollary \ref{coro.infinite.limit}
  and 
 Proposition 
 \ref{prop.compact.orbits}, the presence of $G \subset \Iso(M,g)$ with exponential growth yields a compact 
  $\kiloc$-orbit $\Sigma \subset \mreg$ having Lorentz signature. We can then apply 
  Theorem \ref{thm.embedding1}, which yields a finite index subgroup $\Iso'(M,g)$ leaving $\Sigma$ invariant, and such that the
   inclusion
  $\Iso'(M,g) \to \Iso(\Sigma,g_{|\Sigma})$ is one-to-one and proper. In particular, the group $\Iso(\Sigma,g_{|\Sigma})$ contains a compactly generated 
  subgroup of exponential growth. Theorem \ref{thm.exponential.killing} then ensures that the locally
  homogeneous manifold $(\Sigma,g_{|\Sigma})$ has semisimple isotropy.



\section{Geometry of locally homogeneous Lorentz manifolds with semisimple isotropy}
\label{sec.completeness}

Under the assumptions of Theorem \ref{thm.main}, we showed in Theorem \ref{thm.reduction} the existence of a one-to-one and proper
homomorphism $\rho: \Iso'(M,g) \to \Iso(\Sigma,h)$, where $\Iso(M,g)' \subset \Iso(M,g)$ is a finite index subgroup, and $(\Sigma,h)$
 is a compact, locally homogeneous Lorentz manifold, with semisimple isotropy. 
  This  section is thus devoted to the general description of this class of Lorentz manifolds, 
 which, beside being a crucial step in the proof of Theorem \ref{thm.main},  is a topic of independent interest. 
  
 Precisely, the situation we are investigating is that of a compact, connected, 
 locally homogeneous, Lorentz manifold $(M,g)$. The isotropy algebra at $x$, denoted $\liei_x$, is the 
 algebra of vector fields in $\lieg$ vanishing at $x$.  By local homogeneity, all the algebras $\liei_x$ are pairwise isomorphic, so that we 
 will speak about {\it the isotropy algebra} $\liei$ of $(M,g)$. Our standing assumption is that $(M,g)$ has 
 {\it semisimple isotropy}, 
  which means that $\liei$ is isomorphic to $\oo(1,k) \oplus \oo(m)$, $k \geq 2$.  The following proposition will be one of the crucial steps needed to prove Theorem \ref{thm.main}.
   The reader may take it for granted on a first reading, and go directly to Section \ref{sec.proof-main}.
  
  \begin{proposition}
  \label{prop.homogeneous.partial}
  Let $(M,g)$ be a compact Lorentz manifold.  
  Assume that $M$ is locally homogeneous, and that its isotropy algebra $\liei$ is isomorphic to $\oo(1,k) \oplus \oo(m)$, 
   with $k \geq 2$. Assume that  $\Iso(M,g)$ contains a closed, compactly generated subgroup with exponential growth.
   Then:
  \begin{enumerate}
   \item There exists a simply connected complete homogeneous Riemannian manifold $(N,g_N)$,  and a smooth function
   $w: N \to \RR_+^*$, 
    such that the universal cover 
  $(\tm,\tilde{g})$ is isometric to the warped product $N \times_w X$, 
   where $(X,g_X)$ is isometric to either the $3$-dimensional anti-de Sitter space 
   $\widetilde{\mathbb{ADS}}^{1,2}$ (in which case $k=2$) or Minkowski space ${\RR}^{1,k}$.
   \item The isometry group $\Iso(\tm, \tilde{g})$ is included in $\Iso(N) \times \operatorname{Homot}(X)$.
    In particular, the manifold $(M,g)$ is the quotient of $N \times X$ by a discrete subgroup 
   $\Gamma \subset \Iso(N) \times \operatorname{Homot}(X) $.
  \end{enumerate}

 \end{proposition}  
  
 In the proposition $\operatorname{Homot}(X)$ stands for the group of homothetic transformations of $X$, namely those 
 $\varphi: X \to X$ such that $\varphi^*g_X= \lambda g_X$, for some nonzero scalar $\lambda$. When 
  $X=\widetilde{\mathbb{ADS}}^{1,2}$, this group coincides with the isometry group.
 The  proof of Proposition \ref{prop.homogeneous.partial} will be done in two steps. The first one is a geometric  description
  of the universal cover 
 $(\tm,{\tilde g})$ (see Proposition \ref{prop.universal}). It will
  be the aim of Section \ref{sec.universal.cover}, which is pretty close to \cite[Sec. 2 and 3 ]{zeghibisom2}. 
  
 The second step is more difficult, and establishes completeness results, under the assumption that the limit set of $\Iso(M,g)$
  is infinite at each point. Those problems will be tackled in Sections \ref{sec.warping.constant} and \ref{sec.warping.nonconstant} 
  (see Theorem \ref{thm.complete}).

 \subsection{Geometry of the universal cover $\tm$}
 \label{sec.universal.cover}
 
 Our first aim is to prove 
 \begin{proposition}
  \label{prop.universal}
  Let $(M,g)$ be a compact, locally homogeneous, Lorentz manifold with semisimple isotropy.
  Then the universal cover $(\tm, {\tilde g})$ is isometric to a warped product $N \times_w X$, where
   $(N,g_N)$ is a simply connected, homogeneous, Riemannian manifold, and $(X,g_X)$ is Lorentzian of constant curvature 
   and dimension $\geq 3$. Moreover $\Iso(\tm, {\tilde g}) \subset \Iso(N) \times \operatorname{Homot}(X)$, so that 
    the manifold
    $(M,g)$ is obtained as a quotient of $N \times_w X$ by a discrete subgroup $\Gamma \subset \Iso(N) \times \operatorname{Homot}(X)$.
 \end{proposition}

 \subsubsection{Bifoliation on the universal cover $\tm$}
 \label{sec.bifoliation}
 
 We begin with an important remark abiut Killing fields on $(\tm, {\tilde g})$. The manifold $\tilde{M}$ is locally homogeneous, 
 hence
  the integrability locus $\tilde{M}^{int}$ coincides with $\tilde{M}$. The process of extending analytically Killing fields along pathes  in $\tilde{M}^{int}$ (see Section \ref{sec.analyticcont}), and the simple connectedness of $\tilde{M}$ shows that any local Killing 
  field in $\tilde{M}$ extends to a global one. Thus in the following, all Killing fields on $\tm$ {\it will be globally defined}
  (this does of course not mean that those fields are complete).
  
   For any $x \in \tm$, the isotropy representation  $\liei_x \to \oo(T_x\tm)$
    defined by  $X \mapsto \nabla X(x)$  is faithful. We thus identify $\liei_x$ with a subalgebra of $\oo(1,n)$.
    Our assumption says that $\liei_x$ contains a subalgebra isomorphic to $\oo(1,k)$, 
    $k \geq 2$ (with $k$ maximal for this property). Thus  $\liei_x$ splits, in a unique way,  as a sum 
    $\liei_x=\lies_x \oplus \liem_x$, where $\lies_x$ is isomorphic to $\oo(1,k)$ and $\liem_x$ is isomorphic to a 
    subalgebra of $\oo(n-k)$. This provides 
      a $\liei_x$-invariant
     splitting $T_x\tm={\calf}_x \oplus {\calf}_x^{\perp}$, where ${\calf}_x$ has Lorentz signature and dimension $k+1$, and 
      ${\calf}_x^{\perp}$ is $n-k$ dimensional, orthogonal to ${\calf}_x$ (and hence of Riemannian signature). The Lie algebra
      $\lies_x$ 
    (resp. $\liem_x$) acts
    on ${\calf}_x$ (resp. on ${\calf}_x^{\perp}$) by the standard $(k+1)$-dimensional representation of $\oo(1,k)$ 
    (resp. through the standard $(n-k)$-dimensional representation of $\oo(n-k)$) and 
    trivially on ${\calf}_x^{\perp}$ (resp. on ${\calf}_x$). 
    
  We thus inherits  two (mutually orthogonal) distributions 
  ${\mathcal F}$ and ${\mathcal F}^{\perp}$  on $\tm$.  Observe that any local isometry sending $x$ to $y$
   will map $\liei_x$ to $\liei_y$.  This implies that the distributions $\calf$ and $\calf^{\perp}$
    are invariant by any local isometry of $\tm$. In particular, those distributions are smooth.
  
  \begin{lemme}
   \label{lem.foliations}
   The two distributions $\calf$ and $\calf^{\perp}$ are integrable.  Moreover, the leaves of $\calf^{\perp}$
    are totally geodesic, and those of $\calf$  are totally umbilic, with  constant sectional curvature. 
    \end{lemme}

  \begin{proof}
   Let $x \in \tm$.  Denote by $Z_x$ the subset of $\tm$, where all elements of $\lies_x$ vanish.  
    It is a classical fact that $Z_x$ is a totally geodesic submanifold of $\tm$ (it is easily checked using the fact that exponential map
    linearizes local flows of isometries).  If $y \in Z_x$, then $\lies_y=\lies_x$, because we clearly
     have $\lies_x \subset \lies_y$, and those two algebras have same dimension (namely the dimension of $\oo(1,k)$).
      It follows that $Z_x$ is actually a leaf of $\calf^{\perp}$, and it proves the assertions about 
      $\calf^{\perp}$.
      
   Let us now check that $\calf$ is integrable as well.  Let us consider $Y,Z$ two local vector fields 
   tangent to $\calf$, and
   let us call  $[Y,Z]^{\perp}$ (resp. $II(Y,Z)$) the component of $[Y,Z]$ (resp. of $\nabla_YZ$)
   on $\calf^{\perp}$.  One readily checks that those maps
    are tensorial, namely for any pair of functions $f$ and $g$, then $[fY,gZ]^{\perp}=fg[Y,Z]^{\perp}$, 
     and $II(fY,gZ)=fgII(Y,Z)$.  
      Let us consider $x \in \tm$, and a bilinear map  $b_x: {\calf}_x \times {\calf}_x \to \RR$.
       We can write $b_x(\, ,\,)=g_x(A \, , \,)$ for some endomorphism $A:{\calf}_x \to {\calf}_x$.
        Let us denote by $I_x$ the isotropy group at $x$. If $b_x$ is $I_x$-invariant, then $A$
         must commute with $I_x$, and because the action of $I_x$ on ${\calf}_x$ is irreducible, it means
          that $A$ is an homothetic transformation.  We have shown that any bilinear form on ${\calf}_x$ which
           is $I_x$-invariant is a scalar multiple of $g_x$ (restricted to ${\calf}_x$).
            In particular $[\, , \, ]_x^{\perp}$ must be zero, what shows that $\calf$ is an involutive distribution, 
            hence
             is integrable.  We also get that $II_x(\, , \, )=g_x(\, , \, ) \nu_x$, for some vector
             $\nu_x \in {\calf}_x^{\perp}$, what means precisely that the leaves of $\calf$ are totally umbilic.

   Let us conclude the proof of Lemma \ref{lem.foliations} by showing that the leaves of $\calf$ have constant sectional 
    curvature. Let $x \in \tm$, and $F(x)$ the leaf of $\calf$ containing $x$. 
    For any $Z \in \liei_x$, the local flow 
   $D_x\varphi_Z^t$ acts on the manifold of non degenerate $2$-planes of $T_xM$. If $P$ is such a $2$-plane, then 
    $K(P)=K(D_x \varphi_Z^t(P))$, where $K(P)$ stands for the sectional curvature of $P$ (relatively to the curvature tensor of $g$).
     Because $\liei_x$ contains  a subalgebra isomorphic to $\oo(1,k)$, one easily gets that $K(P)$ is the same for every
      nondegenerate 2 plane $P \subset T_xM$. Now let $\overline{K}(P)$ be the sectional curvature of $P$, computed with 
      respect to the metric induced by $g$ on $F(x)$. Because $F(x)$ is totally umbilic, we have that 
      \begin{equation}
      \label{eq.curvature}
      K(P)=\overline{K}(P)+g_x(\nu_x,\nu_x)
      \end{equation}
      To check this, let us consider two local vector fields $X$ and $Y$, tangent to $F(x)$, such 
      that $g(X,Y)=0$ and $g(X,X)=\epsilon=\pm 1$ and  $g(Y,Y)=1$. Let $\onabla$ be the Levi-Civita connection of 
      the restriction $g_{|F(x)}$.  Using the property that $F(x)$ is totally umbilic,
       we compute $\nabla_X \nabla_YX=\onabla_X\onabla_YX+g(X,\onabla_YX)\nu=\nabla_X\onabla_YX.$
       But $Y.g(X,X)=0=2g(\onabla_YX,X)$. We get $\nabla_X \nabla_YX=\onabla_X\onabla_YX$.
       
       Writting $g(X,X)=\epsilon$, we also have 
       $$\nabla_Y\nabla_XX=\nabla_Y(\onabla_XX+ \epsilon \nu)=\nabla_Y\onabla_XX+ \epsilon \nabla_Y \nu.$$
       This yields 
       $$ \nabla_Y\nabla_XX=\onabla_Y \onabla_XX+g(Y,\onabla_XX)+\epsilon \nabla_Y \nu=\onabla_Y \onabla_XX+\epsilon \nabla_Y \nu.$$
       We thus obtain $R(X,Y,X,Y)=\overline{R}(X,Y,X,Y)- \epsilon g(\nabla_Y \nu,Y).$
       But 
       $$Y.g(\nu,Y)=0=g(\nabla_Y \nu,Y)+g(\nu,\nabla_YY).$$
        It follows that 
        $$g(\nabla_Y \nu,Y)=-g(\nu,\onabla_YY+g(Y,Y)\nu)=-g(\nu,\onabla_YY)-g(\nu,\nu)=-g(\nu,\nu).$$
        Finally,  $R(X,Y,X,Y)=\overline{R}(X,Y,X,Y)+g(\nu,\nu)$, which is precisely (\ref{eq.curvature}).

      Because  $K(P)$ is the same for every
      nondegenerate $2$-plane $P \subset T_xM$, equation (\ref{eq.curvature}) ensures that the same property holds
      for $\overline{K}(P)$. 
       This remark, together with Schur's lemma, implies that the leaves of $\calf$ have constant sectional curvature.

     Actually, under the assumption that $(M,g)$ is locally homogeneous, the sectional curvature $\kappa(x)$ of the leaf $F(x)$
     does not depend on $x$,
      because we already 
     noticed  that local isometries 
      of $\tm$ preserve the distributions $\calf$ and $\calf^{\perp}$.

  \end{proof}
  
 \subsubsection{Warped product structure for $\tm$} 
 \label{sec.warped}
  The arguments until Lemma \ref{lem.produit-complet} below already appear in \cite[Section 4.3]{zeghibisom2}.
   We already noticed that $\calf$ and $\calf^{\perp}$ are invariant by local isometries.  In particular, those 
  two foliations are $\pi_1(M)$-invariant, and thus induce two mutually orthogonal foliations $\ocalf$ and 
  $\ocalf^{\perp}$ on $M$.  One can then choose a Riemannian metric on $T\ocalf$, and still put  the restriction
   of $g$ on $T\ocalf^{\perp}$ in order to build a Riemannian metric $h$ on $M$, for which $\ocalf$
    and $\ocalf^{\perp}$ are still orthogonal.  We observe that $\ocalf^{\perp}$ remains totally geodesic
     for $h$.  Indeed, the property for $\ocalf^{\perp}$ to be totally geodesic is equivalent to $\ocalf$
      being transversally Riemannian, namely the holonomy local diffeomorphisms, between small open sets of 
      leaves of $\ocalf^{\perp}$ are isometries (see \cite[Prop 1.4]{johnson}).  It was the case for the metric $g$ and it is still the case for
      $h$ since those two metrics coincide in restriction to leaves of $\ocalf^{\perp}$.  One  can then use
      \cite[Theorem A]{blumenthal} to infer that the two foliations $\calf$ and $\calf^{\perp}$ define a product structure on $\tm$.
       More precisely, if we pick $z_0 \in \tm$, and call $X$ and $N$ the leaves of $\calf$ and
        $\calf^{\perp}$ containing $z_0$, then there is a diffeomorphism $\psi: N \times X$
         such that each factor $N \times \{ x \}$  (resp.  $\{ n \} \times X$) is sent to a leaf 
         of $\calf^{\perp}$ (resp. of of $\calf$), with moreover $\psi(N \times \{x_0\})=F_0$ 
         and $\psi(\{n_0 \} \times X)=F_0^{\perp}$.  
         
  Let  us now discuss the form of the metric $\tilde{g}$ on $N \times X$.  The metric $\tilde{g}$ restricts to a Riemannian metric $g_N$
   on the leaf $N$, and to a Lorentzian metric $g_X$ of constant curvature $\kappa$ on the leaf $X$.
   Observe that $(N,g_N)$ is complete, because leaves of $\ocalf^{\perp}$ are complete for
   the Riemannian metric $h$ constructed above, and $h$ coincides with
   $g$ on the leaves of $\ocalf^{\perp}$.
     As already mentioned, the maps
   $(n,x) \mapsto (n,x')$ are isometries because leaves $N \times \{ x\}$ are totally geodesic.  Moreover, because
    leaves $\{p \} \times X$ are umbilic, the maps $(n,x) \mapsto (n',x)$ are conformal (see \cite[lemma 5.1]{blumenthal0}).  It follows
     that there exists a Riemannian metric $g_N$ on $N$, a Lorentz metric $g_X$ on $X$ with constant curvature
      $\kappa$, and a function $w: N \times X \to \RR_+^*$ such that $\tilde{g}=g_N \oplus wg_X$.  
      
  \begin{lemme}
   \label{lem.produit-complet}
   The function $w: (n,x) \mapsto w(n,x)$ does not depend on $x$. 
  \end{lemme}
    
  \begin{proof}
   Let us fix $n \in N$, and let us pick  $x_0 \in X$. We will denote  $w_n: X \to \RR$ the function defined by 
   $w_n(x)=w(n,x)$. Let $z=(n,x_0)$, and $Z$ be a vector field in $\lies_z$. Recall that $Z$ is globally defined.
    Because the local flow of $\lies_z$ must preserve $\calf$ and $\calf^{\perp}$, there exist $Z_1$ a 
    Killing field on $(N,g_N)$ and $Z_2$ a vector field on $X$, such that $Z(p,x)=(Z_1(p),Z_2(x))$ 
    for every $p \in N$ and $x \in X$. We already observed  that because $Z$ belongs to  $ \lies_z$,
     it vanishes on $N \times \{ x_0 \}$. It follows that $Z_1=0$. Thus, Because $Z$ is a Killing field on $(\tm,\tilde{g})$, 
     the vector field $Z_2$ is a Killing field for all the metrics $w_pg_X$, $p \in N$.
      In particular, it is Killing for both metrics $g_X$ and ${w_n}g_X$, hence $w_n$ is left invariant by $Z_2$. 
      Now, because the Lie algebra $\lies_z$ acts on $T_{x_0}X$ by the standard $(k+1)$-dimensional representation of $\oo(1,k)$
       the local punctured lightcone of $X$ at $x_0$ 
      is a local pseudo-orbit of ${\lies}_z$. It follows that   $D_{x_0}w_n(u)=0$ for every $u$ in the lightcone of 
      $T_{x_0}X$. This lightcone spans $T_{x_0}X$ as a vector space, hence  $D_{x_0}w_n=0$.  Because $x_0$ was chosen arbitrarily 
       in  $X$, the lemma follows.
  \end{proof}

\subsubsection{Isometries of $N \times_w X$}
\label{sec.isometries}
 
 We have established the warped-product structure of $(\tm,\tilde{g})$ announced in Proposition 
 \ref{prop.universal}. It remains to show that $(N,g_N)$ 
 is homogeneous, and  that $(M,g)$ 
 is obtained as a quotient of $N \times_w X$ by a discrete subgroup $\Gamma \subset \Iso(N) \times \operatorname{Homot}(X)$.
 This will be obtained thanks to points $(2)$ and $(3)$ of Lemma \ref{lem.isometries} below.
 
 We will call in the following $\widetilde{\Iso}(M,g)$ the subgroup of $\Iso(\tm,\tilde{g})$ obtained as all 
 possible lifts of isometries in $\Iso(M,g)$. In particular, if $\Gamma \simeq \pi_1(M)$ denotes the group of deck 
 transformations of the covering $\tm \to M$, then $\Gamma \subset\widetilde{\Iso}(M,g)$.
 
Let us consider the group $\Iso(N,g_N)$, acting on $C^{\infty}(N)$ in the following way: For
every $f \in \Iso(N,g_N)$,
$$ (f.v)(n):=v(f(n)),$$
for every $v \in C^{\infty}(N)$ and $n \in N$.
We denote by $G$ the stabilizer of the line $\RR w$ under this representation. We observe that
 $G$ is a closed subgroup of $\Iso(N,g_N)$, hence a Lie subgroup. We call $\lieg$ its Lie algebra (that may 
 be trivial at this stage). The group $G$ admits a continuous homomorphism
  $\lambda: G \to \RR_+^*$, satisfying $w(g.n)=\lambda(g)w(n)$ for every $g \in G$ and $n \in N$.
 
We call $\homot(X)$ the group of homothetic transformations of $X$, namely diffeomorphisms 
$\varphi \in \operatorname{Diff}(X)$ for which there exists some real number $\lambda \in \RR_+^*$
  such that $\varphi^*g_X=\lambda g_X$.
  We can now describe more precisely the isometries of the manifold $N \times_w X$. The assumptions are still those of 
  Proposition \ref{prop.universal}.
\begin{lemme}
 \label{lem.isometries}
 \begin{enumerate}
  \item A transformation $\varphi=(f,h) \in \diff(N) \times \diff(X)$ acts isometrically on $N \times_w X$
   if and only if $f \in G$ and $h^*g_X=\lambda(f)^{-1}g_X$.
  \item The group $G \times \homot(X)$ contains $\widetilde{\Iso}(M,g)$, and in particular $\Gamma \simeq \pi_1(M)$.
  \item The group $G$ acts transitively on $N$. In particular $(N,g_N)$ is a homogeneous Riemannian manifold.
 \end{enumerate}

\end{lemme}

\begin{proof}
 Let us consider $\varphi=(f,h) \in \diff(N) \times \diff(X)$, and let $\xi=(u,v) \in T_nN \times T_xX$.
 We compute that $|\xi|^2=|u|_{g_N}^2+w(n)|v|_{g_X}^2$ while 
 $|D\varphi(\xi)|^2=|Df(u)|_{g_N}^2+w(f(n))|Dh(v)|_{g_X}^2$. We get that  $\varphi$ will be an isometry if and only if
  $f \in \Iso(N,g_N)$ and $|Dh(v)|_{g_X}^2=\frac{w(n)}{w(f(n))}|v|_{g_X}^2$ for every $v \in TX$.
   In particular $\frac{w(n)}{w(f(n))}$ does not depend on $n \in N$, what proves $f \in G$, and 
   $|Dh(v)|_{g_X}^2=\lambda(f)^{-1}|v|_{g_X}^2$ as claimed in the lemma.
   
 Point $(2)$ is a direct consequence of point $(1)$.  
 
 Let us prove point $(3)$.  
  
  Let us consider $X=(Y,Z)$ a Killing field on $N \times_w X$. Recall that $X$ is defined globally. 
   The same computations as in point $(1)$, for Killing fields, imply that
   $Y$ must be a (global) Killing field of $(N,g_N)$, satisfying moreover that 
   $n \mapsto \frac{D_nw(Y(n))}{w(n)}$ is a constant function. We already observed in Section \ref{sec.warped}
    that $(N,g_N)$ is complete. 
   As a consequence, every Killing field on $N$ must be complete as well. This is a classical property, which follows from the obvious 
   fact that Killing fields have constant norm along their integral curves. Thus incomplete integral curves would yield curves of finite Riemannian 
   length, leaving every compact subset of $N$. This is not possible on a complete manifold.
   
     Constancy of the map $n \mapsto \frac{D_nw(Y(n))}{w(n)}$ implies that the $1$-parameter group $\varphi_Y^t$ integrating $Y$ belongs to the group $G$.
      In other words, $Y \in \lieg$. By local homogeneity of $\tilde{M}$, the evaluation map 
      $\mathfrak{kill}(\tilde{M}) \to T\tilde{M}$ is onto at each point, what implies that the evaluation map 
      $\lieg \mapsto TN$ is also 
      onto at each point. Transitivity of the action of $G$ on $N$ follows.
   \end{proof}

  \begin{remarque}
   \label{rem.general}
   Observe that the first point of Lemma \ref{lem.isometries} holds for general warped products $N \times_w X$, not necessarily
    the universal cover of a compact locally homogeneous Lorentz manifold with semisimple isotropy. 
  \end{remarque}

\subsection{Completeness issues}
\label{sec.completeness.issues}

A weakness in the description made in Proposition \ref{prop.universal} is that the factor 
$(X,g_X)$ might not be a {\it complete}, simply connected manifold of constant curvature. This is a serious limitation in 
 the full understanding of $(\tm, \tilde{g})$, thus of the manifold $(M,g)$.
  Those completeness issues are quite subtle. When the factor $N$ in 
 Proposition \ref{prop.universal} is reduced to a point, then $(X,g_X)$ is complete by a deep theorem of Y. Carri\`ere and
 B. Klingler (see Theorem \ref{thm.carriere-klingler} below). For arbitrary factors $(N,g_N)$ (even homogeneous ones), it is not clear
  how to prove that $(X,g_X)$ is complete, eventhough it is likely to be true (this problem is evocated in 
  \cite[Sec. 4.3]{zeghibisom2}. As the following sections show, some serious difficulties 
  were overlooked there). When proving the completeness of $(X,g_X)$, we will actually need  an extra assumption
  about the group $\Iso(M,g)$.
    To explain this, let us consider the situation of a compact Lorentz manifold $(M,g)$, obtained as a quotient of a warped product
   $N \times_w X$ by a discrete subgroup $\Gamma \subset \Iso(N) \times \operatorname{Homot}(X)$.
     The manifold $M$ is endowed with two foliations $\calf$ and $\calf^{\perp}$, whose leaves are respectively the projections of the sets 
      $\{ n \} \times X$ ans $N \times \{x \}$.  The subgroup of $\Iso(M,g)$ preserving the bifoliation $(\calf, \calf^{\perp})$
       is denoted by $\Iso^{\times}(M,g)$. Let us remark that whenever $(M,g)$ is locally homogeneous with semisimple isotropy, 
        and $N \times_w X$ is the natural warped product structure on $(\tm, \tilde{g})$ exhibited in Section \ref{sec.warped}, 
         then
        $\Iso^{\times}(M,g)=\Iso(M,g)$.
       
       We can now state the completeness theorem we will need:

\begin{theoreme}
  \label{thm.complete}
  Let $(M,g)$ be a compact Lorentz manifold, such  that the universal cover 
  $(\tm,\tilde{g})$ is isometric to a warped product $N \times_w X$  where $(N,g_N)$ is Riemannian 
   and $(X,g_X)$ is Lorentzian of dimension $\geq 3$ and constant sectional curvature. Assume that $M$ is obtained as a quotient of 
    $N \times X$ by a discrete subgroup $\Gamma \subset \Iso(N) \times \operatorname{Homot}(X)$. 
   Assume moreover that  $\Iso^{\times}(M,g)$ has an infinite limit set at every point of $M$.
   
   Then the factor $(X,g_X)$ is complete and isometric to either the $3$-dimensional anti-de Sitter space 
   $\widetilde{\mathbb{ADS}}^{1,2}$ or Minkowski space ${\RR}^{1,k}$, $k \geq 2$.
  \end{theoreme}
  
  Observe that there is no local homogeneity assumption in the previous statement. 
  
  In the case of a locally homogeneous 
  Lorentz manifold with semisimple isotropy, Theorem \ref{thm.complete} implies directly Proposition \ref{prop.homogeneous.partial}.
   Indeed, the growth assumption and Corollary \ref{coro.infinite.limit}  imply that $\Iso(M,g)$ has an infinite limit 
   set at each point,
   and we noticed the equality $\Iso^{\times}(M,g)=\Iso(M,g)$ for those manifolds.

   We will prove the completeness of $(X,g_X)$ 
  in two quite different ways, 
  according to the nature of the group $\Gamma$
  (see Sections \ref{sec.warping.constant} and \ref{sec.warping.nonconstant} below).
  
  \subsubsection{Complete Lorentz spaces of constant curvature}
  \label{sec.model.spaces}
  
  In the following, we will call $\xk$ the complete, simply connected, Lorentz manifold of constant curvature $\kappa$.
   To understand what follows, we must  recall a few fundamental facts about those spaces, and their geometry.
   
  First of all, the model for ${\bf X}_0$ is {\it Minkowski space} $\RR^{1,k}$, namely the space $\RR^{k+1}$
   endowed with the flat Lorentz metric defined by the quadratic form  $q^{1,k}(x)=-x_1^2+x_2^2+ \ldots+x_{k+1}^2$.
   
 Next, to describe $\xk$ when $\kappa>0$, we consider the space $\RR^{1,k+1}$, and the quadric $X_{\kappa}$ defined by the equation
  $q^{1,k+1}=\frac{1}{\sqrt{\kappa}}$. Inducing $q^{1,k+1}$ on $X_{\kappa}$, one gets a Lorentz metric of constant 
   sectional curvature $\kappa$. This yields the model space $\xk$.  When $\kappa = +1$, the usual name for $\xk$ is 
   {\it de Sitter space}, denoted $\dS^{1,k}$.
   
 When $ \kappa < 0$, we consider $\RR^{2,k}$, namely the space $\RR^{k+2}$ endowed with the quadratic form $q^{2,k}$.
  The quadric $X_{\kappa}$ is then defined by the equation $q^{2,k}=\frac{-1}{\sqrt{- \kappa}}$. The restriction of $q^{2,k}$
   to $X_{\kappa}$ is a complete Lorentz metric of constant sectional curvature $\kappa$. However, 
    $X_{\kappa}$ is not simply connected. To get the model space $\xk$, one has to consider {\it the universal cover} $\tilde{X}_{\kappa}$ 
     endowed with the lifted metric. We will call $\pi: \xk \to X_{\kappa}$ the covering map (this is an infinite cyclic covering).
      For $\kappa=-1$, the usual name of $\xk$ is {\it anti-de Sitter space}, denoted $\widetilde{\AdS}^{1,k}$.
      
 One  thing we will have to know about the geometry  of $\xk$, for any $\kappa$, is the notion of {\it lightlike hyperplane}.
  For Minkowski space, this is the usual notion of an affine hyperplane, on which the Minkowski metric is degenerate.
  
  When $ \kappa >0$ (resp. $\kappa <0$), we define a {\it lightlike hyperplane of $X_{\kappa}$} 
   as a connected component of the intersection $u^{\perp} \cap X_{\kappa}$, where $u \subset \RR^{1,k+1}$ 
   (resp. $u \subset \RR^{2,k}$) is any isotropic vector.  
   
   When $\kappa >0$, $X_{\kappa}=\xk$ and we have thus defined the notion of a lightlike hyperplane of $\xk$. 
   
   When $\kappa <0$, we define a {\it lightlike hyperplane of $\xk$} as a connected component of 
   the lift $\pi^{-1}(H) \subset \xk$ of any lightlike hyperplane $H \subset X_{\kappa}$. 
    Let us just mention one peculiarity of the case $\kappa<0$, that will be used later on (this is detailed in 
    \cite[p. 368]{klingler}, where the terminology ``semi-coisotropic hyperplane'' is used). 
      Let us consider $H$ a lightlike hyperplane of $\xk$ (here $\kappa<0$). Then by definition its projection $H'$ on $X_{\kappa}$
       is a lightlike hyperplane of $X_{\kappa}$. But the preimage $\pi^{-1}(H')$ has infinitely many connected component, naturally
        indexed by $\ZZ$, and $H$ is just one of them. The second point is that $\xk \setminus \pi^{-1}(H')$ also has infinitely
         many components. It turns out that only two of thosecomponents contain $H$ in their closure, and we will denote
          them by $U_H^+$ and $U_H^-$. Details about this can be found in \cite[p. 369]{klingler}.
   

  \subsection{Completeness of the factor $(X,g_X)$ when $\Gamma$ is a subgroup of $\Iso(N) \times \Iso(X)$} 
\label{sec.warping.constant}

 Mutiplying if necessary the metric $g$ by a constant, we  don't loose any generality if we take 
 the curvature $\kappa$ of $X$ equal to $0,+1$ or $-1$, what we will do from now on. 
  
  The Lorentz manifold $(X,g_X)$ has constant curvature $\kappa$, thus there exists an isometric
   immersion $\delta: X \to \xk$, as well as a {\it holonomy morphism} $\rho: \Iso(X,g_X) \to \Iso(\xk)$, such
    that $\delta \circ f= \rho(f) \circ \delta$, for every $f \in \Iso(X,g_X)$.
    Our aim is to show that under the hypotheses of Theorem \ref{thm.complete}, 
    the map $\delta$ is an isometry between $(X,g_X)$ and $\xk$.  The situation is reminiscent of 
     the following  celebrated theorem of Y. Carri\`ere (in the flat case $\kappa=0$),
     completed by B. Klingler (for any constant curvature $\kappa$).

    \begin{theoreme}\cite{carriere}, \cite{klingler}
     \label{thm.carriere-klingler}
     Let $(X,g_X)$ be a simply connected Lorentz manifold of constant sectional curvature $\kappa$. 
     Assume that there exists $\Gamma_X \subset \Iso(X,g_X)$ a discrete subgroup acting properly discontinuously
      on $X$ with compact quotient, then the developping map $\delta$ is an isometry between $(X,g_X)$ and $\xk$.
    \end{theoreme}

    In our situation,
       the projection $\Gamma_X$  of $\Gamma$ on $\Iso(X,g_X)$ does act cocompactly on $(X,g_X)$, since $\Gamma$
        acts cocompactly on $N \times X$. However the group $\Gamma_X$ is not necessarily discrete. Reading carefully the proof of
         B. Klingler in \cite{klingler}, one realizes that part of the arguments actually do not use 
         the discreteness of $\Gamma_X $, but only the cocompactness assumption. In particular,
         the begining of the proof carries over to the case where $\Gamma_X$
          is only assumed to be cocompact, until \cite[Proposition 3]{klingler}. Actually,  this proposition uses the discreteness of $\Gamma_X$
           only at the very end, in order to use cohomological dimension arguments. If we just assume cocompactness of $\Gamma_X$,
           Klingler's proof yields 
           the slightly weaker version:
      
      \begin{proposition}\cite[Proposition 3]{klingler}
       \label{prop.klingler}
       Under the asumption that there exists $\Gamma_X \subset \Iso(X,g_X)$ which acts cocompactly, then the 
       developping map $\delta$ is injective. If $\delta$ is not onto, the boundary of $\Omega:=\delta(X)$ in $\xk$ 
        is a disjoint union $H \cup P$, where $H$ is a lightlike hyperplane, and $P$ is either 
        empty, or a lightlike hyperplane.
      \end{proposition}

To conclude the proof of the completeness of $(X,g_X)$, we have to show that the conclusions of 
Proposition \ref{prop.klingler} contradict the hypotheses of Theorem \ref{thm.complete}, so that $\delta$ will be 
 a diffeomorphism.
In \cite{klingler}, arguments of cohomological dimension, and results about flat affine manifolds are used, that 
 don't clearly carry over to our situation. We thus adapt them to conclude. 

\subsubsection{The case of curvature $\kappa=\pm 1$}
\label{sec.pm1}

We refer to Section \ref{sec.model.spaces} for the notion of lightlike hyperplane of $\xk$, and the notations therein.

\begin{lemme}\cite[Sec 4, p. 370]{klingler}
 \label{lem.divergence}
 Let $H$ be a  lightlike hyperplane in $\xk$. Let $U_H^+$ and $U_H^-$ the associated components.
  Let $G_H$
  the stabilizer of $U_H^{+}$ and $U_H^{-}$ in $\Iso(\xk)$. 
  Then there exists a complete vector field $Y$ on $U_H^{+}$ (resp. on $U_H^{-}$), which is preserved by $G_H$,
   and which have constant nonzero divergence $\lambda \in \RR^*$ with respect to the restriction of $g_{\xk}$
    to $U_H^+$ (resp. $U_H^-$).
\end{lemme}

\begin{corollaire}
 \label{coro.divergence}
 Let $H$ be a  lightlike hyperplane in $\xk$, and assume that $\rho(\Gamma_X) \subset G_H$. Then 
 the image $\delta(X)$ can not contain, nor 
 be contained, in a connected component $U_H^{\pm}$. 
\end{corollaire}

\begin{proof}
 Let us call $\Omega=\delta(X)$. We thus have that $N \times X$ is identified with $N \times \Omega$. 
 We consider the vector field $Y$ given by Lemma \ref{lem.divergence}
   and we construct the vector field 
   $\tilde{Y}$ on $N \times U_H^+$ by the formula $\tilde{Y}=(0,Y) \in TN \times T\xk$. 
   If $\Omega \subset U_H^+$, then  $\tilde{Y}$ induces a vector field on $(M,g)$ with constant
   nonzero constant divergence. This is impossible.
   If $U_H^+ \subset \Omega$. Then the quotient $\Gamma \backslash (N \times U_H^+)$ is an open subset of $(M,g)$ 
   (hence has finite Lorentzian volume)
    and $\overline{Y}$ induces on it a complete vector field with nonzero constant divergence. This is again impossible.
\end{proof}

Corollary \ref{coro.divergence} allows to settle easily the case $\kappa=+1$ (actually as in \cite{klingler}).
 Two lightlike hyperplanes must intersect in de Sitter space, so that $P= \emptyset$ in Proposition 
 \ref{prop.klingler}. Then $\delta(X)=U_H$ and Corollary \ref{coro.divergence} yields a contradiction.
 
In the case $\kappa=-1$ we see if  that $P=\emptyset$  in Proposition \ref{prop.klingler}, then $\delta(X)$ 
 contains $U_H^+$ or $U_H^-$, and we get a contradiction by Corollary \ref{coro.divergence}.  If $P \not =\emptyset$, we also 
  get a contradiction as follows.

  First, assume that $P$ is parallel to $H$. It means that if $H'$ and $P'$ denote the projections of $H$ and $P$ on $X_{\kappa}$, 
  then $H'$ and $P'$ are connected components of the intersections $u^{\perp} \cap X_{\kappa}$ and $v^{\perp} \cap X_{\kappa}$
   with $u$ and $v$ isotropic and orthogonal for the form $q^{2,k}$.
   
   Then either $P$ meets $U_H^+$ (resp. $U_H^-$) and then  
  $\delta(X) \subset U_H^+$ (resp. $\delta(X) \subset U_H^-$). Or $\delta(X)$ does not meet $U_H^+$ and $U_H^-$, in which case 
   it contains $U_H^+$ or $U_H^-$.
   Whatever the case, we get a contradiction by Corollary \ref{coro.divergence}.
   
 If $P$ is not parallel to $H$, then we call $G_{H,P}$ the subgroup of $\widetilde{\OO(2,k)}$ leaving 
  the pair $(H,P)$ invariant. The holonomy group $\rho(\Gamma_X)$ is a subgroup of $G_{H,P}$. The 
   hyperplanes $H$ and $P$ project onto two hyperplanes $H'$ and $P'$ in $X_{\kappa}$. The 
   open set $\Omega=\delta(X)$, which is bounded by $H$
   and $P$, projects onto a connected component  $\Omega'$ of $X_{\kappa} \setminus \{ H' \cup P' \}$. 
   By definition of lightlike hyperplanes, there are two lightlike directions $u$ and $v$ in $\RR^{2,k}$
    such that $H'=X_{\kappa} \cap u^{\perp}$ and $P'=X_{\kappa} \cap v^{\perp}$. 
     Since $P$ and $H$ are not parallel, $u$ is not orthogonal to $v$, and 
     calling $E$ the span of $u$ and $v$, we get a decomposition $\RR^{2,k}=E \oplus E^{\perp}$. 
   If we split any $x \in \RR^{2,k}$ as $x=x_E+x_{E^{\perp}}$,
     the function $\lambda: x \mapsto |x_E|^2$ is continuous and unbounded on $\Omega'$. But then, recalling the projection 
     $\pi: \xk \to X_{\kappa}$, we can define $\tilde{\lambda}: \Omega \to \RR$ by the formula $\tilde{\lambda}(x)=\lambda(\pi(x))$.
      This is a continuous function, unbounded on $\Omega$, and moreover $G_{H,P}$-invariant. This contradicts the fact that
       $\rho(\Gamma_X) \subset G_{H,P}$ acts cocompactly on $\Omega$.

\subsubsection{The case of curvature $0$}
\label{sec.curvature0}
It remains to show that when the developping map is not onto, Proposition \ref{prop.klingler} leads to a contradiction when $(X,g_X)$
 is flat. The arguments used in Carri\`ere's work are based on former results of W. Goldman and M. Hirsh
  about flat affine manifolds (see \cite{goldman}), that we can not use straigthforwardly 
  in our product situation. This is here that we will  use for the first time our asumption that $\Iso^{\times}(M,g)$
   has an infinite limit set at every point (actually here, an infinite limit set at some point would be enough). 
  
To this aim, it seems to be useful to isolate the following incompleteness property for Lorentzian 
(actually affine) manifold.

Let $(M,g)$ such a manifold. Let $u \in TM$. We denote by $\gamma_u$ the parametrized geodesic $t \mapsto \exp(tu)$. This geodesic has 
a maximal open interval of definition $(T^-(u),T^+(u))$, with $T^-(u) < 0 < T^+(u)$
(and maybe $T^-(u)=- \infty$ or $T^+(u)=+ \infty$). We say that the direction $u$ is 
{\it uniformly incomplete in the future} (resp. in the past)
 whenever there exist a constant $C>0$ (resp. $C<0$), and a neighborhood $\mathcal V$ of $u$ in $TM$, such that for every 
 $v \in {\mathcal V}$, $T^+(v) \leq C$ (resp. $T^-(v) \geq C$).
 
 We make the following remark:
 
 \begin{lemme}
  \label{lem.incomplete}
  Let $(M,g)$ be a compact Lorentz manifold, let $x \in M$, and let 
  $\Lambda(x) \subset {\mathbb P}(T_xM)$ be the limit set of 
  $\Iso(M,g)$ at $x$. If $[u] \in \Lambda(x)$, then $u$ is neither  uniformly incomplete in the future, 
    nor in the past.
 \end{lemme}

 \begin{proof}
  If $[u] \in \Lambda(x)$, then $u \in T_xM$ is a lightlike vector which is asymptotically stable 
  for some sequence $(f_i)$
   in $\Iso(M,g)$. Namely there exists $(x_i)$ a sequence in $M$ converging to $x$, and $v_i \in T_{x_i}M$ converging to $u$ such that 
   $D_{x_i}f_i(v_i)$ remains bounded. After considering a subsequence, we may assume that 
   $y_i=f_i(x_i)$ converges to $y$. There are also sequences of orthogonal frames 
   $(e_1^{(i)}, \ldots, e_n^{(i)})$ at $x_i$, and $({\epsilon}_1^{(i)}, \ldots, {\epsilon}_n^{(i)})$ at $y_i$, converging in the bundle of frames, 
   so
    that the differential $D_{x_i}f_i$ has the following matrix in those frames:
    
    $$D_{x_i}f_i= \left(  \begin{array}{ccc} 
    \lambda_i & 0& 0\\
    0& I_{k-1} & 0\\
    0& 0& \frac{1}{\lambda_i}
    \end{array} \right)$$
    with $\lim_{i \to \infty} \lambda_i=+ \infty$.
    In particular the stability property of $u$ shows that $e_n^{(i)}$ tends to $\alpha u$, 
    for some $\alpha \in \RR^*$. Let us call $v_i=D_{x_i}f_i(e_n^{(i)})$. This is a sequence 
     of vectors in $T_{y_i}M$ which tends to $0$. As a consequence, $T^+(\pm v_i) \to + \infty$.
      Now, because $\alpha^{-1}e_n^{(i)}=\alpha^{-1}D_{y_i}f_i^{-1}(v_i)$, we get 
      $T^+(\pm \alpha^{-1}e_n^{(i)}) \ to + \infty$.
       Since $\pm \alpha^{-1}e_n^{(i)}$ tends to $\pm u$, this shows that $u$ is neither uniformly 
       incomplete  in the future, nor in the past.
 \end{proof}

Now,  for open subsets of Minkowski space,  one has the obvious lemma:  
 
\begin{lemme}
 \label{lem.plat.incomplet}
 Let $\Omega \subset \RR^{1,k}$ be an open subset, such that $\partial \Omega$ contains a lightlike hyperplane 
 $H=u^{\perp}$. Then all  directions of $T \Omega$ different of $u$,  are uniformly incomplete either in the future, or 
 in the past.
\end{lemme}

Now our standing assumption is that $\Iso^{\times}(M,g)$ has an infinite limit set at every point $x$. 
 Lifting everything to $(\tm, {\tilde g})$, we get a point ${\tilde x}=(n,z) \in \tm$ and, by Lemma \ref{lem.incomplete},
  infinitely many lightlike directions 
  in $T_{{\tilde x}}\tm$ which are neither  uniformly incomplete in the future, nor in the past. On the other hand, 
   assuming for a contradiction that $\delta: X \to \RR^{1,k}$ is not a diffeomorphism, Proposition \ref{prop.klingler}
    says that
      $(X,g_X)$ is isometric to an open subset $\Omega \subset \RR^{1,k}$
     satisfying the
     hypotheses of Lemma \ref{lem.plat.incomplet}. In particular, this lemma says that  the only directions  in $T_{{\tilde x}}\tm$ which 
     are neither  uniformly incomplete in the future, nor in the past, are of the form $v+u_0$, where $v \in T_nN$
    is any direction and $u_0 \in T_zX$ is a specific lightlike direction. Among them, only one, namely $0+u_0$, is lightlike, and
     we get a contradiction. 
 
\subsubsection{The factor $(X,g_X)$ is $\widetilde{\AdS}^{1,2}$ or Minkowski space} 
\label{sec.restriction}

Having established the completeness of the factor $(X,g_X)$ under the assumption $\Gamma \subset \Iso(N) \times \Iso(X)$, it remains to 
 check that not all spaces $\xk$ are possible for $(X,g_X)$, as announced in Theorem \ref{thm.complete}.
 
\begin{proposition}
 \label{prop.factor}
 Let $(M,g)$ be a compact Lorentzian manifold. Assume that $M$ is a quotient of $N \times \xk$ by a discrete subgroup
  of $\Iso(N) \times \Iso(\xk)$, with the dimension of $\xk \geq 3$. If $\Iso^{\times}(M,g)$ has an infinite limit set at every point of $M$, 
   then $\xk$ is either the $3$-dimensional anti-de Sitter space $\widetilde{\AdS}^{1,2}$, or Minkowski space $\RR^{1,k}$.
\end{proposition}
 
  \begin{proof}
Let $x \in M$. If the limit set of $\Iso^{\times}(M,g)$ is infinite at $x$, then $(M,g)$ admits infinitely many 
codimension one, totally geodesic, lightlike 
foliations which are transverse at $x$ (see Theorem \ref{thm.zeghib}).
 The proposition is then the content of  points $(1)$ and $(2)$ of \cite[Theorem 15.1]{zeghibisom1}. The proof can be found in 
  \cite[Sec. 15.1 and 15.2]{zeghibisom1}
\end{proof}

\subsection{Completeness of the factor $(X,g_X)$ in the general case where $\Gamma \subset \Iso(N) \times \operatorname{Homot}(X)$}
\label{sec.warping.nonconstant}

We assume here that $\Gamma$ is not included in $ \Iso(N) \times \Iso(X)$, since this situation was already handled in 
Section \ref{sec.warping.constant}.
The first observation is that a Lorentz manifold of constant sectional curvature $\kappa \not = 0$ does not admit
homothetic transformations which are not isometric. It follows that the factor $(X,g_X)$ is flat, and 
 there exists an isometric immersion $\delta: X \to \RR^{1,k}$. It is easily checked that any
  diffeomorphism between connected open subsets of $\RR^{1,k}$ which is homothetic with respect to the Minkowski metric 
   is the restriction of a global homothetic transformation of $\RR^{1,k}$.
    It thus follows that we have a morphism $$\rho: \homot(X) \to \homot(\RR^{1,k}),$$
    such that the following equivariance relation holds:
    $$ \delta(h.x)=\rho(h).\delta(x),$$
    for every $h \in \homot(X)$, and $x \in X$.
    Our aim is, as in the previous section,  to show that $\delta: X \to \RR^{1,k}$ is an isometry.

 Before doing this, let us make  a few reductions.  Theorem \ref{thm.reduction} yields a compact Lorentz submanifold $\Sigma \subset M$, which
  is locally homogeneous with semisimple isotropy, and which is preserved by a finite index subgroup of $\Iso(M,g)$. This last
   property ensures that the limit set of $\Iso(\Sigma)$ is infinite at each point of $\Sigma$. Moreover, the submanifold $\Sigma$ is actually
    obtained as a $\kiloc$-orbit (see Section \ref{sec.embed}). In particular, if $\tilde{\Sigma}$ denotes the lift of $\Sigma$ to $\tm$, then $\tilde{\Sigma}$
     is  a union of leaves $\{n \} \times X$. Thus $\tilde{\Sigma}=N' \times X$, where $N'$ is a submanifold (maybe not connected)
      of $N$. As a consequence, the universal cover of $\Sigma$ is also of the form $N_0 \times_w X$, for some simply connected
       Riemannian manifold $N_0$. It follows that replacing if necessary $M$ by $\Sigma$, we may assume in the following that $M$ is
        locally homogeneous with semisimple isotropy.

    To make our second reduction, recall the group $G$ that we introduced in Section \ref{sec.isometries}. 
 The group $G$ acts transitively  and
 isometrically on the Riemannian manifold $(N,g_{N})$ (Lemma \ref{lem.isometries}). The same is true for its identity component $G^o$.
  We observe that $G^o$ has finite index in $G$.  Indeed, if 
 $K$ is the stabilizer, in $G$,  of a point $n \in N$, then the isotropy representation identifies $K$ as a compact subgroup of 
 $\OO(n-k)$.
  In particular $K$ has finitely many connected components, and because $G/K=N$ is connected, the group $G$ has 
  finitely many connected components too.
  
  Let us denote in the following ${\widetilde{\Iso}}^{\times}(M,g)$ the subgroup of $\Iso(\tm, \tilde{g})$ comprising all lifts
   to $\tm$ of isometries in $\Iso^{\times}(M,g)$.
  Since $G$ has finitely many connected components,  there exist  finite index subgroups $\Gamma' \subset \Gamma$, and 
  ${\widetilde{\Iso}}^{\times}(M,g)' \subset \widetilde{\Iso}^{\times}(M,g)$ which are both contained in $G^o \times \homot(X)$
  (see Lemma \ref{lem.isometries}). Let $g_G$ be any left-invariant Riemannian metric on $G^o$. Let us 
  denote by $\pi: G^o \to N=G^o/K^o$ the natural
   projection and let $\tilde{w}: G^o \to \RR_+^*$ be the function defined by $\tilde{w}(g)=w(\pi(g))$.
    Then ${\widetilde{\Iso}}^{\times}(M,g)'$ acts isometrically on the warped-product 
    $G^o \times_{\tw} X$
    (see Lemma \ref{lem.isometries} point $(1)$).  Let us call $M'$ the quotient of $G^o \times X$ by $\Gamma'$.
     This manifold is compact because it fibers over $M$ with compact fibers. The group ${\widetilde{\Iso}}^{\times}(M,g)'$ 
     surjects on a 
     finite index subgroup $\Iso^{\times}(M,g)' \subset \Iso^{\times}(M,g)$. The limit set 
     of $\Iso^{\times}(M,g)'$ is thus infinite at every point, and because the fibers of 
      the fibration $M' \to M$ are Riemannian, the limit set of ${\widetilde{\Iso}}^{\times}(M,g)'/ \Gamma'$ on $M'$
       is infinite. Since ${\widetilde{\Iso}}^{\times}(M,g)'/ \Gamma'$ is a subgroup of $\Iso^{\times}(M',g')$, we conclude that 
        the limit set of $\Iso^{\times}(M',g')$ is also infinite at each point. 
        
  All those remarks show that we don't loose any generality if we assume that $\tm=N \times_w X$ is actually the space 
   $G^o \times_{\tw} X$, and moreover $\widetilde{\Iso}^{\times}(M,g) \subset G^o \times \operatorname{Homot}(X)$. We will make those 
   asumptions from now on. 
 
 Recall (see Section \ref{sec.isometries}) that the  group $G^o$ admits a  continuous homomorphism $\lambda: G^o \to \RR_+^*$, satisfying 
 $\tw(g_1g_2)=\lambda(g_1)\tw(g_2)$ for every $g_1,g_2$ in $G^o$. This homomorphism is nontrivial. Indeed triviality
  of $\lambda$ would mean that $\tw$ is constant. Sticking to the notations of Lemma \ref{lem.isometries}, $\tw$ constant 
  implies that the group
 $G$ coincides with $\Iso(N,g_N)$, and  $\lambda(f)=1$ for every $f \in \Iso(N,g_N)$. It follows 
  from Lemma \ref{lem.isometries} and Remark \ref{rem.general} that $\Iso^{\times}(\tm, \tilde{g})$ 
  coincides with the product $ \Iso(N,g_N) \times \Iso(X,g_X)$, and we are then in the realm of Section \ref{sec.warping.constant}. 
   A nontrivial $\lambda$ yields a decomposition $G^o=AH$, where 
 $H=\Ker \lambda$, and $A$ is a $1$-parameter group $\{a^s\}$ satisfying for every $g \in G^o$ 
 $\tw(a^sg)=e^{\alpha t}\tw(g)$, with $\alpha>0$. Let us denote by $\tw_0$ the value taken by $\tw$ on 
 the subgroup $H$.
  Let $g \in G^o$ that we write $g=a^th$, for some $h \in H$. We get $\tw(g)=e^{\alpha t}\tw_0$.
   We also compute 
  $\tw(ga^s)=\tw(a^{t+s}a^{-s}ha^s)=e^{\alpha(t+s)}\tw_0$, the last equality holding 
  because $H$ is normalized by $A$.
   We end up with the relation 
   \begin{equation}
    \label{eq.multdroite}
    \tw(ga^s)=e^{\alpha s}\tw(g).
   \end{equation}

The  $1$-parameter group of transformations $\tilde{\psi}^s: (g,x) \in G^o \times X \mapsto (ga^s,x)$ commutes 
 with the action of $G^o \times \homot(X)$. In particular, it commutes with the action of $\widetilde{\Iso}(M,g)$ 
 (see point $(2)$ of Lemma \ref{lem.isometries}), hence  
induces a flow $\psi^s$ on $M$ which commutes with $\Iso(M,g)$. Let us consider $z_0 \in M$ a recurrent point (in the future) 
 for the flow $\psi^s$. Such a point exists since $M$ is compact. Let  $(g_0,x_0) \in G^o \times X$ projecting
 on $z_0$.
   Saying that $z_0$ is recurrent in the future means that there exist a sequence $s_i \to \infty$, and a sequence 
   $(\alpha_i,\beta_i)$ in $\Gamma \subset G^o \times \homot(X)$ such that $g_i=\alpha_i.g_0.a^{s_i}$ tends to $g_0$, 
    and 
   $x_i=\beta_i.x_0$ tends to $x_0$.  By equation (\ref{eq.multdroite}), we see that 
    $\tw(\alpha_ig_0a^{s_i})=e^{\alpha s_i}\lambda(\alpha_i)\tw(g_0)$. Since this quantity must converge to
    $\tw(g_0)$
     we get that $\lambda(\alpha_i)\underset{ + \infty}{\sim}e^{- \alpha s_i}$. By point $(1)$ of Lemma 
     \ref{lem.isometries}, $\beta_i$ is an homothetic transformation of $X$, with dilatation 
     $\lambda_i=\lambda(\alpha_i)^{-1}$. In particular,  $\lambda_i \underset{ + \infty}{\sim} e^{\alpha s_i}$. 
     
    Then the holonomy $\rho(\beta_i)$ satisfies $\rho(\beta_i).\delta(x_0) \to \delta(x_0)$,
    and $\rho(\beta_i)=\lambda_i A_i +T_i$, for $(A_i)$ a sequence in $\OO(1,k)$, and $(T_i)$ a 
    sequence in $\RR^{k+1}$.  
    \begin{lemme}
     \label{lem.bounded.sequence}
     The sequence $(A_i)$ is  bounded.
    \end{lemme}

    \begin{proof}
     We consider the foliation
      $\tilde{\calf}$ on $G^o \times X$, whose leaves are the sets  $\{ g \} \times X$. It induces a foliation
       $\calf$ on $M$, the leaves of which are Lorentzian. 
       For each $x \in M$, we denote by $\Lambda^{\times}(x)$ the limit set of the group $\Iso^{\times}(M,g)$.
        We observe that if $x \in M$, and $[u] \in \lig^{\times}(x)$,
        then $u$ is tangent to $\calf$. Now the flow  $\psi^s$ maps each leaf of $\calf$ 
        conformally to another leaf. In particular, $D\psi^s$ preserves the set of lightlike vectors tangent
      to $\calf$. 
      We now use our asumption that $\Lambda^{\times}(x)$ is infinite for every  $x \in M$. 
      In 
       particular there are three distinct directions $[u_1],[u_2],[u_3]$ belonging to $\lig(z_0)$.
        Associated to those directions are three sequences $(f_i^{(1)}), (f_i^{(2)}), (f_i^{(3)})$
         going to infinity in $\Iso(M,g)$, such that $u_1^{\perp}$, $u_2^{\perp}$
          and $u_3^{\perp}$ coincide with the asymptotically stable distributions (see Section 
          \ref{sec.semi-continuity}) $AS(f_i^{(1)})(z_0)$, $AS(f_i^{(2)})(z_0)$ and $AS(f_i^{(3)})(z_0)$.
           Theorem \ref{thm.zeghib} yields three Lipschitz fields of lightlike directions $x \mapsto \xi_j(x)$,
            $j=1,2,3$, such that $\xi_j(x)^{\perp}=AS(f_i^{(j)})(x)$ for every $x \in M$, $j=1,2,3$.
      Now, we already observed that $\psi^s$ commutes with $\Iso(M,g)$. In particular, $D\psi^s$
       maps $AS(f_i^{(j)})$ to itself, for $j=1,2,3$. Because $\psi^s$ maps lightlike directions tangent to 
       $\calf$
        to lightlike directions tangent to $\calf$, we moreover infer that $D\psi^s(\xi_j)=\xi_j$.
        We lift the fields of directions $\xi_j$ to Lipschitz fields of directions 
        $\tilde{\xi_j}$ on $G^o \times X$. They are tangent to the leaves $\{g \} \times X$,  are
         $\Gamma$-invariant, and also invariant by $\tilde{\psi^s}$. Projecting $\xi_j(g_0,x_0)$ 
         on the factor $\{0 \} \times T_{x_0}X \subset T_{(g_0,x_0)}M$, and 
         taking the images by $D_{(g_0,x_0)}\delta$, we get
          three distinct lightlike 
          directions $\overline{\xi}_1$, $\overline{\xi}_2$ and $\overline{\xi}_3$ at $\delta(x_0)$, such that 
          $A_i.\overline{\xi}_j \to \overline{\xi}_j$, for $j=1,2,3$.
           Because $A_i$ is a sequence in $\OO(1,k)$, this forces $A_i$ to stay in a compact subset.
    \end{proof}

%
%
%
%
    
\begin{lemme}
 \label{lem.assiettes}
 Let $U$ and $V$ be open subsets of $X$ on which $\delta$ is injective. Assume that $U \cap V \not = \emptyset$, and that $\overline{\delta(U)} \subsetneq \delta(V)$.
  Then $U \subsetneq V$.
\end{lemme}

We pick $U$ an open subset containing $x_0$ such that $\delta$ restricted to $U$ is injective. 
We know that for every $i$ large enough $\beta_i(U) \cap U \not = \emptyset$, and $\delta$ is injective on 
 $\beta_i(U)$ as well, because of the equivariance relation $\delta \circ \beta_i=\rho(\beta_i) \circ \delta$.
 Now $\rho(\beta_{n_i})(U)$ is an increasing sequence of open subsets exhausting $\RR^{1,k}$, for a suitable 
  subsequence $n_i$. It follows from Lemma \ref{lem.assiettes} that  $\beta_{n_i}(U)$ is an 
  increasing sequence of open subsets in $X$.
   The union $\bigcup_{k \in \NN} \beta_{n_i}(U)$ is an open subset $\Omega \subset X$, which is mapped 
    diffeomorphically 
    by $\delta$ onto $\RR^{1,k}$. We then must have $\Omega=X$, because otherwise, looking at point on $\partial \Omega$, we would
    check that $\delta$ could not be  injective on $\Omega$. This finishes the proof of the completeness of $(X,g_X)$, and that of Theorem 
    \ref{thm.complete}.

\section{Proof of Theorem \ref{thm.main} and conclusion}
\label{sec.proof-main}

This section is devoted to the proof of Theorem \ref{thm.main}. The proof will be done in two steps. First, we will deal with two
particular cases, namely the manifolds $(M,g)$ which are quotients of $N \times \widetilde{\AdS}^{1,2}$, and those which are locally homogeneous 
with semisimple isotropy (see Sections \ref{sec.factor.ads} and 
 \ref{sec.factor.flat} below). Thanks to all the work done so far, and a last important 
 extension result (see Theorem \ref{thm.sl2}) we will be able to derive the general statement from those two particular cases.

%
%
%
 
 \subsection{Theorem \ref{thm.main} for quotients of $N \times \widetilde{\AdS}^{1,2}$}
 \label{sec.factor.ads}
 
 We recall the notation $\Iso^{\times}$ introduced in Section \ref{sec.completeness.issues}.
 We are going to  prove:
   
 \begin{proposition}
  \label{prop.ads}
  Let $(M,g)$ be a compact Lorentz manifold. 
   Assume that there exists $N$ a Riemannian manifold, such that $M$ is a quotient of 
    a warped product $N \times_w \widetilde{\AdS}^{1,2}$
   by a discrete subgroup 
$\Gamma \subset \Iso(N) \times \Iso(\widetilde{\AdS}^{1,2})$. Assume that the limit set of $\Iso^{\times}(M,g)$
is infinite at each point.
 Then $\Iso(M,g)$ is virtually 
an extension of $\PSL(2,\RR)$
by a compact Lie group.
 \end{proposition}
 
 Observe that the hypothesis involves $\Iso^{\times}(M,g)$, but the conclusion is about the full isometry group $\Iso(M,g)$.
 
 The proof is   discussed in \cite[Section 15.2]{zeghibisom1}. The model for $3$-dimensional anti-de Sitter space 
   $\widetilde{\mathbb{ADS}}^{1,2}$ is the Lie group $\widetilde{\PSL(2,\RR)}$ endowed with the left Lorentzian metric 
   obtained from the Killing form on $\mathfrak{sl}(2,\RR)$. This metric $g_{AdS}$ turns out to be bi-invariant, and
    we get an isometric action of the product $\widetilde{\PSL(2,\RR)} \times \widetilde{\PSL(2,\RR)}$ 
    (by left and right translations). The action is not faithful because $\widetilde{\PSL(2,\RR)} $ has a nontrivial center $Z$, so
    that the 
   isometry group of $\widetilde{\mathbb{ADS}}^{1,2}$ is up to finite index 
   $(\widetilde{\PSL(2,\RR)} \times \widetilde{\PSL(2,\RR)})/Z$. As already mentioned, conformal transformations of
   $\widetilde{\mathbb{ADS}}^{1,2}$ are isometric. It follows from Proposition 
   \ref{prop.homogeneous.partial} that $(M,g)$ is the quotient of $N \times \widetilde{\mathbb{ADS}}^{1,2}$
    by a discrete subgroup 
    $\Gamma \subset \Iso(N) \times \widetilde{\PSL(2,\RR)} \times \widetilde{\PSL(2,\RR)}$. 
    The projection of 
    this group on each factor is denoted by $\Gamma_N$, $\Gamma_L$ (left) and $\Gamma_R$ (right).
    
    The splitting $N \times \widetilde{\AdS}^{1,2}$ induces two transverse foliations $\calf^{\perp}$ and $\calf$ on $M$, which are
    preserved by $\Iso^{\times}(M,g)$. In particular, each direction $[u]$ belonging to the limit set of $\Iso^{\times}(M,g)$
     must be tangent to $\calf$. This remark, together with the hypothesis that the limit set of $\Iso^{\times}(M,g)$ is infinite 
      yields, by Theorem \ref{thm.zeghib}, infinitely distinct codimension $1$ lightlike geodesic foliations on $M$, whose lightlike direction is 
      tangent to $\calf$.
      As explained in \cite[Section 15.2]{zeghibisom1}, this forces $\Gamma_L$ or $\Gamma_R$ to be trivial.
       Let say that $\Gamma_R$ is trivial. Then the action of $\widetilde{\PSL(2,\RR)}$ by 
        right-multiplication on $N \times_w \widetilde{\PSL(2,\RR)}$ induces a 
        nontrivial isometric action of  $\widetilde{\PSL(2,\RR)}$ on $(M,g)$.
        Under these circomstances, the action is not faithfull, but one gets that 
        $\Iso^o(M,g)$ is finitely covered by $\PSL(2,\RR)^{(m)} \times K$, for $K$ a connected compact Lie group 
        (see \cite{gromov}, \cite{adams.stuck}, \cite{zeghibidentity}). Here 
        $\PSL(2,\RR)^{(m)}$ 
         denotes the $m$-fold cover of $\PSL(2,\RR)$. In particular, $\Iso^o(M,g)$ is a compact extension of $\PSL(2,\RR)$.
         
         The action  of $\PSL(2,\RR)^{(m)}$ on $(M,g)$ is locally free 
         (see \cite[Th 5.4.A]{gromov}), and its orbits 
         have Lorentz signature. As  observed in \cite[Corollary 6.2]{zeghibpic},  this implies that  $\Iso(M,g)$
          has finitely many connected components. Proposition \ref{prop.ads} follows. We observe that  
          we are precisely in the first case  of Theorem \ref{thm.main}.

 \subsection{Theorem \ref{thm.main} when $(M,g)$ is locally homogeneous with semisimple isotropy}
  \label{sec.factor.flat}
  
 \begin{proposition}
  \label{prop.minkowski}
  Let $(M,g)$ be a compact, locally homogeneous, Lorentz manifold.
   Assume that the isotropy algebra $\liei$ is isomorphic to $\oo(1,k) \oplus \oo(m)$, with $k \geq 2$. 
   Assume that the limit set of $\Iso(M,g)$ is infinite at each point.
Then the conclusions of Theorem \ref{thm.main} hold. 
 \end{proposition}
 
By Proposition \ref{prop.homogeneous.partial}, the universal cover $(\tm, \tilde{g})$ is isometric to a warped product 
$N \times_w \widetilde{\AdS}^{1,2}$, or $N \times_w \RR^{1,k}$, where $N$ is a homogeneous Riemannian manifold. 
 Moreover, it follows from point $(2)$ of 
Proposition \ref{prop.homogeneous.partial} that $\Iso(\tm, {\tilde{g}}) \subset \Iso(N) \times \Iso(\widetilde{\AdS}^{1,2})$
 (resp. $\Iso(\tm, {\tilde{g}}) \subset \Iso(N) \times \Homot(\RR^{1,k})$), and $\Iso^{\times}(M,g)=\Iso(M,g)$.
 
 When $(\tm, \tilde{g})$ is isometric to 
$N \times_w \widetilde{\AdS}^{1,2}$, we can directly apply Proposition \ref{prop.ads}. The group $\Iso(M,g)$ is then virtually a compact extension of 
$\PSL(2, \RR)$, and we are in the first case of Theorem \ref{thm.main}.

We now prove Proposition \ref{prop.minkowski} when $(\tm, \tilde{g})$ is isometric to  $N \times_w \RR^{1,k}$.

\subsubsection{Getting a proper homomorphism $\rho: \Iso(M,g) \to \PO(1,d)$}  
The arguments until Section \ref{sec.concon} are essentially discussed in \cite[Section 15.3]{zeghibisom1}. 
The manifold $M$ is  a quotient of $N \times \RR^{1,k}$ by a discrete subgroup $\Gamma \subset \Iso(N) \times \Homot(\RR^{1,k})$.
We call $\Gamma_N$ and $\Gamma_X$ the projections of $\Gamma$
 on each factor. As explained above, since the limit set  is infinite at each point, $M$ admits infinitely many 
  codimension one, totally geodesic lightlike foliations. The lift of each of those foliations to $\tilde{M}$ 
   is preserved by $\Gamma$. Now, codimension one, totally geodesic, lightlike foliations of $N \times_w \RR^{1,k}$
    are  easy to describe (see \cite[Section 15.3]{zeghibisom1}). They are parametrized by lightlike directions $[u]$
     in $\RR^{1,k}$. The leaves of $\calf_u$ determined by such a direction,  are  products $N \times H_u$, where $H_u$ is an affine hyperplane of $\RR^{1,k}$, parallel to  
    $u^{\perp}$. Let us call $E$ the span of all  lightlike directions $u$ which give rise to a foliation $\calf_u$, 
    coming from the asymptotically stable foliation of a sequence $(f_k)$ in $\Iso(M,g)$ 
    (see Theorem \ref{thm.zeghib}). Because the limit set 
    $\Lambda(x)$ has more than $3$ elements at each $x \in M$, the space $E \subset \RR^{1,k}$ is 
    Lorentz of dimension $d+1$, 
    with $d \geq 2$. This yields an orthogonal  splitting 
    $\RR^{1,k}=E \oplus F$, with $F$ Riemannian. Let us call $\widetilde{\Iso}(M,g)$ the group comprising all 
    lifts to $\tm$ of elements  of $\Iso(M,g)$. Then $\widetilde{\Iso}(M,g)$
     preserves the splitting $\tm=N \times E \times F$, and $\widetilde{\Iso}(M,g) \subset \Iso(N) \times \Homot(E) \times \Homot(F)$.
      The projection on the second factor yields a morphism 
      $\pi_E: \widetilde{\Iso}(M,g) \to \Homot(E)\simeq (\RR_*^+ \times \OO(1,d)) \ltimes \RR^{d+1} $.
      Postcomposing with the natural morphism $(\RR_*^+ \times \OO(1,d)) \ltimes \RR^{d+1} \to \PO(1,d)$, we get 
    a homomorphism $\tilde{\rho}:  \widetilde{\Iso}(M,g) \to \PO(1,d) $.
       The projection of the group  $\Gamma$  on each factors will be denoted $\Gamma_N$, 
       $\Gamma_E$ and $\Gamma_F$. Because each lightlike direction in $\Lambda(x)$ corresponds to a $\Gamma_E$-invariant
       direction in $E$, we see that  elements of $\Gamma_E$ have a linear part acting by similarities: 
       $$x \mapsto \lambda x,$$
       for $\lambda \in \RR^*$.

In other words, $\Gamma \subset \Ker \tilde{\rho}$, and we finally get a well-defined morphism
    $$ \rho: \Iso(M,g) =  \widetilde{\Iso}(M,g)/ \Gamma \to \PO(1,d).$$
 
 \begin{lemme}
  \label{lem.properness}
  The homomorphism $\rho: \Iso(M,g)  \to \PO(1,d)$ is a proper map.
 \end{lemme}

 \begin{proof}
 We want to show that the map $\rho$ is proper, namely the preimage of every bounded sequence 
 is a bounded sequence. The splitting $\tm=N \times E \times F$ is preserved by $\Gamma$, hence induces
  an orthogonal 
  splitting $TM=TM_N \oplus TM_E \oplus TM_F$.  For each $x$, the space 
  $T_xM_E$ is Lorentz, while $T_xM_N \oplus TM_F$ is Riemannian. Let  $(f_k)$  be a sequence of $\Iso(M,g)$.
   The condition $\rho(f_k)$ bounded is easily seen to imply that ${Df_k}_{|TM_E}$ is bounded. But since $TM_E$
    has Lorentz signature, it shows that the $1$-jet of $f_k$ is bounded. Because $\Iso(M,g)$ acts properly on the orthonormal 
    bundle $\hm$ (see Section \ref{sec.topology}), this implies $(f_k)$ bounded in $\Iso(M,g)$.
  \end{proof}  
 
 \subsubsection{Conclusion under the assumption that $\Iso^o(M,g)$ is compact}
 
 Sticking to the previous notations, we consider the homomorphism $\rho: \Iso(M,g)  \to \PO(1,d)$, and we have shown that 
 it is a proper map. It means that $\rho( \Iso(M,g))=H$ is a closed subgroup of $\PO(1,d)$, and $H^o=\rho( \Iso(M,g)^o)$ is a compact, normal 
  subgroup of $H$.  In particular, the set of fixed point of the action of $H^o$ on $\HH^d$ is nonempty.
   This set of fixed points  $\operatorname{Fix}(H^o)$ is a totally geodesic submanifold of $\HH^d$, isometric to some $\HH^{d'}$.
    Because $H$ normalizes $H^o$, it acts on $\operatorname{Fix}(H^o)$, yielding a  morphism 
    $\rho': H \to \Iso(\HH^{d'}) \simeq \PO(1,d)$, 
     which is proper. The image of $\rho'$ is discrete since $H^o$ does not act on $\operatorname{Fix}(H^o)$. 
     We are then in the second case of Theorem \ref{thm.main}.
 
 \subsubsection{Compactness of the identity component $\Iso^o(M,g)$}
 \label{sec.concon}
 It remains to prove the compactness of $\Iso^o(M,g)$.
 Let us introduce a bit of notations. We call $Z$ the centralizer of $\Gamma_E$ in $\Homot(E)$.
  This is an algebraic group, with Lie algebra $\liez$. We will denote by $Z_L$ the projection of $Z \subset \Homot(E)$ on $\PO(1,d)$, 
   and $Z_O$ the projection of $Z \cap \OO(1,d)$ on $\PO(1,d)$.

 \begin{lemme}
  \label{lem.z}
  \begin{enumerate}
   \item There is an action of $Z$ on $M$, which is an isometric action  in restriction to $Z \cap \Iso(E)$.
   \item One has the inclusions
   $$ Z_O \subset \Iso(M,g),$$
   $$ \rho(\Iso^0(M,g)) \subset Z_L,$$
   and
   $$\rho(\Iso(M,g)) \subset \operatorname{Nor}_{\OO(1,d)}(Z_L).$$
  \end{enumerate}

 \end{lemme}
   
   In the statement, $\operatorname{Nor}_{\OO(1,d)}(Z_L)$ denotes the normalizer of $Z_L$ in $\OO(1,d)$.

 \begin{proof}
  For every $h \in Z$, one defines $\tilde{h} \in \Iso(N) \times \Homot(E) \times \Homot(F)$
   by the formula $\tilde{h}(n,x,y)=(n,h(x),y)$ for every  $(n,x,y) \in N \times E \times F$. Obviously, $\tilde{h}$
    centralizes $\Gamma$, hence induces a diffeomorphism on $M$. If moreover $h \in Z \cap \Iso(E)$, then $\tilde{h}$
     acts isometrically on $(\tm,\tilde{g})$, and the induced action on $M$ is isometric. This proves point $1)$.
     
 As for point $2)$,  $ Z_O \subset \Iso(M,g)$ is a direct consequence of point $1)$. 
  Every flow $f^t$ in $\Iso^o(M,g)$ lifts to $\tilde{f}^t \in \Iso(\tm, \tilde{g})$ centralizing $\Gamma$. The component $\tilde{f}_E^t$
   on $\Homot(E,g)$ belongs to $Z$, and the definition of $\rho$ yields $\rho(f^t) \in Z_L$.
   
   Finally, every element $\tilde{f}$ of $\widetilde{\Iso(M,g)}$ normalizes $\Gamma$. It follows that the component $\tilde{f}_E$
    on $\Homot(E)$ normalizes $Z$, hence the last inclusion $\rho(\Iso(M,g)) \subset \operatorname{Nor}_{\OO(1,d)}(Z_L).$
    
 \end{proof}
 
  We observed in the proof that for every $\tilde{f} \in \widetilde{\Iso(M,g)}$, the component $\tilde{f}_E$ normalizes $Z$,
   hence induces an automorphism of $\liez$. Since $Z$ centralizes $\Gamma$, this automorphism is trivial when $\tilde{f} \in \Gamma$.
    We thus inherits a well-defined representation $\zeta: \Iso(M,g) \to \liez$.
  
The first point of Lemma \ref{lem.z} shows that each  $\xi \in \liez$ yields a vector field $X_{\xi}$ on $M$.
 The very definition of the representation $\zeta$ lead to the tautological but useful relation:
   \begin{equation}
    \label{eq.tautological}
     f_* X_{\xi}=X_{\zeta(f)\xi}
   \end{equation}

 We are now in position to prove:
 
 \begin{lemme}
  \label{lem.just.translations}
  If  the group $\Gamma_E \subset \Homot(E)$ does not contain only translations, then $\Iso^o(M,g)$ is compact.
 \end{lemme}

 \begin{proof}
  We already observed that elements $\gamma \in \Gamma_E$ are of the form $x \mapsto \lambda_{\gamma} x + T_{\gamma}$, for 
  some $\lambda \in \RR^*$. We assume, for a contradiction, that some element $\gamma$ satisfies $\lambda_{\gamma} \not = 1$.
   Conjugating everything in $\Iso(\tm)$, we may assume $T_{\gamma}=0$. Then, it is clear that 
   $Z \subset \RR \times \OO(1,d) \subset \Homot(E)$. Let us recall the group $Z_O$, and its Lie algebra $\liez_O$.
    The inclusion $Z \subset \RR \times \OO(1,d)$ forces $\liez_O$ to be $\zeta(\Iso(M,g))$-invariant.
    By the very definitions of the representations 
   $\zeta$ and $\rho$, we  infer that if 
    $\xi \in \liez_O$, 
    $$\zeta(f).\xi=\Ad(\rho(f)).\xi.$$
    Here $\Ad$ is the adjoint representation of $\PO(1,d)$ on its Lie algebra.
    
   The group $Z$ is real algebraic, hence can be decomposed as a semi-direct product
  $(S.T) \ltimes U$, where $S$ is semisimple, $T$ is a torus, and $U$ a unipotent subgroup.
  
  We observe that $S$ and $U$ are included in $\OO(1,d)$. We first claim that $S$ is compact. If not, it would contain 
  a subgroup $S_1$ isomorphic to $\SO(1,2)$, and by  Lemma \ref{lem.z}, $S_1$ would act isometrically on $M$.
   But looking carefully at the action of $S_1$ on $E$, we see that its orbits have dimension $0$ or $2$. 
    The same property should hold
   for the isometric action of $S_1$ on $M$,
    but we already mentioned (see \cite[Thm. 5.4.A]{gromov}) that isometric actions of $\SO(1,2)$ on compact Lorentz manifolds must be locally free, yielding a contradiction.
    
    We now check that $U$ is trivial. First, observe that $U$ is the unipotent radical of the algebraic group $Z \cap \OO(1,d)$.
    Because we already observed that $\rho(\Iso(M,g))$ normalized $\liez_O$, it must normalize he Lie algebra $\lieu$.
  
  We begin our discussion with the case where $U$ is nontrivial. Observe that then $U \subset \OO(1,d)$.
   The normalizer of $\lieu$
   in $\PO(1,d)$ is a group of the form $(\RR_+^* \times K) \ltimes U_{max}$, where $K$ is compact, $U_{max}$
    is a maximal unipotent of $\PO(1,d)$ (isomorphic to $\RR^{d-1}$), and 
    $\RR_+^*$ acts on $U_{max} \simeq \RR^{d-1}$ by homothetic transformations 
    $u \mapsto \alpha  u$, $\alpha >0$.  The group $\rho(\Iso(M,g))$ normalizes $\lieu$, hence $\rho(\Iso(M,g)) \subset 
    (\RR_+^* \times K) \ltimes U_{max}$. 
    
     If actually $\rho(\Iso(M,g)) \subset K \ltimes U_{max}$, then every compactly generated 
     subgroup of $\Iso(M,g)$ must have polynomial growth (because $\rho$ is proper), contradicting 
     our hypothesis. If $\rho(\Iso(M,g)) \not \subset K \ltimes U_{max}$, we get $f \in \Iso(M,g)$, 
     and $\xi \in \lieu$, such that
      $\Ad(\rho(f^k))\xi \to 0$ as $k \to + \infty$. On the lightcone through the origin of $E$, the orbits of the 
      flow $\{ e^{t\xi} \}$ are spacelike (except at the origin).  It means that the vector field $X_{\xi}$ (see the discussion 
      after Lemma \ref{lem.z}) is spacelike
       on an open subset $\Omega$ of $M$. But $D_yf^k(X_{\xi}(y)) \to 0$ as $k \to + \infty$ by the relation
       (\ref{eq.tautological}). When $y \in \Omega$, this contradicts the fact that $f^k$ are isometries.
       
    The previous discussion shows that $U$ is trivial and $S$ is compact. We now look at the torus $T$. It  may be written 
    as a product $T_s \times T_e$, where elements of $T_s$ are $\RR$-split and 
       those of $T_e$ are diagonalisable over $\CC$, with eigenvalues of modulus $1$. 
       The projection $Z_L$ of $Z$ on $\PO(1,d)$ is then a product $T_s' \times K$, where $K$ is compact and $T_s'$ 
       is trivial, or a 1-dimensional $\RR$-split torus in $\PO(1,d)$. The normalizer of $Z_L$ in $\PO(1,d)$ must normalize $T_s'$.
       Now if   $T_s'$ is nontrivial, its normalizer in $\PO(1,d)$ is  a group of the form $T_s' \times K'$, where $K'$ is compact. 
        The inclusion $\rho(\Iso(M,g)) \subset \operatorname{Nor}(Z_L)$ proved in Lemma \ref{lem.z}
         would imply  $\rho(\Iso(M,g)) \subset T_s' \times K'$. Again, this forces every closed compactly generated subgroup of $\Iso(M,g)$
          to have polynomial (actually linear) growth: Contradiction.
      We conclude that $T_s'$ is trivial, hence $Z_L$ is compact. But by Lemma \ref{lem.z}, $\rho(\Iso^o(M,g)) \subset Z_L$.
       Because $\rho$ is proper, we conclude that $\Iso^o(M,g)$ is compact.

 \end{proof}

 The compactness of $\Iso^o(M,g)$ will follow from Lemma \ref{lem.just.translations} and the following
 \begin{lemme}
  If  the group $\Gamma_E \subset \Homot(E)$  contains only translations, then $\Iso^o(M,g)$ is compact.
 \end{lemme}

 \begin{proof}
  If $\Gamma_E$ comprises only translations in $\Homot(E)$. Then, all 
 translations of $\Homot(E)$ commute with $\Gamma_E$, hence are contained in $Z$. By Lemma \ref{lem.z} this induces
 an isometric action of $\RR^{d+1}$ on $M$, which is locally free.
  Obviously, elements of this action stay in $\Ker \rho$, which is a compact Lie group. Let us consider $\liek$ the Lie algebra of $\Ker \rho$.
   it splits as a sum $\liea \oplus \liem$, where $\liea$ is abelian and $\liem$ is the 
   Lie algebra of a compact semisimple group. The previous remark shows that $\liea$ contains $\RR^{d+1}$.
    It thus integrates into a torus $\TT \subset \Iso^o(M,g)$, which is normalized by $\Iso(M,g)$. It is thus centralized by 
    $\Iso^o(M,g)$. There are timelike translations in $\Homot(E)$, thus there exists a Killing field $Y$
     in $\liea$ which is everywhere timelike on $M$, and commutes with $\Iso^o(M,g)$. The vector field
      $Y$ yields a reduction of the bundle $\hm$ to a subbundle $\hm'$ with compact structure group.
       Hence $\Iso^o(M,g)$ preserves the compact subset $\hm' \subset \hm$. Since its action 
       on $\hm$ is proper, $\Iso^o(M,g)$ is compact.
 \end{proof}

 \subsection{Proof of Theorem \ref{thm.main} in full generality}
 \label{sec.ful.generality}
 
 We are now ready to prove Theorem \ref{thm.main}. By hypothesis, $(M,g)$ is a $(n+1)$-dimensional compact Lorentz manifold, $n \geq 2$.
  The group $\Iso(M,g)$ is assumed to have a closed, compactly generated subgroup with exponential growth. By Theorem 
  \ref{thm.reduction}, there exists a compact, locally homogeneous Lorentz submanifold $\Sigma \subset M$, a finite index subgroup
   $\Iso'(M,g)$ leaving $\Sigma$ invariant, and a proper homomorphism $\rho: \Iso'(M,g) \to \Iso(\Sigma,g)$. Moreover, still 
    by Theorem \ref{thm.reduction}, $(\Sigma, g)$ has semisimple isotropy, and $\Iso(\Sigma,g)$ contains a closed, compactly
    generated subgroup of exponential growth. 
    
%
      
      We apply Proposition \ref{prop.minkowski} to $(\Sigma,g)$. Two cases may then occur. 
       In the first case (which corresponds to the second case in Theorem \ref{thm.main} for the group $\Iso(\Sigma,g)$), 
        $\Iso(\Sigma, g)$ is virtually
       a compact extension of a discrete subgroup $\Lambda \subset \OO(1,d)$, $2 \leq d \leq n$.  The proper homomorphism
        $\rho: \Iso'(M,g) \to \Iso(\Sigma,g)$ thus shows that $\Iso(M,g)$ is also virtually a compact extension of some subgroup of 
         $\Lambda$. We are thus in the second case of Theorem \ref{thm.main} for the group $\Iso(M,g)$.
         
   In the second case,  there exists an epimorphism $\rho': \Iso(\Sigma,g) \to \PSL(2,\RR)$, with compact kernel, yielding 
    a proper homomorphism 
     $\rho' \circ \rho: \Iso'(M,g) \to \PSL(2,\RR)$. We are not done, because we don't know if this homomorphism is onto, and this is the last 
     difficulty we have to overcome.
      At this stage, we just get that $\Iso(M,g)$ is virtually a compact extension of a closed
     subgroup $H$ of 
      $\PSL(2,\RR)$. Moreover, this closed subgroup must have exponential growth. Considering  
      a finite index subgroup if necessary,  there are  only
       four possibilities.
       \begin{enumerate}[(i)]
        \item The group $H$ is $\PSL(2,\RR)$.
        \item The group $H$ is a nonelementary discrete subgroup of $\PSL(2,\RR)$.
        \item The 
        group $H$ is conjugated in $\PSL(2,\RR)$ to 
        $$\operatorname{Aff}(\RR)= \left \{ \left(  \begin{array}{cc}
                              \lambda & t \\
                              0 & \lambda^{-1}
                             \end{array}
\right) \ | \ \lambda \in \RR^*, t \in \RR \right \}.$$
\item There exists $\lambda \in \RR_+^* \setminus \{1\}$, such that he group $H$ is conjugated in $\PSL(2,\RR)$ to 
$$\ZZ \ltimes_{\lambda} \RR= \left \{ \left(  \begin{array}{cc}
                              \lambda^{\frac{m}{2}} & t \\
                              0 & \lambda^{-\frac{m}{2}}
                             \end{array}
\right) \ | \ m \in \ZZ, t \in \RR \right \}.$$
       \end{enumerate}

       Observe that those four cases are mutually exclusive.
       
 Cases $(i)$ and $(ii)$ in the list above lead to respectively the first, and the second case of Theorem \ref{thm.main}.
  Theorem \ref{thm.main} will thus be proved if we show that cases $(iii)$ and $(iv)$ actually do not occur. This is basically 
  known for case $(iii)$. Indeed, it was 
  shown in \cite{adams.stuck2} and \cite[Th. 1.1]{zeghibidentity} that if the group $\operatorname{Aff}(\RR)$ acts isometrically (and faithfully)
   on a compact Lorentz manifold, then it yields an isometric action of a finite cover of $\PSL(2,\RR)$.  Hence
    if $H=\operatorname{Aff}(\RR)$ in the list above, it actually implies $H=\PSL(2,\RR)$.  Our last task is to extend this result to
     the smaller group $\ZZ \ltimes_{\lambda} \RR$, and this is the content of our last statement:
 
 \begin{theoreme}[Compare \cite{adams.stuck2}, \cite{zeghibidentity}]
  \label{thm.sl2}
  Let $(M,g)$ be a compact Lorentz manifold. Assume that $\Iso(M,g)$ contains a  closed subgroup $G$ 
    which is  a compact extension of $\ZZ \ltimes_{\lambda} \RR$, $\lambda \in \RR_+^* \setminus \{ 1 \}$. Then $\Iso(M,g)$ contains a finite cover of $\PSL(2,\RR)$. 
 \end{theoreme}

 \begin{proof}
  
 We consider the following subgroup of $\PSL(2,\RR)$:
 $$H=\left \{ \left(  \begin{array}{cc}
                              \lambda^{\frac{m}{2}} & t \\
                              0 & \lambda^{-\frac{m}{2}}
                             \end{array}
\right) \ | \ m \in \ZZ, t \in \RR \right \}.$$

We will call $a:=\left(  \begin{array}{cc}
                              \lambda^{\frac{1}{2}} & 0 \\
                              0 & \lambda^{-\frac{1}{2}}
                             \end{array}
\right) $, $\{ u^t  \}_{t \in \RR}:=\left(  \begin{array}{cc}
                              1 & t \\
                              0 & 1
                             \end{array}
\right)$, and $F:= \left(  \begin{array}{cc}
                              0 & 1 \\
                              0 & 0
                             \end{array}
\right)$.  Observe that $aFa^{-1}=\lambda F$.
Our assumption is the existence of an epimorphism between Lie groups $\rho: G \to H$.

Let us consider $\varphi \in G$ such that $\rho(\varphi)=a$. Let us call $U \subset G$ the inverse image $\rho^{-1}(\{ u^t \}_{t \in \RR})$.
 The group $U$ is a closed Lie subgroup of $G$, and its Lie algebra is a sum $\lieu= \RR \oplus \liek$, where $\liek$
  integrates into a compact Lie subgroup of $G$. The $\RR$-factor in this decomposition   is mapped onto $\RR F$ by $\rho_*$.
   The algebra $\lieu$ is normalized by $\varphi$, and with respect to the splitting $\lieu=\RR \oplus \liek$, the action reads like:
  $$ \ad(\varphi)=\left(  \begin{array}{cc}
                               \lambda & 0 \\
                              B & C
                             \end{array}
\right).$$
   We infer the existence of $Y \in \lieu$ satisfying $\Ad(\varphi)(Y)=\lambda Y$, and such that $\rho_*(Y)=F$.
   
  In what follows, we will see $Y$ as a Killing vector field on $M$, and denote by $\{Y^t\}_{t \in \RR}$ the $1$-parameter group it generates.
   We observe that $\{Y^t\}_{t \in \RR}$ is closed and noncompact in $G$. It is clearly noncompact because it is mapped onto 
   $\{ u^t \}_{t \in \RR}$ by $\rho$. Would it not be  closed, its closure in $G$ would be a torus, which would be mapped onto 
   $\{u^t \}_{t \in \RR}$
   by $\rho$: Contradiction. From all this discussion, one infers easily that the group $<\varphi, \{ Y^t \}>$, generated by $\varphi$
    and the $1$-parameter group $ \{ Y^t \}$, is closed in 
    $G$, hence in $\Iso(M,g)$. This group is clearly a compact extension of $H$. Thus, in the following we will assume
     $G=<\varphi, \{ Y^t \}>$. 
   Moreover, we will also assume, without loss of generality, that $0 < \lambda <1$.
   
   The begining of the proof follows \cite{adams.stuck2}
      and \cite{zeghibidentity}. The relation $\varphi_*(Y)=\lambda Y$ implies that $Y$ is a lightlike vector field, and it is a 
      classical fact that nonzero lightlike Lorentzian Killing fields are nowhere vanishing. Moreover, one sees that $\log \lambda$
       is a negative Lyapunov exponent for $\varphi$. One infers that there is a measurable Oseledec splitting for $\varphi$ of the form:
       $$ TM=E^+ \oplus E^o \oplus E^-.$$
       The bundles $E^+$, $E^o$ and $E^-$ are respectively associated to Lyapunov exponents $- \log \lambda$, $0$ and $\log \lambda$.
        The bundles $E^+$ and $E^-$ are $1$-dimensional and lightlike. Moreover $E^-=\RR Y$, and 
        there exists a unique measurable vector field $Z$ such that $g(Z,Y)=1$ and $E^+=\RR Z$.
        One also checks $E^o=(E^- \oplus E^+)^{\perp}$. It is shown in \cite[Lemma 5.1]{adams.stuck2} that 
         the  Oseledec splitting extends to an everywhere defined, continuous splitting. Precisely, $Z$ extends 
         to a continuous vector field, and $\varphi_*Z=\lambda^{-1}Z$ everywhere. It follows from Zeghib's theorem \ref{thm.zeghib}
          that $Z$ is actually Lipschitz. By Rademacher's theorem there exists $\Omega \subset M$ a subset of full measure, on which
           $Z$ is differentiable. Hence $T:=[Y,Z]$ makes sense on $\Omega$, and defines there a measurable vector field. Moreover,
            from the relation $D_x \varphi Z(x)= \lambda^{-1}Z(\varphi(x))$, available for every $x \in M$, we see that 
             $\Omega$ is $\varphi$-invariant and $\varphi_*T=T$.
       For every $s \in \RR^*$, and every $x \in M$, we define $\calf_x^s=\operatorname{span}\{ Z(x), (Y^s)_*Z(x), Y(x) \}$.
       
   \begin{lemme}\cite[Fact 3.4]{zeghibidentity}
    \label{lem.dimens3}
    For every $x \in M$, the space $\calf_x^s$ is $3$-dimensional, Lorentzian, and does not depend on $s \in \RR^*$.
   \end{lemme}
  
  \begin{proof}
   Let us fix $s \in \RR^*$. We first observe that for every $x \in M$, $\calf_x^s$ is $3$-dimensional. If this is 
   not the case, there
   is a closed subset $F$ where $(Y^s)_*Z$ belongs to $\operatorname{span}\{ Z,Y\}$, hence is colinear to $Z$. The relation
    $<Z,Y>=1=<(Y^s)_*Z,(Y^s)_*Y>$ then shows that for every $y \in F$, we have $(Y^s)_*Z(y)=Z(y)$.  Since $(Y^s)_*Y=Y$, we get
     that in the splitting $TM=E^+ \oplus E^o \oplus E^-$ above $F$, the differentials $D Y^{ms}$ have the form
     $$ D Y^{ms} = \left(  \begin{array}{ccc} 1 & 0 & 0 \\ 0 & K_m & 0\\ 0 & 0 & 1 \end{array}\right),$$
     where $K_m$ is a compact subset of $\OO(n-1)$.  We would get that $\{ Y^{ms} \}_{m \in \ZZ}$ has compact closure in $\Iso(M,g)$,
      but we already checked at the begining of the proof, that this is not the case.
    We conclude that  $\calf_x^s$ is $3$-dimensional for every $x \in M$, and this space is Lorentzian because
    it contains two linearly independent lightlike directions $Z(x)$ and $Y(x)$. Actually, the previous arguments show that 
     $(Y^s)_*Z(x), (Y^{s'})_*Z(x)$ and $Y(x)$ are linearly independent for every $s \not = s'$ in $\RR^*$, and $x \in M$.
     
    The vector field $T=[Z,Y]$ is well-defined, and measurable on the  subset of full measure $\Omega$.  
    We next show that for  $x \in M$, the space $\calf_x^s$ does not depend on $s \in \RR^*$. To check this, one considers, for 
     a given $s \in \RR^*$, 
     $V=(Y^s)_*Z-Z+sT$. The vector field $V$ is defined on $\Omega$. For $x \in \Omega$ the definition of the Lie bracket yields
      \begin{equation}
      \label{eq.crochet}
      (Y^t)_*Z(x)=Z(x)-t[Y,Z](x)+t \epsilon(t),
      \end{equation}
      with $\lim_{t \to 0} \epsilon(t)=0$. 
      One gets for $m \in \ZZ$, and $x \in \Omega$, $(\varphi^m)_*V(x)=\lambda^{-m}((Y^{\lambda^m s})_*Z(x)-Z(x))+sT(x).$
       By (\ref{eq.crochet}), we obtain 
       $\lim_{m \to + \infty} (\varphi^m)_*V(x)=0,$
        or in other words
       \begin{equation}
        \label{eq.to.zero}
        \lim_{m \to + \infty}D_x\varphi^{m}V(\varphi^{-m}x)=0.
       \end{equation}

       Let us denote by $\mu$ the measure defined by our Lorentzian metric $g$, renormalized to ensure $\mu(M)=1$.
     For every $k \in \NN^*$, Lusin's theorem yields a compact set $K_k \in M$ of measure at least $ 1-\frac{1}{k}$, such that
      $T$ is continuous on $\Omega_k=\Omega \cap K_k$. Poincar\'e recurrence implies that there exists $E_k \subset \Omega_k$
       a conull set in $\Omega_k$ such that if $x \in E_k$, one can find a sequence $(m_i)$ going to infinity in $\NN$, with 
       $\varphi^{-m_i}(x)$ belonging to $E_k$ for every $i$, and $\lim_{i \to + \infty}\varphi^{-m_i}(x)=x$.
       Because $V$ is continuous on $\Omega_k$, the sequence $V(\varphi^{-m_i}x)$ tends to $V(x)=(Y^s)_*Z(x)-Z(x)+sT(x)$. By (\ref{eq.to.zero}) and Osseledec splitting properties, we get 
       that $V(x)$ is colinear to $Y(x)$. It follows that $\calf_x^s=\operatorname{Span}\{ Z(x), T(x), Y(x) \}$
        for every $x \in E_k$, and every $s \not = 0$. In particular, if $s,s'$ are in $\RR^*$, then $\calf_x^s=\calf_x^{s'}$
         for every $x \in \bigcup_{k \in \NN^*}E_k$. Since $\bigcup_{k \in \NN^*}E_k$ has full measure, hence is dense  in $M$, and because the 
         distributions  $\calf^s$ and $\calf^{s'}$ are Lipschitz (each one is spanned by $3$ Lipschitz vector fields), 
         we conclude that  $\calf^s=\calf^{s'}$ everywhere on $M$.
     
  \end{proof}

   In the sequel, we will write $\calf$ instead of   $\calf^s$, since there is no dependence in $s$. This is a $3$-dimensional distribution, which is 
   Lipschitz and Lorentzian. Lemma \ref{lem.dimens3} shows that it is invariant by $\varphi$ and $\{Y^t\}_{t \in \RR}$, hence $G$-invariant.

   We denote by  $\calf^{\perp}$   the distribution orthogonal to $\calf$. This distribution  is Riemannian. An important remark is
   that $\calf^{\perp}$ is tangent to a Riemannian, totally geodesic, transversally Lipschitz foliation. To see this, we observe
    that the distribution $Y^{\perp}$ is the asymptotically stable distribution of $\{ \varphi^m \}_{m \in \NN}$ 
    (see Section \ref{sec.semi-continuity}
     for the definition). Zeghib's theorem \ref{thm.zeghib} ensures that $Y^{\perp}$ is everywhere  tangent to a codimension one, totally geodesic 
     lightlike foliation $\calf_1$ (which is transversally Lipschitz). For the same reasons, given $s \in \RR^*$, the distributions $Z^{\perp}$
      and $((Y^s)_*Z)^{\perp}$ are tangent to codimension one, totally geodesic 
     lightlike foliation $\calf_2$ and $\calf_3$. Thus $\calf^{\perp}$ is tangent to $\calf_1 \cap \calf_2 \cap \calf_3 $, which is Riemannian and totally geodesic.
      At this stage, we know that the leaves of $\calf^{\perp}$ are smooth, but the foliation is only transversally Lipschitz.
      
      We are going to show that $\calf$ is integrable as well. Since this distribution is only Lipschitz, we will have to use
       a Frobenius-type theorem for Lipschitz distributions. Such a result was proved in \cite{rampazzo}. Before quoting it, we
        recall the following definition (see \cite[Def. 4.7]{rampazzo}). Given a Lipschitz disctribution $\cald$, let us consider a Lipschitz local frame field 
        $(X_1, \ldots, X_k)$ of $\cald$.  Each Lipschitz field $X_i$ is differentiable on some set $\Omega_i$. One says that $\cald$
         is {\it involutive almost everywhere} when for each $(i,j) \in \{ 1,\ldots,k  \}^2$, $[X_i,X_j]$ belongs to $\cald$
          almost everywhere on $\Omega_i \cap \Omega_j$.

\begin{theoreme}\cite[Th. 4.11]{rampazzo}
 \label{thm.frobenius}
 Any Lipschitz distribution which is involutive almost everywhere, is everywhere tangent to a transversally Lipschitz foliation, 
  with $C^{1,1}$ leaves.
\end{theoreme}
    
 To prove that $\calf$ is integrable we are thus going to show:
 \begin{lemme}
 \label{lem.involutive}
  The distribution $\calf$ is involutive almost everywhere. Moreover, it is of class $C^1$.
 \end{lemme}

 \begin{proof}
 To show that the distribution $\calf$ is involutive almost everywhere, we  consider  $U \subset M$ an open subset. 
  If $U$ is small enough, the distribution $\calf_{|U}$ is spanned by $Y$, $Z$, and a third Lipschitz 
 vector field $X$ satisfying $g(X,X)=1$ and  $g(X,Y)=g(X,Z)=0$. Observe that this property almost characterizes $X$, in the sense that 
 only $X$ and $-X$ satisfy those relations.
 We denote by $U'$ the subset of $U$ where $X,Y$ and $Z$ are differentiable. This is a subset
 of full measure in $U$, and is contained in $\Omega$. 
 
 If $x \in U'$, then $[Y,Z]=T$ and we saw that $Z(x), T(x), Y(x)$ span $\calf_x$.   Hence $[Y,Z](x) \in \calf_x$ if $x \in U'$.  
   
   Let us now show that $[X,Y] \in \calf$ almost everywhere on $U'$. 
    The vector field $[X,Y]$ is measurable on $U'$, hence Lusin's theorem yields for every $k > >1$ a compact subset $K_k \subset U$
     such that $\mu(K_k) \geq \mu(U)- \frac{1}{k}$, and $[X,Y]$ is continuous on $U' \cap K_k$. Poincar\'e recurrence theorem yields 
      $E_k^+ \subset K_k$ a subset of full measure in $K_k$ such that for every $y \in E_k^+$, there exists a sequence $(m_i)$
       satisfying $\varphi^{m_i}(y) \in K_k$ for all $i$, and $\lim_{i \to + \infty} \varphi^{m_i}(y)=y$.
     Let us now consider $x \in E_k^+ \cap U'$. Let $(m_i)$ be a sequence as above, witnessing that $x \in E_k^+$.
       The vector fields $(\varphi^{m_i})_*X$ and  $(\varphi^{m_i})_*Y$ are defined in a small neighborhood of $\varphi^{m_i}(x)$
        contained in $U$. Here they satisfy $(\varphi^{m_i})_*X=\epsilon_{m_i}X$, where $\epsilon_{m_i}=\pm 1$, and $(\varphi^{m_i})_*Y= \lambda^{m_i}Y$, what proves that
         $\varphi^{m_i}(x) \in U'$.  Moreover, for every $i \in \NN$:
         \begin{equation}
   \label{eq.crochet}
   D_x\varphi^{m_i}([X,Y](x))=[(\varphi^{m_i})_*X,(\varphi^{m_i})_*Y](\varphi^{m_i}(x)),
  \end{equation}
         
         Equation (\ref{eq.crochet}) reads:
         $$ D_x\varphi^{m_i}([X,Y](x))=\epsilon_{m_i} \lambda^{m_i}[X,Y](\varphi^{m_i}(x)).$$
         Since $x \in E_k^+ \cap U'$, we have that $[X,Y](\varphi^{m_i}(x))$ tends to $[X,Y](x)$ as $i \to + \infty$.
 We conclude that $ D_x\varphi^{m_i}([X,Y](x))$ tends to $0$, so that $[X,Y](x)$ is colinear to $Y(x)$.
 We finally get that $[X,Y](x) \in \calf_x$ for every $x \in \bigcup_{k \in \NN}(E_k\cap U')$, hence $[X,Y] \in \calf$ almost everywhere on $U$.

 We proceed in the same way to prove that $[X,Z] \in \calf$ almost everywhere on $U$, what yields involutivity almost 
 everywhere of the distribution $\calf$.
 We conclude, applying Theorem 
    \ref{thm.frobenius}, that $\calf$ is tangent to a foliation, whose leaves are of class $C^{1,1}$.
    
  In particular, $\calf$ is (tautologically) $C^1$ in the direction of its leaves. 
  Recall that $\calf^{\perp}$ is integrable as well, with totally geodesic, hence $C^1$, leaves.
   It follows that $\calf$ is  $C^1$ in the direction of the 
   leaves of $\calf^{\perp}$. 
    Since $TM=\calf \oplus \calf^{\perp}$, we conclude that $\calf$ is $C^1$. Of course, the same is true for $\calf^{\perp}$.
 \end{proof}
 
 \begin{lemme}
  \label{lem.umbilic}
  The leaves of $\calf$ are totally umbilic, and have constant sectional curvature.
 \end{lemme}

 \begin{proof}
  For every $x \in M$, we call $\lies_x^{\calf}$ the Lie algebra of local Killing fields $X$ around $x$, satisfying $X(x)=0$,
   and such that the $1$-parameter group $\{D_x X^t\}_{t \in \RR}$ preserves $\calf_x^{\perp}$ and acts trivially on it. Observe that if 
   $X \in \lies_x^{\calf}$, then $D_xX^t$ preserves the splitting $T_xM=\calf_x \oplus \calf_x^{\perp}$. One expects that generally, 
   $\lies_x^{\calf}= \{ 0 \}$, but we claim that this is not the case. To check this, let us fix a bounded orthonormal  frame field
    $(X_1, \ldots,X_n)$ on $M$, such that $X_1,X_2,X_3$ (resp. $X_4, \ldots, X_n$) span $\calf$ (resp. $\calf^{\perp}$). This
     yields a bounded section $\sigma : M \to \hm$, defining a coarse embedding $\dd_x: \Iso(M,g) \to \OO(1,d)$ (see Section 
     \ref{sec.derivative.cocycle}). Actually, because $G$ preserves the splitting $\calf \oplus \calf^{\perp}$, the restriction of 
      $\dd_x$ to $G$ takes values in a subgroup $\OO(1,2) \times \OO(n-1) \subset \OO(1,n)$. Projecting to the first factor, one gets
       for every $x  \in M$ a coarse embedding $\dd_x': G \to \OO(1,2)$. Following the notations 
       of Sections \ref{sec.limit.coarse}
        and \ref{sec.infinite.semisimple}, we define ${\mathcal G}_x:=\dd_x(G)$, and denote by  
        $\Lambda_{\mathcal G}(x) \subset \partial \HH^2$ the limit set of ${\mathcal G}_x$. We already observed that for every $s \in \RR^*$, 
         $Y(x)$, $Z(x)$ and $(Y^s)_*Z(x)$ are asymptotically stable directions associated to the 
         sequences $(\varphi^m)_{m \in \NN}$,
         $(\varphi^{-m})_{m \in \NN}$  and $(Y^s\varphi^{-m}Y^{-s})_{m \in \NN}$. The interpretation of the limit set as 
         asymptotically stable lightlike directions (see Lemma \ref{lem.identification.limitset}) shows that $\Lambda_{\mathcal G}(x)$
          is infinite for every $x \in M$ and 
         $d_{\Lambda_{\mathcal G}}(x)=3$ for every $x \in M$. Let us now choose $x$ in the integrability locus $\mint$, and consider the generalized curvature 
          map $\kg$
         (see Sections \ref{sec.generalized} and  \ref{sec.integrability.locus}). The vector $\kg(\sigma(x))$ is stable under
          ${\mathcal G}_x$. Proposition \ref{prop.stability} then ensures that the stabilizer of $\kg(\sigma(x))$ inside
           $\OO(1,2) \times \OO(n-1)$ contains the factor $\SO^o(1,2)$. Corollary \ref{coro.stabilizer} then shows that
            $\lies_x^{\calf}$ contains a subalgebra isomorphic to $\oo(1,2)$, hence $\lies_x^{\calf}=\oo(1,2)$.
             Indeed the isotropy representation being faithfull, $\lies_x^{\calf}$ is at most $3$-dimensional.
            
       Since we showed that the distribution $\calf$ is of class $C^{1}$, it makes sense to consider, for every $x \in M$, the 
        second fundamental form $II_x$ of the leaf $F(x)$. Every $Z \in \lies_x^{\calf}$ defines a $1$-parameter group 
         $\{ D_xZ^t\}_{t \in \RR} \subset \OO(T_xM)$, which preserves the splitting $\calf_x \oplus \calf_x^{\perp}$
          and preserves $II_x$. When $x \in \mint$, the irreducibility of the action of $\lies_x^{\calf}$ on $\calf_x$
           forces, as in Lemma \ref{lem.foliations}, 
           the equality $II_x( \ , \ )=g_x( \ , \ )\nu_x$,  for some vector $\nu_x \in \calf_x^{\perp}$. Because $\mint$ is dense
            in $M$, such an equality must hold everywhere.
            This shows that the leaves of $\calf$ are totally umbilic.  
   
\end{proof}

\begin{lemme}
 \label{lem.fin}
 The distributions $\calf$ and $\calf^{\perp}$ are $C^{\infty}$. The leaves of $\calf$
  have constant sectional curvature. The universal cover $(\tm,\tilde{g})$ is isometric to a warped product
   $N \times_w \widetilde{\AdS}^{1,2}$ or $N \times_w \RR^{1,2}$, where $N$ is a $1$-connected complete Riemannian manifold.
\end{lemme}

\begin{proof}
The key point is to show the smoothness of $\calf$. For that, we are going to show that for every $x \in M$, the leaf 
 $F(x)$ of $\calf$ containing $x$ is a $C^{\infty}$ (injectively) immersed submanifold of $M$. It will show that $\calf$
  is $C^{\infty}$ of its leaves. But the leaves of $\calf^{\perp}$ are totally geodesic, hence $C^{\infty}$. We  will conclude 
   that $\calf$ is also $C^{\infty}$ in the direction of the leaves of $\calf^{\perp}$, yielding smoothness of $\calf$ on $M$.
   
 Let us consider $F$ a leaf of $\calf$, and let $x \in F$. Let us first remark that $Z^t$-orbits are lightlike geodesics 
 (the parametrization might not be affine). Indeed, $Z^{\perp}$ is at every point the asymptotically stable distribution
 of $\{ \varphi^{-m} \}_{m \in \NN}$, and Theorem \ref{thm.zeghib} ensures that $Z^{\perp}$ is tangent to a totally geodesic, 
 lightlike foliation.
 In particular, the $Z^t$-orbit of $x$, $t \in (-\epsilon, \epsilon)$
 is a piece of lightlike geodesic contained in $F$. Let $y=Z^{-\epsilon/2}.x$, and assume $\epsilon <<1$. 
  Let us choose $U \subset T_yM$ a small neighborhood of $0_y$ such that $\exp_y$
   is injective on $U$, and $\exp_y(U)$ contains $x$.  A second important remark is that because  $F$ is totally umbilic, any lightlike geodesic of $M$ which is 
  somewhere tangent to $F$ must be contained in $F$.  Thus, if ${\calc}_y \subset T_yM$ denotes the lightcone of $g_y$, then 
  $\exp_y(U \cap \calf_y\cap \calc_y)$ is included in $F$, and there exists $x' \in U \cap \calf_y\cap \calc_y$ with $\exp_y(x')=x$.
   Now choose $V \subset U$ a small open subset containing $x'$, and call $\Sigma:=\exp_y(V \cap \calf_y\cap \calc_y)$.
    This is a piece of $C^{\infty}$ lightlike surface in $F$, 
    containing $x$, and that we call $\Sigma$. Observe that $Z(x) \in T_x \Sigma$. Because $Y(x)$ is 
    transverse to $Z(x)$ and lightlike,
       then $Y(x)$ must be transverse to $T_x \Sigma$. As a consequence, for 
     $\delta>0$ very small, the map $\psi: (-\delta,\delta) \times (V \cap \calf_y\cap \calc_y) \to M$
      defined by $\psi(t,z):=Y^t.\exp_y(z)$ is a $C^{\infty}$ immersion, whose image is an open neighborhood of $x$ in $F$. 
      Smoothness of $F$ follows.
      
 We now consider a leaf $F$ of $\calf$. Considering a small piece of it, it is an embedded, smooth,  submanifold $F'$ of $M$. The 
 restriction of 
  $g$ to $F'$ is called $\overline{g}$, and its sectional curvature denoted by $\overline{K}$. We already observed that 
   for $x \in \mint$, the Lie algebra $\lies_x^{\calf}$ is isomorphic to $\oo(1,2)$. The same computations as those made at the end of Lemma 
   \ref{lem.foliations} show that 
   $\overline{K}(x)$ is constant (on the Grassmannian of non-degenerate $2$-planes) for every $x \in \mint$.
   Again, density of $\mint$ in $M$ show that this is true for every $x \in M$. Schur's lemma then say that all leaves $F$ have 
   constant curvature.
  
 We then follow the same arguments as at the begining of Section \ref{sec.warped} (before Lemma \ref{lem.produit-complet}), and get
  that the universal cover $\tm$ is a product $N \times X$, where $X$ is a $3$-dimensional Lorentz manifold of constant sectional curvature, and
    $N$ is a complete Riemannian manifold. The sets $\{n \} \times X$ (resp. $N \times \{ x \}$) project on the leaves of $\calf$
     (resp. the leaves of $\calf^{\perp}$).
      The 
   metric $\tilde{g}$ has the form $g_N \oplus w g_X$ for some function $w:N \times X \to \RR_+^*$.
  To check that we have a warped product structure, namely that $w(n,x)$ does not depend on $x$, we recall that
   given $t \not = 0$, for all $x \in M$, the directions $Z(x), \ (Y^t)_*Z(x)$ and $Y(x)$ span $\calf$. Moreover, 
    the distributions $Z^{\perp}, \ ((Y^t)_*Z)^{\perp}$ and $Y^{\perp}$ are the asymptotically stable distibutions
     of $(\varphi^{-m})_{m \in \NN}, \ (Y^t\varphi^mY^{-1})_{m \in \NN}$ and $(varphi^m)_{m \in \NN}$ hence are tangent to three
     totally geodesic, lightlike, codimension one foliations $\calf_1,\calf_2,\calf_3$, such that $\calf^{\perp}=\calf_1 \cap \calf_2 \cap \calf_3$.
      Thus leaves of $\calf^{\perp}$ are included in leaves of $\calf_i$ for $i=1,2,3$. We can then apply  
      \cite[Prop. 2.4]{zeghibisom2} and conclude that $(\tm, \tilde{g})$ is a warped product $N \times_w X$.
       The manifold $(M,g)$ is obtained as a quotient of $N \times_w X$ by a discrete subgroup 
       $\Gamma \subset \Iso(N) \times \operatorname{Homot}(X)$ (see point $(1)$ of Lemma \ref{lem.isometries} and Remark \ref{rem.general}).
      Theorem \ref{thm.complete}
  ensures that $X$ is actually isometric to $\widetilde{\AdS}^{1,2}$ or $\RR^{1,2}$.

\end{proof}
 We are now ready to conclude the proof of Theorem \ref{thm.sl2}.  If in the previous Lemma, the factor 
  $X$ is isometric to $\RR^{1,2}$, Proposition 
  \ref{prop.minkowski} ensures that $\Iso(M,g)$ is virtually a compact extension of a discrete subgroup $\Lambda \subset \PO(1,2)$.
   This is in contradiction with our standing assumption that $H$ is the group $\ZZ \ltimes_{\lambda} \RR$.
   
   It follows that $X$ is isometric to $\widetilde{\AdS}^{1,2}$. Then Proposition \ref{prop.ads} says that $\Iso(M,g)$ is virtually an extension of 
   $\PSL(2,\RR)$ by a compact Lie group, which is precisely what we wanted to show.

 \end{proof}

 
 \begin{remarque}
  The assumption that $G$ is closed in Theorem \ref{thm.sl2} is crucial. Indeed, there are flat Lorentz tori $\TT^2$, with an 
  isometric action of $A=\left( \begin{array}{cc} 2 & 1\\
                                 1 & 1\\
                                \end{array}
\right).$ The isometry group of such a $\TT^2$ clearly contains a (non closed) subgroup isomorphic to 
$\ZZ \ltimes_{\lambda} \RR$, but no subgroup locally isomorphic to $\PSL(2,\RR)$.
 \end{remarque}

 \subsection{Isometric actions of lattices. Proof of Corollary \ref{thm.lattices}}
 \label{sec.action.lattices}
 
 We finish this section with the proof of Corollary \ref{thm.lattices}. Our assumption is that 
  $(M^{n+1},g)$ is a $(n+1)$-dimensional, compact, Lorentz manifold, and that $\Iso(M,g)$ contains a discrete subgroup
   $\Lambda$ isomorphic to a lattice in a noncompact simple Lie group $G$.
   
 By Theorem  \ref{thm.main}, there exists a finite index subgroup $\Lambda' \subset \Lambda$, and
  a morphism $\rho: \Lambda' \to \PO(1,d)$ with discrete image and finite kernel. Since $\Lambda'$ is a lattice too, we will assume 
  $\Lambda'=\Lambda$ in what follows.
  
  If $G$ is not locally isomorphic to $\PO(1,m)$ or $\operatorname{PU}(1,m)$, then $\Lambda$ has Kazhdan's property (T).
   The morphism $\rho: \Lambda \to \PO(1,d)$ provides an action of $\Lambda$ on $\HH^d$, and it is known that 
    such an action should have a fixed point, \cite[prop 2.6.5]{bekka}, namely $\rho(\Lambda)$ should be relatively compact.
     But $\rho(\Lambda)$ is infinite discrete since $\rho$ is proper: Contradiction.
     
 Assume now that $G$ is  isomorphic to $\SU(1,m)$, for $m \geq 2$. We first rule out the case when $\Lambda$ is not uniform. In this case,
  the {\it thick-thin} decomposition ensures that $\Lambda$ contains a subgroup virtually isomorphic
   to a lattice in Heisenberg group $\operatorname{Heis}(2m-1)$. The image by $\rho$ of this subgroup would yield a discrete,
    nilpotent subgroup of $\PO(1,d)$. Such subgroups are virtually abelian, hence can not be virtually isomorphic to lattices in
    $\operatorname{Heis}(2m-1)$. 
    
   Now, if $\Lambda$ is a uniform lattice in $\SU(1,m)$, we may assume that it is torsion-free (again by replacing $\Lambda$ by 
   a finite index subgroup).  We can then conclude by an argument involving harmonic maps. I thank Pierre Py for 
   pointing this out to me.  First, observe that the Zariski closure of $\rho(\Lambda)$ in $\PO(1,d)$ is reductive, because if not 
   $\rho(\Lambda)$, which is a discrete group,  would be virtually abelian, contradicting that it is virtually isomorphic to $\Lambda$.
    Then it follows from \cite[Corollary 3.7]{carlson} that $\rho(\Lambda)$ is virtually contained in a surface group. The reader may
    also 
    look at  \cite[Section 3]{py} regarding this point.
    This forces $\rho(\Lambda)$ to have asymptotic dimension $\leq 2$. This is a contradiction since 
     the asymptotic dimension of $\Lambda$ is $2m$, $m \geq 2$, and $\rho: \Lambda \to \rho(\Lambda)$ is proper,
     hence a coarse embedding (see Lemma \ref{lem.asdim}).
     
 We have thus proved the first part of the corollary, namely $G$ is locally isomorphic to $\PO(1,m)$. By Proposition \ref{prop.reseaux},
  we must have  $m \leq n$. 
  
  It remains to understand what happens when equality $m=n$ holds. We then go back to Theorem \ref{thm.reduction}. There exists a compact 
   Lorentz submanifold $\Sigma \subset M$ which is preserved by a finite index subgroup $\Lambda' \subset \Lambda$. Moreover $\Sigma$
    is locally homogeneous, and its isotropy algebra contains a subalgebra $\oo(1,k)$ (with $n \geq k \geq 2$ maximal for this property).
     Again, we will assume $\Lambda'=\Lambda$ in the following.
     The frame bundle of $\Sigma$ admits a reduction to $\OO(1,k) \times L$, for some compact group $L$.
      Now Corollary \ref{coro.subbundle} says that $\Lambda$ coarsely embeds into $\OO(1,k) \times L$, hence into $\OO(1,k)$.
       By the same arguments as in the proof of proposition \ref{prop.reseaux}, we have $k \geq n$, hence $k=n$. It follows that  
        $\Sigma=M$, hence $M$ is locally homogeneous, with  isotropy algebra $\oo(1,n)$. As a consequence,  $M$ has 
        constant sectional  curvature.
    Since $\Lambda$ has exponential growth,  Proposition \ref{prop.homogeneous.partial} applies and 
    says that $M$ is either $3$-dimensional and of curvature $-1$, or flat
     (of any dimension $\geq 3$). In the first case, the study made in Section \ref{sec.factor.ads} yields the structure of the
     manifold $M$. It is a quotient $\PSL(2,\RR)^{(m)}/ \Gamma$, for some uniform lattice $\Gamma$.
     
     In the flat case, $M$ is the quotient of $\RR^{1,n}$ by a discrete subgroup $\Gamma$ of $\OO(1,n) \ltimes \RR^{n+1}$. It was shown 
      in Section \ref{sec.factor.flat} that $\RR^{n+1}$ splits as a sum $E \oplus F$, with $E$ a lorentzian subspace of dimension 
       $d+1$, with $\Gamma$ acting by $\pm Id$ on $E$. Lemma \ref{lem.properness} says that there is a proper homomorphism from 
       $\Lambda$
        to $\PO(1,d)$. Again, the same arguments  as in the proof of Proposition \ref{prop.reseaux} show that $d<n$ is impossible, 
        hence $d=n$. 
        As a consequence, the linear part of $\Gamma$ is contained in $\pm Id$. This shows that $M$ is a Lorentzian flat torus, or 
         a two-fold cover of such a torus. Corollary \ref{thm.lattices} is proved.

\ \\

{\bf Aknowledgment:} I warmly thank Pierre Py, Romain Tessera and Abdelghani Zeghib for enlightning conversations.

\end{document}